\documentclass{amsart}

\usepackage{geometry}
\usepackage{amsmath}
\usepackage{amssymb}
\usepackage{amsthm}
\usepackage{color}
\usepackage[shortlabels]{enumitem}

\usepackage{hyperref}

\newtheorem{cor}{Corollary}
\newtheorem{lemma}{Lemma}
\newtheorem{thm}{Theorem}
\newtheorem{defn}{Definition}

\numberwithin{equation}{section}

\usepackage[dvipsnames]{xcolor}

\begin{document}

\title[Additive decompositions of sums of squares]{The Large Sieve Inequality   and \\ Additive decompositions of sums of squares}
\date{}
\author{Christian Elsholtz and Julia Stadlmann}

\begin{abstract} Building on the large sieve,  we prove a new sieve estimate for sets which are missing various residue classes modulo prime squares. For certain interesting residue class configurations, our 
bound is a significant improvement over the corresponding 
Johnsen-Selberg  sieve inequality. 

Our   application of this sieve estimate is 
motivated by Ostmann's problem. We   study additive decompositions of sums of squares and show that any non-trivial sumset decomposition must consist of two sets of roughly equal size: If $\mathcal{A}_1, \mathcal{A}_2 \subseteq \mathbb{Z}$ with $\mathcal{A}_1+\mathcal{A}_2=\{x^2+y^2: x,y \in \mathbb{Z}\}$ and $|\mathcal{A}_1|, |\mathcal{A}_2| \geq 2$, then 
$\sqrt{x}/(\log x)^{9/2} \ll |\mathcal{A}_i\cap [1,x]| \ll \sqrt{x} (\log x)^4$ for $i \in \{1,2\}.$ 
\end{abstract}

\maketitle

\section{Introduction}

In this paper, we prove new sieve bounds for sets which have various residue class restrictions modulo prime squares. Our bounds are particularly strong when there are residue classes modulo $p$ for which many corresponding residue classes modulo $p^2$ are missing. The study of these residue class configurations was inspired by Ostmann's problem, as explained below: \medskip

It is impossible to write the set of primes as a sumset $\mathcal{B}_1+\mathcal{B}_2$ with $|\mathcal{B}_1|, |\mathcal{B}_2|\geq 2$. 
Ostmann's problem~\cite{Ostmann:1968:AZT} asks if there exists an  \textit{asymptotic} additive decomposition of the primes: This would be a pair  of sets $\mathcal{B}_1, \mathcal{B}_2 \subseteq \mathbb{Z}$ with $|\mathcal{B}_1|, |\mathcal{B}_2|>1$ so that the sumset $\mathcal{B}_1 + \mathcal{B}_2$ differs from the set of primes by only finitely many elements. Progress on this problem was made by   Pomerance, Sárközy and Stewart~\cite{Pomerance:1988:DSI}, Hofmann and Wolke~\cite{Hofmann:1996:ADP} and  Elsholtz~\cite{Elsholtz:2001:IGP}.  Finally, Elsholtz and Harper~\cite{Elsholtz:2015:ADR}  proved that if  $\mathcal{B}_1 + \mathcal{B}_2$ differs from the set of primes by only finitely many elements, and  $|\mathcal{B}_1|, |\mathcal{B}_2| \geq 2$, then both $|\mathcal{B}_1\cap [1,x]| $   and $|\mathcal{B}_2\cap [1,x]|$ must satisfy the bounds
\begin{align} \label{thm:els-harper}
\sqrt{x}/(\log x \log \log x) \ll |\mathcal{B}_i\cap [1,x]| \ll \sqrt{x} \log \log x \quad \mbox{ for } i \in \{1,2\}.
\end{align}
If  $\mathcal{B}_1 + \mathcal{B}_2$ is an asymptotic additive decomposition of the primes and $b_1 \in \mathcal{B}_1$, then for any prime $p$, there are only finitely many $b_2 \in \mathcal{B}_2$ with $b_2 \equiv - b_1 \mbox{ mod } p$. Hence elements of $\mathcal{B}_1$ restrict the residue classes which appear in $\mathcal{B}_2$, and the other way around.   Results like~(\ref{thm:els-harper}) were  thus  obtained by combining various sieve estimates, such as Selberg's sieve, Montgomery's large sieve and Gallagher's larger sieve.    Recently, OpenAI~\cite{OpenAI:2026:OST} claimed a full resolution of Ostmann's problem, showing that an asymptotic additive decomposition of the primes cannot exist.

\medskip
For a full resolution of Ostmann's problem, sets $\mathcal{B}_1$ and $\mathcal{B}_2$ which  occupy approximately half the residue classes mod $p$ for each prime $p$ formed  the main obstacle in~\cite{Elsholtz:2015:ADR}. This is because when $S \subseteq \mathbb{N} \cap [1,x]$ covers $\frac{p+1}{2}$ residue classes mod $p$ for each large prime $p$, then Montgomery's large sieve and Gallagher's larger sieve both give the sieve bound $S=O(x^{1/2})$, and this upper bound is sharp, as shown by the choice $S= \{n^2: 0 \leq n \leq x^{1/2}\}$. It is believed that quadratics are the only example for which these sieve estimates are sharp, and that improvements are possible for other sets. This is  known as the inverse question for the large sieve. For residue classes with some additive structure, Green and Harper have made interesting progress on this topic~\cite{Green:2014:IQL}, but   the problem remains unresolved for the general case.  

\subsection{Decompositions of sums of squares}
Since sets with residue class configurations similar to those of the squares formed a key obstacle in Ostmann's problem, we were motivated to study the squares themselves. Unlike primes, the set of sums of two squares obviously can be written as a sumset in a ``non-trivial" way: $\{x^2+y^2: x,y \in \mathbb{Z}\} = \mathcal{S}+\mathcal{S}$ for $\mathcal{S} = \{n^2: n \in \mathbb{N} \cup \{0\}\}$. In this paper we therefore ask what  other, genuinely different  sumset decompositions of $\mathcal{S}+\mathcal{S}$ look like. We  find that   any non-trivial sumset decomposition $\mathcal{A}_1+\mathcal{A}_2=\mathcal{S}+\mathcal{S}$ must consist of two sets $\mathcal{A}_1$ and $\mathcal{A}_2$ of roughly equal size:
\begin{thm}\label{thm:2squares} 
Suppose $\mathcal{A}_1, \mathcal{A}_2 \subseteq \mathbb{Z}$ with $\mathcal{A}_1+\mathcal{A}_2=\{x^2+y^2: x,y \in \mathbb{Z}\}$ and $|\mathcal{A}_1|, |\mathcal{A}_2| \geq 2$. 

Then  for sufficiently large $x$, 
\begin{align}\label{inequ:2square}
\sqrt{x}/(\log x)^{9/2} \ll |\mathcal{A}_i\cap [1,x]| \ll \sqrt{x} (\log x)^4 \quad \mbox{ for } i \in \{1,2\}.
\end{align}
$($The implied constants  and the required size of $x$ depend  on   $ \mathcal{A}_1 $ and $ \mathcal{A}_2$. $)$. 
\end{thm} 
 We also consider some explicit constructions of  alternative sumset decompositions of $\mathcal{S}+\mathcal{S}$, which involve subsets or supersets of the squares. 
  
\begin{thm}\label{thm:construction2}
Let  $\mathcal{S} = \{n^2: n \in \mathbb{N} \cup \{0\}\}$. The following is true:
\begin{enumerate}[{\rm (1)}]
\item Many sumset decompositions via subsets: Let $p$ be a prime with $p \equiv 1 \bmod 4$. Let  $M_1, M_2 \subseteq \mathbb{N}$ with $M_1 \cap M_2 = \emptyset$. Let $ \mathcal{A}_1=\mathcal{S}\setminus\{p^2 m^2: m\in M_1\}$ and $\mathcal{A}_2=\mathcal{S}\setminus\{p^2 m^2: m\in M_2\}$.  Then 
\begin{align*} 
&\mathcal{A}_1 + \mathcal{A}_2= \mathcal{S}+\mathcal{S}.
 \end{align*} 
\item No sumset decompositions via supersets: Let $\mathcal{A}_1, \mathcal{A}_2 \subseteq \mathbb{Z}$ with $\mathcal{S} \subseteq \mathcal{A}_1$ and $\mathcal{S} \subseteq \mathcal{A}_2$. Suppose further  that $\mathcal{A}_1 \neq \mathcal{S}$ or $\mathcal{A}_2 \neq \mathcal{S}$. Then 
\begin{align*} 
&\mathcal{A}_1 + \mathcal{A}_2 \neq \mathcal{S}+\mathcal{S}.
 \end{align*} 
\end{enumerate}
\end{thm}

In particular, Theorem~\ref{thm:construction2} implies that for any $c \in [\frac{9}{10},1]$,  the set of sums of two squares can be written as a sumset $\mathcal{A}_1+\mathcal{A}_2$ with $|\mathcal{A}_i\cap [1,x]| \sim c\sqrt{x}$ for $i \in \{1,2\}$. However, there is certainly still a lot of room between our constructive example $|\mathcal{A}_i\cap [1,x]| \sim \frac{9}{10}\sqrt{x}$ from Theorem~\ref{thm:construction2} and our proven lower bound $|\mathcal{A}_i\cap [1,x]|  \gg \sqrt{x}/(\log x)^{9/2}$ from Theorem~\ref{thm:2squares}. 
 
 \subsection{A new sieve estimate modulo prime squares} While results (\ref{thm:els-harper})   and (\ref{inequ:2square}) look  quite similar, there is a core difference between their proofs: The main ingredients of (\ref{thm:els-harper}) are standard sieve bounds for sets which are missing some residue classes modulo primes. For sums of two squares, no residue classes modulo primes are entirely missing, and we must instead consider residue class restrictions modulo prime squares. This situation is far less well studied.  For the residue class configurations relevant to this problem,   results available in the literature are rather weak: the Johnsen-Selberg prime power sieve~\cite{Johnsen:1971:JSS} produces lower and upper bounds of about $x^{1/4}$ and $x^{3/4}$ for $|\mathcal{A}_i\cap [1,x]|$. In this paper we build on the  large sieve inequality to prove a new sieve bound for sets which are missing various residue classes modulo prime squares. Our result gives significantly better results than the Johnsen-Selberg prime power sieve for the relevant residue class configurations.  

\begin{defn}[Important factors used in Theorem~\ref{thm:largesieveimprovement}]\label{def:sievefactors}
Let $\mathcal{P} \subseteq \{p \in  [\tfrac{1}{2}M,M]: p \mbox{ is prime}\}$. Let $\Omega(p) \subset  \{1,\dots,p\}$ and write $\omega(p)=|\Omega(p)|$.  Assume $\omega(p) \not\in \{ 0, p\}$ for each $p \in \mathcal{P}$. Let $\Theta \in \mathbb{N}$. We define: 
\begin{align*}
&c^*(k,\omega,\mathcal{P})=\left( \sum_{p \in \mathcal{P}} \left(\frac{\omega(p)}{p-\omega(p)} \right) \right)^{-1} \left\{\frac{1}{3} \sum_{p \in \mathcal{P}} \left(\frac{\omega(p)}{p-\omega(p)} \right) - k \max_{p \in \mathcal{P}}  \left(\frac{\omega(p)}{p-\omega(p)} \right) \right\}, \\
&d^*(k,\omega,\mathcal{P})=\left( \sum_{p \in \mathcal{P}} \left(\frac{\omega(p)}{p-\omega(p)} \right) \right)^{-1} \left\{\frac{c^*(k,\omega,\mathcal{P})^k}{   2^{k+5} k} \sum_{p \in \mathcal{P}} \left(\frac{\omega(p)}{p-\omega(p)} \right) - k \max\left\{1,\max_{p \in \mathcal{P}}  \left(\frac{p-\omega(p)}{\omega(p)} \right) \right\} \right\},\\
&\xi(k,\omega,\mathcal{P},\Theta)=64\left(\frac{ 32 \Theta }{c^*(k,\omega,\mathcal{P})  d^*(k,\omega,\mathcal{P}) } \right)^k\max\left\{ \left(\frac{   d^*(k,\omega,\mathcal{P}) }{ 32 \Theta }  \right)^k,   \left(\frac{|\mathcal{P}|}{ M }\right)^k   ,  \left(  \frac{ \sum_{p \in \mathcal{P}} \left( \frac{\omega(p )}{p -\omega(p )}\right) }{M } \right)^k \right\}.
\end{align*}
\end{defn} 
  
\begin{thm}\label{thm:largesieveimprovement}
Let $ M \in (0,\infty)$ and $k \in \mathbb{N}$ with $M^k \leq N^{1/2}$.  Let $\Theta \in \mathbb{N}$.

Let $\mathcal{P} \subseteq \{p \in  [\tfrac{1}{2}M,M]: p \mbox{ is prime}\}$. 
For each  prime $p \in \mathcal{P}$, choose a non-empty set $\Omega(p) \subset  \{1,\dots,p\}$ and write $\omega(p)=|\Omega(p)|$.       Let $c^*(k,\omega,\mathcal{P})$, $d^*(k,\omega,\mathcal{P})$ and $\xi(k,\omega,\mathcal{P},\Theta)$ be as given in Definition~\ref{def:sievefactors} and  assume that these quantities are all  positive. 

\medskip 
Suppose  that $\mathcal{A} \subseteq \mathbb{N} \cap [1,N]$ has the following property: For every $p \in \mathcal{P}$ and $c\in \Omega(p)$, 
$$|\{ (a \mbox{ mod } p^2) \in \mathbb{Z}/p^2\mathbb{Z}: a \in \mathcal{A}, a \equiv c \mbox{ mod } p\}| \leq \Theta.$$
  Then the following upper bound is satisfied: 
\begin{align*}
&|\mathcal{A}| \leq  \xi(k,\omega,\mathcal{P},\Theta) \, 2N \left(  \mathop{\sum \dots \sum}_{\substack{ p_1, \dots, p_k \in \mathcal{P} \\ p_1> \dots > p_k }}  \prod_{i=1}^k \left( \frac{\omega(p_i)}{p_i-\omega(p_i)}\right) \right)^{-1}.
\end{align*}
\end{thm} 
  In practice, the assumptions $c^*(k,\omega,\mathcal{P}) >0$  and $d^*(k,\omega,\mathcal{P}) > 0$  are only minor technical restrictions.   The condition  $c^*(k,\omega,\mathcal{P}) >0$ fails if  there are   a few values of $p$ for which $(\frac{\omega(p)}{p-\omega(p)} )$ is much larger than average. Similarly, if $ (\frac{\omega(p)}{p-\omega(p)}  )$  is small for many $p$, and there are some very large $ (\frac{p-\omega(p)}{\omega(p)} )$, then $d^*(k,\omega,\mathcal{P}) > 0$ fails. Therefore, when applying Theorem~\ref{thm:largesieveimprovement} it may be necessary to remove from $\mathcal{P}$ some choices  of $p$ for which  $\omega(p)$ is either much larger or much smaller than average. Since these conditions make  Theorem~\ref{thm:largesieveimprovement} look  a little technical, we  record below a simplified version for the case where all $\omega(p)$ are contained in an interval $[c_2p,c_3p]$ and where $|\mathcal{P}|$ is of size at least $\frac{c_1 M }{\log M}$.

\begin{cor}\label{cor:john-selb-improvement}
Let $0<c_1<c_2<c_3<1$. Let $ M \in (0,\infty)$ and $k \in \mathbb{N}$ with $M^k \leq N^{1/2}$. Let $\Theta \in \mathbb{N}$.

 Let $\mathcal{P} \subseteq \{p \in  [\tfrac{1}{2}M,M]: p \mbox{ is prime}\}$ with $|\mathcal{P}| \geq \frac{c_1 M }{\log M}$.

For each  prime $p \in \mathcal{P}$, choose a non-empty set $\Omega(p) \subset  \{1,\dots,p\}$ and write $\omega(p)=|\Omega(p)|$.

Assume that for every $p \in \mathcal{P}$, we have $\omega(p) \in [c_2p,c_3p]$.  

Suppose   $\mathcal{A} \subseteq \mathbb{N} \cap [1,N]$  and for every $p \in \mathcal{P}$ and $c \in \Omega(p)$,   $|\{\overline{a} \in \mathbb{Z}/p^2\mathbb{Z}: a \in \mathcal{A}, a \equiv c \mbox{ mod } p\}| \leq \Theta.$
  Then 
\begin{align*}
&|\mathcal{A}| \ll_{k,\Theta,c_1,c_2,c_3} 2N \left(  \mathop{\sum \dots \sum}_{\substack{ p_1, \dots, p_k \in \mathcal{P} \\ p_1> \dots > p_k }}  \prod_{i=1}^k \left( \frac{\omega(p_i)}{p_i-\omega(p_i)}\right) \right)^{-1} \ll_{k,\Theta,c_1,c_2,c_3} \frac{N (\log M)^k}{M^k}.
\end{align*} 
$($Here $\ll_{k,\Theta,c_1,c_2,c_3}$ indicates a big O constant independent of $N$, but possibly dependent on $k$, $\Theta$ and $c_i$.$)$
\end{cor} 

 Theorem~\ref{thm:largesieveimprovement} is the key ingredient of our proof of Theorem~\ref{thm:2squares}. However, this   sieve inequality is actually considerably more general than what we  really need to prove Theorem~\ref{thm:2squares}. We discuss the   strength of this sieve estimate below.

 \subsubsection{Comparison with the Johnsen-Selberg prime power sieve} \label{ssec:largesievecomparison}

To provide context for our results, we first discuss the Johnsen-Selberg prime power sieve.  
 This sieve is a more general version of the  standard Selberg sieve, which considers not only residue classes in $\mathbb{Z}/p\mathbb{Z}$, but also uses information about the distribution of $\mathcal{A}$ in $\mathbb{Z}/p^k\mathbb{Z}$ for $k\geq 2$. It first appears in a paper of Selberg~\cite{Selberg:1977:RMF}, which improves on earlier work of Johnsen~\cite{Johnsen:1971:JSS}. (See Chapter I.4 of Tenenbaum's book~\cite{Tenenbaum:2015:APN} for a nice exposition.) If we only consider primes and prime squares, the Johnsen-Selberg sieve bound may be stated as follows: 
\begin{lemma}[Johnsen-Selberg sieve]\label{lem:johnsen} For each prime $p$, let $\Omega(p) \subseteq \mathbb{Z}/p\mathbb{Z}$ and $\Omega(p^2)  \subseteq \mathbb{Z}/p^2\mathbb{Z}$ with $$\{n \in \mathbb{Z}: (n \mbox{ \rm mod } p) \in \Omega(p)\} \cap \{n \in \mathbb{Z}: (n \mbox{ \rm mod } p^2) \in \Omega(p^2)\}  = \emptyset.$$ Let $\omega(p^k)=|\Omega(p^k)|$ for $k \in \{1,2\}$.

Suppose $\mathcal{A} \subseteq [1,N] \cap \mathbb{N}$ with $\{a \in \mathcal{A}: (a \mbox{ \rm mod } p^k) \in \Omega(p^k)\} = \emptyset$ for $k \in \{1,2\}$.  Then
\begin{align*}
|\mathcal{A}| \leq 2N \left(\sum_{d \leq N^{1/2}}^*\prod_{p \mid\mid d} \left(\frac{\omega(p)}{p-\omega(p)} \right)\prod_{p^2 \mid\mid d} \left(\frac{p \omega(p^2)}{(p^2-p\omega(p)-\omega(p^2))(p-\omega(p))} \right) \right)^{-1}, 
\end{align*}
where  $\sum^*$ denotes summation over $d$  with $p^3 \nmid d$ for all primes $p$, and where   $p^k \mid\mid d$ means that $p^k$ is the largest power of $p$ which divides $d$.
\end{lemma}

This sieve estimate was for instance used by   Cochrane and Dressler~\cite{Cochrane:1987:CSS} in their work on consecutive triples of sums of squares and in Nowak's work on arithmetic progressions of sums of squares~\cite{Nowak:2005:DMB}. Further, referencing Gallagher's paper~\cite{Gallagher:1974:SPP},  van Doorn and Tao~\cite{Doorn:2026:GSP} reproved   a version of the Johnsen-Selberg sieve  and used it in their work on   Erdős problems 1102 and 1103 (numbering according to the Erdős problem website).  

\medskip
Now  consider the case where $\mathcal{A} \subseteq [1,N] \cap \mathbb{N}$ with $|\{\overline{a} \in \mathbb{Z}/p^2\mathbb{Z}: a \in \mathcal{A}, a \equiv c \mbox{ mod } p\}|=1$ for every $p \in [\frac{1}{2}N^{1/4},N^{1/4}]$ and every $c \in \{1, \dots, \frac{p+1}{2}\}$. That is, the set $\mathcal{A}$ avoids all but one residue class mod $p^2$ for every residue class $c \in \{1, \dots, \frac{p+1}{2}\}$ mod $p$. The Johnsen-Selberg sieve   gives upper bound
\begin{align*} 
 |\mathcal{A}| \leq 2N \,\,\left(\sum_{ p \in [\frac{1}{2}N^{1/4},N^{1/4}]}   \left(1-\tfrac{2}{p^2+1} \right) \right)^{-1} \ll N^{3/4} \log(N).
\end{align*}
On the other hand, if  $\mathcal{B} \subseteq [1,N] \cap \mathbb{N}$ has the property that $\{b \in \mathcal{B}: b \equiv c \mbox{ mod } p\} = \emptyset$ for every $p \in [\frac{1}{2}N^{1/4},N^{1/4}]$ and every $c \in \{1, \dots, \frac{p+1}{2}\}$, then the ordinary Selberg sieve or, equivalently, Montgomery's large sieve are applicable. (These are   Lemma~\ref{lem:johnsen} with $\Omega(p^2)=\emptyset$ for all $p$ or Lemma~\ref{lem:largesieve.montgomery}.) We have the bound
\begin{align*} 
 |\mathcal{B}| \leq 2N \,\,\left(\sum_{ p_1,p_2 \in [\frac{1}{2}N^{1/4},N^{1/4}]}  \mu(p_1p_2)^2 \left(1+\tfrac{2}{p_1-1}\right) \left(1+\tfrac{2}{p_2-1}\right) \right)^{-1} \ll N^{1/2} \log(N)^2.
\end{align*}
In short, the Johnsen-Selberg sieve gives a much weaker bound on $\mathcal{A}$ (which avoids all but one residue class mod $p^2$ for many residue classes  $c$ mod $p$)  than on $\mathcal{B}$ (which avoids all residue classes mod $p^2$ for many residue classes  $c$ mod $p$). However, intuitively, not many elements of $\mathcal{A}$ should be in these very restricted residue classes $c$ mod $p$. We confirm this suspicion:  Theorem~\ref{thm:largesieveimprovement} gives the  improved bound  
\begin{align*} 
 |\mathcal{A}| \ll N^{1/2} \log(N)^2,
\end{align*}
which is now  of the same size as the bound on $|\mathcal{B}|$ (up to a constant factor). This improved bound  plays a key role in our treatment of asymptotic additive decompositions of sums of two squares.  More generally,   Theorem~\ref{thm:largesieveimprovement} gives much stronger bounds than the Johnsen-Selberg prime power sieve when $|\{\overline{a} \in \mathbb{Z}/p^2\mathbb{Z}: a \in \mathcal{A}, a \equiv c \mbox{ mod } p\}|$ is  small for many values of $c$.

\subsubsection{Comparison with Montgomery's large sieve}  In particular, Corollary~\ref{cor:john-selb-improvement}  tells us that if $M^k \leq N^{1/2}$ and  $\mathcal{P} \subseteq  [\tfrac{1}{2}M,M]$ with $|\mathcal{P}| \geq \frac{c_1M}{\log M}$  and $\omega(p) \in [c_2p,c_3p]$ for every $p\in \mathcal{P}$, then the residue class restriction $|\{\overline{a} \in \mathbb{Z}/p^2\mathbb{Z}: a \in \mathcal{A}, a \equiv c \mbox{ mod } p\}| \leq \Theta$ for $p \in \mathcal{P}$ and $c\in \Omega(p)$ implies  
\begin{align*}
&|\mathcal{A}| \ll_{k,\Theta,c_1,c_2,c_3} 2N \left(  \mathop{\sum \dots \sum}_{\substack{ p_1, \dots, p_k \in \mathcal{P} \\ p_1> \dots > p_k }}  \prod_{i=1}^k \left( \frac{\omega(p_i)}{p_i-\omega(p_i)}\right) \right)^{-1}.
\end{align*} 
This upper bound differs from Montgomery's large sieve only by a constant factor dependent on $k$, $\Theta$ and $c_1, c_2,c_3$. (See Lemma~\ref{lem:largesieve.montgomery} below.) This means that if we choose some residue classes $\Omega(p)$ modulo $p$, and remove most of the corresponding residue classes modulo $p^2$, we get upper bounds of nearly the same strength as if we remove $\Omega(p)$ entirely.

\medskip
 However, one drawback of our  new sieve estimate is that we are restricted to considering products of  primes $p$  which are contained in some dyadic interval $ [\tfrac{1}{2}M,M]$. The standard Selberg sieve or large sieve would allow us to use any product of primes less than $N^{1/2}$. This small defect is the reason why we obtain the bound  $\sqrt{x}/(\log x)^{9/2} \ll |\mathcal{A}_i\cap [1,x]| \ll \sqrt{x} (\log x)^4$  for non-trivial additive decomposition of sums of two squares, whereas  Elsholtz and Harper~\cite{Elsholtz:2015:ADR} proved the stronger bound $
\sqrt{x}/(\log x \log \log x) \ll |\mathcal{B}_i\cap [1,x]| \ll \sqrt{x} \log \log x $   for asymptotic additive decompositions of the primes.

\subsubsection{The large sieve inequality for square moduli} We conclude this expository section with a brief discussion of the large sieve inequality for square moduli.   For any sequence $(a_n)$ and any integers $M$ and $N$, work of Davenport and Halberstam~\cite{Davenport:1966:VTP} gives
 \begin{align}\label{inequ:squarelargesieve}
        \sum_{\substack{d \leq D \\ d\in \mathcal{S}}} \sum_{\substack{a=1 \\ (a,d)=1}}^d \left|\sum_{n=M+1}^{M+N} a_n \exp\left(\frac{2 \pi i n a}{d} \right) \right|^2 \ll  (N+D^2) \sum_{n=M+1}^{M+N} |a_n|^2.
    \end{align}
There are many   papers dedicated to replacing the factor $N+D^2$ by a smaller quantity. See for instance work of Zhao~\cite{Zhao:2004:LSS}, Baier~\cite{Baier:2006:OLS}, Baier and Zhao~\cite{Baier:2008:LSS}, Halupczok~\cite{Halupczok:2011:ANB} and Baker, Munsch and Shparlinski\cite{Baker:2022:AEL}. 
A more detailed exposition   can be found in Baier's paper~\cite{Baier:2026:LSS}.   Inequalities like~(\ref{inequ:squarelargesieve}) can be used to prove variants of the Johnsen-Selberg sieve for square moduli. Indeed, very recently Ramar{\'e}~\cite{Ramare:2026:WLP} improved on the weighted large sieve inequality and derived a corresponding version of the Johnsen-Selberg sieve.

\medskip  Since  Theorem~\ref{thm:largesieveimprovement}  involves  residue classes modulo prime squares, it is tempting to try to use the results discussed above in its proof. However, any resulting sieve inequality would involve summing over prime squares. If we take $D=N^{1/2}$, this immediately restricts our summation to a sum over no more than $O(N^{1/4}/\log(N))$ choices of $d$. By comparison, if we instead sift modulo primes, we can sum over $O(N^{1/2}/\log(N))$ choices of $d$. This is the reason why the Johnsen-Selberg sieve gives such bad bounds for the residue class configurations discussed in Section~\ref{ssec:largesievecomparison}. As such, our approach to proving  Theorem~\ref{thm:largesieveimprovement} instead builds on the standard Linnik's large sieve inequality (see Lemma~\ref{lem:largesieve.linnik}) rather than the large sieve inequality for square moduli.
 
\subsection{Additional context} 
 The work in this paper predates OpenAI's resolution of Ostmann's problem. However, while our initial motivation for   studying sums of squares was Ostmann, we believe that our new sieve estimate is interesting in its own right.

\subsection{Structure of the paper} 
In Section~\ref{sec:sieve} we prove the sieve estimate given in Theorem~\ref{thm:largesieveimprovement}. The starting point is Linnik's  large sieve inequality, which we carefully manipulate to obtain good upper bounds for the relevant residue class configurations.\medskip

In Section~\ref{sec:ostmannproofs} we combine Theorem~\ref{thm:largesieveimprovement} with Gallagher's larger sieve to  prove  Theorem~\ref{thm:2squares}  on additive decompositions of squares. The proof strategy is similar to the method described in~\cite{Elsholtz:2015:ADR} for asymptotic additive decompositions of primes.    As a corollary, we deduce that $\mathcal{S}+\mathcal{S}$ cannot be written as a sum of three different sets in a non-trivial way.

\medskip
In Section~\ref{sec:constructions} we discuss concrete alternative ways of writing $\mathcal{S}+\mathcal{S}$ as a sumset. We prove Theorem~\ref{thm:construction2}, which shows that suitable sumsets can be constructed by deleting a positive proportion of the squares, and that  no additional elements can be added to the squares.

\subsection{Notation} 

Below is a list of notation and terminology used throughout this paper.

\medskip
In this paper, lower case Roman letters  usually denote integers, and letters $p$ and $p_i$ denote primes.

 $1_{\mathcal{A}}(n)$ is the indicator function of set $\mathcal{A}$. Both $|\mathcal{A}|$ and $\#\mathcal{A}$ denote the number of elements of $\mathcal{A}$.
 
 Given a set $\mathcal{A}$, we let $\mathcal{A}(N)= \mathcal{A} \cap [1,N]$.  
 $\mathcal{S}$ is the set of squares. 
 $\mu(n)$ is the Möbius function.

 \medskip
 We   use Vinogradov  notation ($\ll$ and $\gg$) and write $X \asymp Y$ if $Y \ll X \ll Y$. If we write $X \ll_A Y$, then the implied constant may depend on $A$.   Throughout the paper, $x$ is assumed to be large.

\subsection{Standard sieve inequalities}

Several standard sieve inequalities are recorded here for later use.

\begin{lemma}[Linnik's large sieve inequality]\label{lem:largesieve.linnik} Let  $\mathcal{A} \subseteq [1,N] \cap \mathbb{N}$. Consider $(a_n)$ with $a_n \in \mathbb{C}$. Then  
 \begin{align*}
        \sum_{d \leq N^{1/2}} \sum_{\substack{a=1 \\ (a,d)=1}}^d \left|\sum_{n\leq N} a_n \exp\left(\frac{2 \pi i n a}{d} \right) \right|^2 \leq 2 N \sum_{n=1}^N |a_n|^2.
    \end{align*}
\end{lemma}

\begin{lemma}[Montgomery's large sieve] \label{lem:largesieve.montgomery}
Let $\mathcal{A} \subseteq [1,N] \cap \mathbb{N}$ and  $C(d,h) =  \frac{|\{a \in \mathcal{A}:a \equiv h \mbox{ \rm {\scriptsize mod} } d\}|}{|\mathcal{A}|} $.  Then the following inequality holds:
\begin{align}\label{equ1:largesieve.montgomery}
    |\mathcal{A}| \leq 2N \Bigg(\sum_{d \leq N^{1/2}} d \sum_{h=1}^d \Bigg|\sum_{e \mid d} \frac{\mu(e)}{e}\, C\!\left(\tfrac{d}{e},h\right)\Bigg|^2 \Bigg)^{-1}.
\end{align}
Suppose further that sets $\Omega(p) \subseteq \mathbb{Z}/p\mathbb{Z}$  satisfy $\{a \in \mathcal{A}: (a \mbox{ \rm mod } p) \in \Omega(p)\} = \emptyset$ for each prime $p$. Write $\omega(p)= |\Omega(p)|$. It follows that 
\begin{align}\label{equ2:largesieve.montgomery}
    |\mathcal{A}| \leq 2N \left(\sum_{d \leq N^{1/2}} \mu(d)^2 \prod_{p \mid d} \left(\frac{\omega(p) }{p- \omega(p)}\right) \right)^{-1}.
\end{align}
\end{lemma}

Lemma~\ref{lem:largesieve.linnik} appears in Bombieri and Davenport's paper~\cite{Bombieri:1969:OLS}. Lemma~\ref{lem:largesieve.montgomery}, a version of which first appears in Montgomery's paper~\cite{Montgomery:1968:NLS}, can be deduced from Lemma~\ref{lem:largesieve.linnik}. Both of these papers  build on earlier work of Linnik~\cite{Linnik:1941:TLS} and R\'enyi~\cite{Renyi:1951:TLS}. The second inequality in Lemma~\ref{lem:largesieve.montgomery} can also be obtained using the Selberg sieve (see for instance~\cite{Tenenbaum:2015:APN}).

\medskip

For sets which occupy less than half of the residue classes modulo $p$ for most primes $p$, Gallagher's sieve gives stronger bounds than the large sieve.

\begin{lemma}[Gallagher's larger sieve]\label{lem:largesieve.gallagher}
Let $\mathcal{A} \subseteq [1,N] \cap \mathbb{N}$. Let   $v(p)=|\{\overline{a} \mbox{ \rm mod } p: a \in \mathcal{A}\}|$. Let  $\mathcal{P}$ be a subset of the primes.    Then, assuming that the denominator of the following fraction is positive, we have
\begin{align*}
    |\mathcal{A}| \leq \frac{-\log(N) + \sum_{p \in \mathcal{P}} \log(p)}{-\log(N) + \sum_{p \in \mathcal{P}} \frac{\log(p) }{v(p)}}.
\end{align*}
\end{lemma}
This was proven by Gallagher in~\cite{Gallagher:1971:LS}. 

\medskip

Finally, in one later proof we will briefly need to make use of a simple small sieve inequality:

\begin{lemma}[Brun's small sieve]\label{lem:small.brun}
Let $\mathcal{C} \subseteq \mathbb{N}$. Write $X =|\mathcal{C}|$. Let $\varepsilon>0$ and let $\mathcal{P} \subseteq [1,X^\varepsilon]$ be a subset of the primes. Suppose that for some $\kappa, D>0$, the integers $\omega(p) $ satisfy the following properties: 
\begin{enumerate}[{\rm (1)}]
    \item $0 \leq \frac{\omega(p) }{p} \leq \frac{1}{2}$ when $p \in \mathcal{P}$, 
    \item $\left|\#\{c \in \mathcal{C}: p_1 \dots p_j \,\, \mbox{\rm divides } c\} - \frac{\omega(p_1) \dots \omega(p_j)}{p_1 \dots p_j} \#\mathcal{C}\right| \leq \omega(p_1) \dots \omega(p_j)$ when $p_1, \dots, p_j \in \mathcal{P}$ $($distinct$)$,
    \item $\sum_{w \leq p <X^\varepsilon, \,p \in \mathcal{P}} \frac{\omega(p) \log(p)}{p} \leq \kappa \log(\frac{X^\varepsilon}{w}) +D$ when $2 \leq w \leq X^\varepsilon.$
\end{enumerate}
Then we have the following upper bound: $$\#\{c \in \mathcal{C}: \forall p \in \mathcal{P} \,\, (c,p)=1\} \ll_{\varepsilon, \kappa, D}  X\prod_{p \in \mathcal{P}} \left(1-\frac{\omega(p)}{p} \right).$$
\end{lemma}

Brun's sieve was developed by Viggo Brun in~\cite{Brun:1915:BSS}. The version stated above is Theorem~2.2 of~\cite{Halberstam:2013:SM}.

\section{Sieving modulo prime squares}\label{sec:sieve}

The aim of this section is to improve upper bounds for $|\mathcal{A}|$ when  there are many primes $p$ and residue classes $c$ mod $p$ for which  $|\{\overline{a} \in \mathbb{Z}/p^2\mathbb{Z}: a \in \mathcal{A}, a \equiv c \mbox{ mod } p\}|$ is small. That is, unlike in a typical application of  the Selberg sieve or the large sieve (see inequality (\ref{equ2:largesieve.montgomery}) of Lemma~\ref{lem:largesieve.montgomery}), the set $\mathcal{A}$ might not avoid any residue classes mod $p$ entirely, but there are many  $c$ mod $p$ for which only few of the corresponding residue classes  $\{d \mbox{ mod } p^2: d \equiv c \mbox{ mod } p\}$ occur in $\mathcal{A}$.   

\medskip
 
The main result of Section~\ref{sec:sieve}, stated in Theorem~\ref{thm:largesieveimprovement}, is the following: 
If  $\mathcal{A} \subseteq \mathbb{N} \cap [1,N]$  with 
$$|\{\overline{a} \in \mathbb{Z}/p^2\mathbb{Z}: a \in \mathcal{A}, a \equiv c \mbox{ mod } p\}| \leq \Theta \quad \mbox{ for all } p \in \mathcal{P} \mbox{ and } c \in \Omega(p),$$
 (where $\mathcal{P} \subseteq [\frac{1}{2}M,M]$ with $M^k \leq N^{1/2}$), then the following upper bound is satisfied: 
\begin{align*}
&|\mathcal{A}| \leq  \xi(k,\omega,\mathcal{P},\Theta) \, 2N \left(  \mathop{\sum \dots \sum}_{\substack{ p_1, \dots, p_k \in \mathcal{P} \\ p_1> \dots > p_k }}  \prod_{i=1}^k \left( \frac{\omega(p_i)}{p_i-\omega(p_i)}\right) \right)^{-1}.
\end{align*}

The quantity $ \xi(k,\omega,\mathcal{P},\Theta)$ is described in Definition~\ref{def:sievefactors} and is independent of $N$.

\subsubsection{Proof idea}
 When applying Linnik's large sieve inequality to the indicator function $1_{\mathcal{A}}(n)$, the size of the LHS depends on the distribution of $\mathcal{A}$ over the residue classes modulo $d$ for $d\leq N^{1/2}$ -- the LHS is big if many residue classes are missing in $\mathcal{A}$, and the LHS is small if $\mathcal{A}$ is approximately equidistributed over many residue classes.  We have residue class restrictions modulo prime squares. To utilize these, one might consider the contribution of square moduli $d=(p_1 \dots p_k)^2$ to the large sieve. (This is how the Johnsen-Selberg sieve inequality is obtained.)   However, the condition $d \leq N^{1/2}$ then implies that one may only use moduli $p_1 \dots p_k \leq N^{1/4}$, which is a severe restriction on the size of the LHS. To get around this issue,  we instead observe that the given residue class restrictions modulo prime squares imply that certain residue classes modulo primes are underrepresented, and   then  sieve modulo $d=(p_1 \dots p_k)$, allowing for a larger selection of moduli in the denominator. 
 
\medskip 
 In short, to prove Theorem~\ref{thm:largesieveimprovement}, we will give good lower bounds on the LHS of the large sieve inequality when certain residue classes modulo $d$ are not missing, but  instead occur only rarely in $\mathcal{A}$. This  key idea   lets us  use moduli $p_1 \dots p_k$ rather than $(p_1  \dots p_k)^2$, allowing us to sum over a much larger collection of moduli in the large sieve, and thus improve on the Johnsen-Selberg sieve.

\subsection{Manipulating the large sieve inequality}\label{ssec:denominator.manipulation}

As a first step towards the proof of Theorem~\ref{thm:largesieveimprovement}, we rearrange the LHS of the large sieve inequality into the following convenient expression:

\begin{lemma}\label{lem:largesieve.denominator}
Let $(a_n)$ be an arbitrary complex sequence. Then 
 \begin{align*}
         \sum_{\substack{a=1 \\ (a,d)=1}}^d \left|\sum_{n\leq N} a_n \exp\left(\frac{2 \pi i n a}{d} \right) \right|^2  &= d\sum_{e \mid d} \frac{\mu(e)}{e}
       \sum_{h=1}^{d/e} \Bigg|\sum_{\substack{n\leq N \\ n \equiv h (d/e)}} a_n   \Bigg|^2.
    \end{align*}
\end{lemma}

\begin{proof}
We use  orthogonality relations for sums of roots of unity.
   \begin{align*}
         \sum_{\substack{a=1 \\ (a,d)=1}}^d \left|\sum_{n\leq N} a_n \exp\left(\frac{2 \pi i n a}{d} \right) \right|^2 &=
        \sum_{n\leq N}\sum_{m\leq N} a_n \overline{a_m} \sum_{\substack{a=1 }}^d 1_{(a,d)=1} \exp\left(\frac{2 \pi i (n-m) a}{d} \right)\\
        &=
        \sum_{n\leq N}\sum_{m\leq N} a_n \overline{a_m} \sum_{e \mid d} \sum_{\substack{a=1   }}^d 1_{e\mid a} \, \mu(e) \exp\left(\frac{2 \pi i (n-m) a}{d} \right) \\
         &=\sum_{e \mid d} \mu(e)
        \sum_{n\leq N}\sum_{m\leq N} a_n \overline{a_m} \sum_{\substack{b=1  }}^{d/e}  \exp\left(\frac{2 \pi i (n-m) b}{(d/e)} \right)
         \\
         &=\sum_{e \mid d} \mu(e)
        \sum_{n\leq N}\sum_{\substack{m\leq N}} a_n \overline{a_m} \,\,\frac{d}{e} 1_{n \equiv m (d/e)}  
         \\
         &=d\sum_{e \mid d} \frac{\mu(e)}{e}
       \sum_{h=1}^{d/e} \Bigg|\sum_{\substack{n\leq N \\ n \equiv h (d/e)}} a_n   \Bigg|^2.
    \end{align*}
 This is what we wanted to show.
    \end{proof}
    
In particular, if $\mathcal{A} \subseteq [1,N] \cap \mathbb{N}$ and  $C(d,h) =  \frac{|\{a \in \mathcal{A}:a \equiv h \mbox{ \rm {\scriptsize mod} } d\}|}{|\mathcal{A}|} $, Lemma~\ref{lem:largesieve.linnik} with $a_n=1_{\mathcal{A}}(n)$ now gives the following inequality:
\begin{align*}
|\mathcal{A}| \leq \frac{2N}{\sum_{d\leq N^{1/2}}d\sum_{e \mid d} \frac{\mu(e)}{e}
       \sum_{h=1}^{d/e} C\left(\frac{d}{e},h\right)^2}.
\end{align*}

Next we show that if we choose a set $S$ of residue classes  $h$  modulo $d$ and replace $\mathcal{A}$ by a set $\mathcal{B}$    which has the same number of elements in $S$ as $\mathcal{A}$, but which is  equidistributed over these residue classes $S$ (and which has the same elements  as $\mathcal{A}$ for $h \not\in S$), then the LHS of the large sieve becomes smaller.  The following inequality, which shows how taking averages decreases a quantity, is the key to this observation:

\begin{lemma} \label{lem:usedinlargesieve}
Let $p_1, \dots, p_k$ be distinct primes. Consider some  numbers $\alpha(h_1,\dots, h_k) \in \mathbb{R}$, a set $U \subseteq \{1,\dots,k\}$, and some non-empty sets $S_i \subset \{1, \dots, p_i\}$ for $i \in U$. Define $\beta(h_1,\dots,h_k)$ to equal
\begin{align*}
\beta(h_1,\dots,h_k) = \left(\prod_{i \in U} \frac{1}{|T_i|}\right)\sum_{l_1 \in T_1}\!\! \dots \!\! \sum_{l_k \in T_k} \alpha(l_1,\dots,l_k) \,\, \mbox{ where } \,\, T_i = \begin{cases}
S_i \quad \quad\quad\quad\quad\,\,\,\, \mbox{ if } i \in U, h_i \in S_i, \\
\{1,\dots,p_i\} \setminus S_i \,\, \mbox{ if } i \in U, h_i \not\in S_i, \\
\{h_i\}   \quad \quad\quad\quad\quad\! \mbox{ if } i \not\in U.
\end{cases}
\end{align*}
Then the following inequality holds:
\begin{align*}  &\sum_{I \subseteq \{1,\dots,k\}}\left(\frac{(-1)^{\#I}}{\prod_{i \in I} p_i}\right)
     \mathop{\sum \dots \sum}_{\substack{ h_i \in \{1, \dots, p_i\} \\ \mbox{ \rm \begin{footnotesize}
     for
     \end{footnotesize}} i \not\in I }}  \Bigg(   \mathop{\sum \dots \sum}_{\substack{ h_i \in \{1, \dots, p_i\} \\ \mbox{ \rm \begin{footnotesize}
     for
     \end{footnotesize}} i \in I }} \alpha(h_1,\dots,h_k)  \Bigg)^2 \\ 
     \geq 
     &\sum_{I \subseteq \{1,\dots,k\}}\left(\frac{(-1)^{\#I}}{\prod_{i \in I} p_i}\right)
     \mathop{\sum \dots \sum}_{\substack{ h_i \in \{1, \dots, p_i\} \\ \mbox{ \rm \begin{footnotesize}
     for
     \end{footnotesize}} i \not\in I }}  \Bigg( \mathop{\sum \dots \sum}_{\substack{ h_i \in \{1, \dots, p_i\} \\ \mbox{ \rm \begin{footnotesize}
     for
     \end{footnotesize}} i \in I }} \beta(h_1,\dots,h_k)  \Bigg)^2.
       \end{align*}
      \end{lemma}
 
\begin{proof} Note that if $U = \{1,\dots,j\}$, then
\begin{align*}
\beta(h_1,\dots,h_k) = \frac{1}{|T_1|}\sum_{l_1 \in T_1}\left( \frac{1}{|T_2|}\sum_{l_2 \in T_2} \left(  \dots  \left(\frac{1}{|T_j|}\sum_{l_j \in T_j} \alpha(l_1,\dots,l_j,h_{j+1},\dots,h_k)\right) \dots \right)\right), 
\end{align*}
where $T_i = S_i$ if $h_i \in S_i$ and $T_i = \{1,\dots, p_i\} \setminus S_i$ otherwise. In particular, the function $\beta(h_1,\dots,h_k)$ can also be produced by repeatedly applying the construction  given at the start of Lemma~\ref{lem:usedinlargesieve} with  a singleton $U$, starting with $\alpha(h_1,\dots,h_k)$ and using $U=\{j\}$, $U=\{j-1\}$, \dots, and finally $U= \{1\}$.  
Hence if we can show that Lemma~\ref{lem:usedinlargesieve} holds for any  singleton $U=\{i\}$, then it must also hold for $U=\{1,\dots,j\}$. However, by symmetry it is clear that if Lemma~\ref{lem:usedinlargesieve} holds for  $U=\{1,\dots,j\}$, then it holds for any other $U$ of cardinality $j$. Thus it suffices to prove Lemma~\ref{lem:usedinlargesieve} for $U= \{1\}$. From now on we will thus assume that $U= \{1\}$.

\medskip

\underline{Contribution 1}: Sets $I$ with $1 \in I$. 
Note that for $U= \{1\}$ and $1 \in I$ we have \begin{align*} 
     &\mathop{\sum \dots \sum}_{\substack{ h_i \in \{1, \dots, p_i\} \\ \mbox{ \rm \begin{footnotesize}
     for
     \end{footnotesize}} i \not\in I }}  \Bigg|   \mathop{\sum \dots \sum}_{\substack{ h_i \in \{1, \dots, p_i\} \\ \mbox{ \rm \begin{footnotesize}
     for
     \end{footnotesize}} i \in I }} \beta(h_1,\dots,h_k)  \Bigg|^2 \\&= 
     \mathop{\sum \dots \sum}_{\substack{ h_i \in \{1, \dots, p_i\} \\ \mbox{ \rm \begin{footnotesize}
     for
     \end{footnotesize}} i \not\in I }}    \Bigg|  \mathop{\sum \dots \sum}_{\substack{ h_i \in \{1, \dots, p_i\} \\ \mbox{ \rm \begin{footnotesize}
     for
     \end{footnotesize}} i \in I\setminus \{1\} }} \!\! \Bigg(\sum_{h_1 \in S_1}  \frac{1}{|S_1|} \sum_{l_1 \in S_1}\alpha(l_1,h_2,\dots,h_k) \! \Bigg)\!+\! \Bigg(\sum_{h_1 \not\in S_1} \frac{1}{p_1-|S_1|} \sum_{l_1 \not\in S_1}\alpha(l_1,h_2,\dots,h_k) \! \Bigg)\Bigg|^2\\
      &=\mathop{\sum \dots \sum}_{\substack{ h_i \in \{1, \dots, p_i\} \\ \mbox{ \rm \begin{footnotesize}
     for
     \end{footnotesize}} i \not\in I }}  \Bigg|   \sum_{l_1 \in \{1,\dots,p_1\}}   \mathop{\sum \dots \sum}_{\substack{ h_i \in \{1, \dots, p_i\} \\ \mbox{ \rm \begin{footnotesize}
     for
     \end{footnotesize}} i \in I \setminus \{1\}}} \alpha(l_1,h_2,\dots,h_k)  \Bigg|^2= \mathop{\sum \dots \sum}_{\substack{ h_i \in \{1, \dots, p_i\} \\ \mbox{ \rm \begin{footnotesize}
     for
     \end{footnotesize}} i \not\in I }}  \Bigg|   \mathop{\sum \dots \sum}_{\substack{ h_i \in \{1, \dots, p_i\} \\ \mbox{ \rm \begin{footnotesize}
     for
     \end{footnotesize}} i \in I }} \alpha(h_1,\dots,h_k)  \Bigg|^2 .
       \end{align*}
       Hence it suffices to consider the contribution of sets $I$ with  $I \subseteq \{2,\dots, k\}$.  \medskip

       \underline{Contribution 2}: Sets with $1 \not\in I$.  We are done if we can show that for each $T_1 \in \{S_1, \{1,\dots,p_1\}\setminus S_1\}$, the following holds:
       \begin{align} \label{inequ:needtoshow} &\sum_{I \subseteq \{2,\dots,k\}}\left(\frac{(-1)^{\#I}}{\prod_{i \in I} p_i}\right)
  \sum_{h_1 \in T_1}   \mathop{\sum \dots \sum}_{\substack{ h_i \in \{1, \dots, p_i\} \\ \mbox{ \rm \begin{footnotesize}
     for
     \end{footnotesize}} i \in \{2,\dots,k\}\setminus I }}  \Bigg(   \mathop{\sum \dots \sum}_{\substack{ h_i \in \{1, \dots, p_i\} \\ \mbox{ \rm \begin{footnotesize}
     for
     \end{footnotesize}} i \in I }} \alpha(h_1,\dots,h_k)  \Bigg)^2  \nonumber \\  
     \geq 
     &\sum_{I \subseteq \{2,\dots,k\}}\left(\frac{(-1)^{\#I}}{\prod_{i \in I} p_i}\right)
    \sum_{h_1 \in T_1} \mathop{\sum \dots \sum}_{\substack{ h_i \in \{1, \dots, p_i\} \\ \mbox{ \rm \begin{footnotesize}
     for
     \end{footnotesize}} i \in \{2,\dots,k\}\setminus I }}  \Bigg( \mathop{\sum \dots \sum}_{\substack{ h_i \in \{1, \dots, p_i\} \\ \mbox{ \rm \begin{footnotesize}
     for
     \end{footnotesize}} i \in I }} \beta(h_1,\dots,h_k)  \Bigg)^2.
       \end{align}       
      For this purpose, we  define the function $\gamma(t;h_1,\dots,h_k)=t\alpha(h_1,\dots,h_k)+(1-t)\beta(h_1,\dots,h_k)$. The following expression is a polynomial  of degree at most $2$ in $t$:
       \begin{align*} 
   \Phi(t)= \sum_{I \subseteq \{2,\dots,k\}}\left(\frac{(-1)^{\#I}}{\prod_{i \in I} p_i}\right)
     \sum_{h_1 \in T_1}   \mathop{\sum \dots \sum}_{\substack{ h_i \in \{1, \dots, p_i\} \\ \mbox{ \rm \begin{footnotesize}
     for
     \end{footnotesize}} i \in \{2,\dots,k\}\setminus I }} \Bigg(   \mathop{\sum \dots \sum}_{\substack{ h_i \in \{1, \dots, p_i\} \\ \mbox{ \rm \begin{footnotesize}
     for
     \end{footnotesize}} i \in I }} \gamma(t;h_1,\dots,h_k)  \Bigg)^2.
     \end{align*}
       Differentiating with respect to $t$,  we get that $\Phi'(t)$ equals
             \begin{align*} 
\sum_{I \subseteq \{2,\dots,k\}}\left(\frac{(-1)^{\#I}}{\prod_{i \in I} p_i}\right)
    \sum_{h_1 \in T_1} \!\!\!  \mathop{\sum \dots \sum}_{\substack{ h_i \in \{1, \dots, p_i\} \\ \mbox{ \rm \begin{footnotesize}
     for
     \end{footnotesize}} i \in \{2,\dots,k\}\setminus I }} \!\!\!\! 2\Bigg(   \mathop{\sum \dots \sum}_{\substack{ h_i \in \{1, \dots, p_i\} \\ \mbox{ \rm \begin{footnotesize}
     for
     \end{footnotesize}} i \in I }} \gamma(t;h_1,\dots,h_k)  \Bigg)\Bigg(   \mathop{\sum \dots \sum}_{\substack{ h_i \in \{1, \dots, p_i\} \\ \mbox{ \rm \begin{footnotesize}
     for
     \end{footnotesize}} i \in I }} (\alpha-\beta)(h_1,\dots,h_k) \Bigg).
     \end{align*}
     Note that if we only consider $h_1 \in T_1$, then $\gamma(0;h_1,\dots,h_k) = \beta(h_1,\dots,h_k) = \frac{1}{|T_1|} \sum_{l_1 \in T_1} \alpha(l_1,h_2,\dots,h_k)$ does not depend on $h_1$. Write $\gamma(0;h_1,\dots,h_k)=\gamma^*(h_2,\dots,h_k)$ for   $h_1 \in T_1$. In particular, we have
         \begin{align*} \Phi'(0)=\!\!\!\! 
\sum_{I \subseteq \{2,\dots,k\}}\!\!\!\left(\frac{2(-1)^{\#I}}{\prod_{i \in I} p_i}\right)
    \!\!\!\!\!\!  \mathop{\sum \dots \sum}_{\substack{ h_i \in \{1, \dots, p_i\} \\ \mbox{ \rm \begin{footnotesize}
     for
     \end{footnotesize}} i \in \{2,\dots,k\}\setminus I }} \!\!\! \!\!\!\! \Bigg(   \mathop{\sum \dots \sum}_{\substack{ h_i \in \{1, \dots, p_i\} \\ \mbox{ \rm \begin{footnotesize}
     for
     \end{footnotesize}} i \in I }} \gamma^*(h_2,\dots,h_k)\! \Bigg)\!\!\sum_{h_1 \in T_1}\!\!\Bigg(   \mathop{\sum \dots \sum}_{\substack{ h_i \in \{1, \dots, p_i\} \\ \mbox{ \rm \begin{footnotesize}
     for
     \end{footnotesize}} i \in I }} (\alpha-\beta)(h_1,\dots,h_k)\! \Bigg).
     \end{align*}
     But for any $h_2, \dots, h_k$, we have 
     \begin{align*}
   \sum_{h_1 \in T_1}  (\alpha-\beta)(h_1,\dots,h_k) =  \sum_{h_1 \in T_1}\alpha(h_1,\dots,h_k) -   \sum_{h_1 \in T_1} \frac{1}{|T_1|}\sum_{l_1 \in T_1} \alpha(l_1,h_2, \dots,h_k) =0.
     \end{align*}
     Hence we now know that $\Phi'(0)=0$. \medskip
     
      Next we compute the second derivative of $\Phi(t)$. 
      \begin{align*}
      \Phi''(t)&=\sum_{I \subseteq \{2,\dots,k\}}\left(\frac{(-1)^{\#I}}{\prod_{i \in I} p_i}\right)
    \sum_{h_1 \in T_1} \!\!\!  \mathop{\sum \dots \sum}_{\substack{ h_i \in \{1, \dots, p_i\} \\ \mbox{ \rm \begin{footnotesize}
     for
     \end{footnotesize}} i \in \{2,\dots,k\}\setminus I }} \!\!\!\! 2\Bigg(   \mathop{\sum \dots \sum}_{\substack{ h_i \in \{1, \dots, p_i\} \\ \mbox{ \rm \begin{footnotesize}
     for
     \end{footnotesize}} i \in I }} (\alpha-\beta)(h_1,\dots,h_k) \Bigg)^2 \\
     &=\frac{2}{p_2 \dots p_k} \sum_{h_1 \in T_1}\Bigg\{\sum_{I \subseteq \{2,\dots,k\}}(-1)^{\#I} \Big(\prod_{\substack{i=2\\ i \not\in I}}^{k}p_i \Big)
    \!\!\!  \mathop{\sum \dots \sum}_{\substack{ h_i \in \{1, \dots, p_i\} \\ \mbox{ \rm \begin{footnotesize}
     for
     \end{footnotesize}} i \in \{2,\dots,k\}\setminus I }} \!\!\!\! \Bigg(   \mathop{\sum \dots \sum}_{\substack{ h_i \in \{1, \dots, p_i\} \\ \mbox{ \rm \begin{footnotesize}
     for
     \end{footnotesize}} i \in I }} (\alpha-\beta)(h_1,\dots,h_k) \Bigg)^2\Bigg\}.
     \end{align*}
Using orthogonality relations for exponential sums, we notice  
\begingroup
\allowdisplaybreaks
\begin{align*}
&\mathop{\sum \dots \sum}_{\substack{ s_i \in \{1, \dots, p_i-1\} \\ \mbox{ \rm \begin{footnotesize}
     for
     \end{footnotesize}} i \in \{2,\dots,k\} }}\Bigg|   \mathop{\sum \dots \sum}_{\substack{ h_i \in \{1, \dots, p_i\} \\ \mbox{ \rm \begin{footnotesize}
     for
     \end{footnotesize}} i \in \{2,\dots,k\} }}  \exp\left(\frac{2\pi i h_2s_2}{p_2}\right) \dots\exp\left(\frac{2\pi i h_ks_k}{p_k} \right)(\alpha-\beta)(h_1,\dots,h_k)  \Bigg|^2\\
     =& \mathop{\sum \dots \sum}_{\substack{ h_i, v_i \in \{1, \dots, p_i\} \\ \mbox{ \rm \begin{footnotesize}
     for
     \end{footnotesize}} i \in \{2,\dots,k\} }}(\alpha-\beta)(h_1,h_2,\dots,h_k) (\alpha-\beta)(h_1,v_2,\dots,v_k) \sum_{s_2=1}^{p_2-1} \dots \sum_{s_k=1}^{p_k-1}  \Bigg(\prod_{j=2}^k\exp\left(\frac{2\pi i (h_j-v_j)s_j}{p_j}\right)\Bigg) 
     \\
          =& \mathop{\sum \dots \sum}_{\substack{ h_i, v_i \in \{1, \dots, p_i\} \\ \mbox{ \rm \begin{footnotesize}
     for
     \end{footnotesize}} i \in \{2,\dots,k\} }}(\alpha-\beta)(h_1,h_2,\dots,h_k) (\alpha-\beta)(h_1,v_2,\dots,v_k)\prod_{j=2}^k \Bigg(-1+ \sum_{s_j=1}^{p_j}    \exp\left(\frac{2\pi i (h_j-v_j)s_j}{p_j}\right) \Bigg)
     \\
     =& \mathop{\sum \dots \sum}_{\substack{ h_i, v_i \in \{1, \dots, p_i\} \\ \mbox{ \rm \begin{footnotesize}
     for
     \end{footnotesize}} i \in \{2,\dots,k\} }}(\alpha-\beta)(h_1,\dots,h_k) (\alpha-\beta)(h_1,v_2,\dots,v_k) \!\!\!\!\!\! \sum_{I \subseteq \{2, \dots, k\}}  \!\!\! (-1)^{\# I} \prod_{\substack{j \not\in I\\j\neq 1}} \Bigg( \sum_{s_j=1}^{p_j} \! \exp\! \left(\frac{2\pi i (h_j-v_j)s_j}{p_j}\right)\! \Bigg)
          \\
     =& \sum_{I \subseteq \{2,\dots,k\}}(-1)^{\#I}\Big(\prod_{\substack{ j=2 \\ j \not \in I }}^{k} p_j \Big) \mathop{\sum \dots \sum}_{\substack{ h_i \in \{1, \dots, p_i\} \\ \mbox{ \rm \begin{footnotesize}
     for
     \end{footnotesize}} i \in \{2,\dots,k\}\setminus I }}\Bigg(\mathop{\sum \dots \sum}_{\substack{ h_i \in \{1, \dots, p_i\} \\ \mbox{ \rm \begin{footnotesize}
     for
     \end{footnotesize}} i \in I }}(\alpha-\beta)(h_1,\dots,h_k)     \Bigg)^2.
      \end{align*}
      \endgroup
      Hence we now know that $\Phi''(t)$ equals
       \begin{align*}
      \frac{2}{p_2 \dots p_k} \sum_{h_1 \in T_1}\Bigg\{\mathop{\sum \dots \sum}_{\substack{ s_i \in \{1, \dots, p_i-1\} \\ \mbox{ \rm \begin{footnotesize}
     for
     \end{footnotesize}} i \in \{2,\dots,k\} }}\Bigg|   \mathop{\sum \dots \sum}_{\substack{ h_i \in \{1, \dots, p_i\} \\ \mbox{ \rm \begin{footnotesize}
     for
     \end{footnotesize}} i \in \{2,\dots,k\} }}  \exp\left(\frac{2\pi i h_2s_2}{p_2}\right) \dots\exp\left(\frac{2\pi i h_ks_k}{p_k} \right)(\alpha-\beta)(h_1,\dots,h_k)  \Bigg|^2\Bigg\}
     \end{align*}
     and therefore $\Phi''(t) \geq 0$ for all $t$. Since $\Phi(t)$ is a polynomial of degree at most $2$ and  $\Phi'(0)=0$, this implies $\Phi(t) \geq \Phi(0)$ for all $t$. In particular, $\Phi(1) \geq \Phi(0)$. But this is precisely the inequality (\ref{inequ:needtoshow}) we needed to show. This concludes the proof.
\end{proof}

We can combine Lemma~\ref{lem:largesieve.denominator} and Lemma~\ref{lem:usedinlargesieve} as follows:

\begin{cor}\label{cor:denominator.rearrangement}
Let $\mathcal{A} \subseteq [1,N] \cap \mathbb{N}$. Let $p_1, \dots, p_k$ be distinct primes,   $\Omega(p_i) \subset \mathbb{Z}/p_i\mathbb{Z}$ with $\Omega(p_i) \neq \emptyset$ and
\begin{align*}
\Lambda(e_1,\dots,e_k) = \#\left\{n \in \mathcal{A}: \begin{array}{l} n  \mbox{ \rm mod } p_i\in \Omega(p_i)  \mbox{ when } e_i=0 \\  n \mbox{ \rm mod } p_i\not\in \Omega(p_i)   \mbox{ when } e_i=1 \end{array}\right\}.
\end{align*}
Set $t_i(0) = |\Omega(p_i)|$ and $t_i(1) = p_i - |\Omega(p_i)|$. We then have the following lower bound: 
 \begin{align*}
         \sum_{\substack{a=1 \\ (a,p_1\dots p_k)=1}}^{p_1\dots p_k} \left|\sum_{\substack{n\leq N \\ n \in \mathcal{A}}} \exp\left(\frac{2 \pi i n a}{p_1\dots p_k} \right) \right|^2   \geq \!\!
       \sum_{I \subseteq \{1, \dots, k\}} (-1)^{\# I}
          \mathop{\sum \dots \sum}_{\substack{ e_i \in \{0,1\}\\ \mbox{ \rm \begin{footnotesize}
     for
     \end{footnotesize}} i \not\in I }} \prod_{i   \not\in I} \Big(\frac{p_i}{t_i(e_i)} \Big)  \Bigg(     \mathop{\sum \dots \sum}_{\substack{ e_i \in \{0,1\}\\ \mbox{ \rm \begin{footnotesize}
     for
     \end{footnotesize}} i \in I }} \Lambda(e_1,\dots,e_k)   \Bigg)^2.
    \end{align*}
\end{cor}

\begin{proof}
Applying Lemma~\ref{lem:largesieve.denominator} with $d = p_1 \dots p_k$, we have
 \begin{align}
         &\sum_{\substack{a=1 \\ (a,p_1\dots p_k)=1}}^{p_1\dots p_k} \left|\sum_{\substack{n\leq N \\ n \in \mathcal{A}}} \exp\left(\frac{2 \pi i n a}{p_1\dots p_k} \right) \right|^2  = \Big(\prod_{i=1}^k p_i \Big)\sum_{e \mid p_1\dots p_k} \left(\frac{\mu(e)}{e}\right)
       \sum_{h=1}^{(p_1\dots p_k)/e} \Bigg|\sum_{\substack{n\leq N \\ n \equiv h (p_1 \dots p_k/e)}} 1_{\mathcal{A}}(n)    \Bigg|^2 \nonumber \\
       &= \Big(\prod_{i=1}^k p_i \Big)\sum_{I \subseteq \{1, \dots, k\}} \left(\frac{(-1)^{\# I}}{\prod_{i \in I} p_i}\right)
          \mathop{\sum \dots \sum}_{\substack{ h_i \in \{1, \dots, p_i\}   \\ \mbox{ \rm \begin{footnotesize}
     for
     \end{footnotesize}} i \not\in I }}\Bigg|    \mathop{\sum \dots \sum}_{\substack{ h_i \in \{1, \dots, p_i\} \\ \mbox{ \rm \begin{footnotesize}
     for
     \end{footnotesize}} i \in I }}\Bigg(\sum_{\substack{n\leq N \\ \forall i: n \equiv h_i (p_i) }} 1_{\mathcal{A}}(n)  \Bigg)   \Bigg|^2. \label{cor:denominator.step1}
     \end{align}

   We recall Lemma~\ref{lem:usedinlargesieve}. For $U = \{1,\dots,k\}$ and $S_i \subset  \{1,\dots,p_i\}$ with $(S_i \mbox{ mod } p_i) = \Omega(p_i)$ and some arbitrary real valued function $\alpha(h_1, \dots, h_k)$,  the lemma gives  
    \begin{align}  &\sum_{I \subseteq \{1,\dots,k\}}\left(\frac{(-1)^{\#I}}{\prod_{i \in I} p_i}\right)
     \mathop{\sum \dots \sum}_{\substack{ h_i \in \{1, \dots, p_i\} \\ \mbox{ \rm \begin{footnotesize}
     for
     \end{footnotesize}} i \not\in I }}  \Bigg(   \mathop{\sum \dots \sum}_{\substack{ h_i \in \{1, \dots, p_i\} \\ \mbox{ \rm \begin{footnotesize}
     for
     \end{footnotesize}} i \in I }} \alpha(h_1,\dots,h_k)  \Bigg)^2 \nonumber \\ 
     &\geq \nonumber
     \sum_{I \subseteq \{1,\dots,k\}}\left(\frac{(-1)^{\#I}}{\prod_{i \in I} p_i}\right)
     \mathop{\sum \dots \sum}_{\substack{ h_i \in \{1, \dots, p_i\} \\ \mbox{ \rm \begin{footnotesize}
     for
     \end{footnotesize}} i \not\in I }}  \Bigg( \mathop{\sum \dots \sum}_{\substack{ h_i \in \{1, \dots, p_i\} \\ \mbox{ \rm \begin{footnotesize}
     for
     \end{footnotesize}} i \in I }} \left(\prod_{i =1}^k \frac{1}{|T_i|}\right)\sum_{l_1 \in T_1}\!\! \dots \!\! \sum_{l_k \in T_k} \alpha(l_1,\dots,l_k)  \Bigg)^2 \,\, \\&\mbox{ where } \,\, T_i = \begin{cases}
S_i \quad \quad\quad\quad\quad\,\,\,\,\, \mbox{ if } ( h_i \mbox{ mod } p_i ) \in \Omega(p_i), \\
\{1,\dots,p_i\} \setminus S_i \,\, \mbox{ if }  (h_i \mbox{ mod } p_i) \not\in \Omega(p_i)
\end{cases}  \nonumber \\
 &= \label{cor:denominator.step2}
     \sum_{I \subseteq \{1,\dots,k\}}\left(\frac{(-1)^{\#I}}{\prod_{i \in I} p_i}\right)
     \mathop{\sum \dots \sum}_{\substack{ e_i \in \{0,1\} \\ \mbox{ \rm \begin{footnotesize}
     for
     \end{footnotesize}} i \not\in I }} \Bigg( \prod_{i \not \in I} \frac{1}{|T_i(e_i)|} \Bigg)  \Bigg( \mathop{\sum \dots \sum}_{\substack{ e_i \in \{0,1\} \\ \mbox{ \rm \begin{footnotesize}
     for
     \end{footnotesize}} i \in I }} \sum_{l_1 \in T_1(e_1)}\!\! \dots \!\! \sum_{l_k \in T_k(e_k)} \alpha(l_1,\dots,l_k)  \Bigg)^2
     \\
     &\mbox{ where } T_i(0) = S_i \mbox{ and }  T_i(1)= \{1,\dots,p_i\} \setminus S_i.  \nonumber
       \end{align}

 We now apply inequality (\ref{cor:denominator.step2})   to identity (\ref{cor:denominator.step1}) 
by choosing  $$\alpha(h_1,\dots,h_k) =\Bigg(\sum_{\substack{n\leq N \\ \forall i: n \equiv h_i (p_i) }} 1_{\mathcal{A}}(n)  \Bigg)= \#\{n \in \mathcal{A}: n \equiv h_i \mbox{ mod } p_i \mbox{ for } i \in \{1,\dots, k\}\}.$$ 
  We also  recall  that $|T_i(0)|=|S_i|=|\Omega(p_i)|=t_i(0)$ and $|T_i(1)| = t_i(1)$ and note that
 \begin{align*}
 \Lambda(e_1,\dots,e_k) =   \sum_{l_1 \in T_1(e_1)}\!\! \dots \!\! \sum_{l_k \in T_k(e_k)} \alpha(l_1,\dots,l_k)
 \end{align*}
for  $T_i(0) = S_i \mbox{ and }  T_i(1)= \{1,\dots,p_i\} \setminus S_i$. This gives the following lower bound:
 \begin{align*}
         \sum_{\substack{a=1 \\ (a,p_1\dots p_k)=1}}^{p_1\dots p_k} \left|\sum_{\substack{n\leq N \\ n \in \mathcal{A}}} \exp\left(\frac{2 \pi i n a}{p_1\dots p_k} \right) \right|^2  \geq \!\! 
          \sum_{I \subseteq \{1,\dots,k\}}\!\!\!(-1)^{\#I}
     \mathop{\sum \dots \sum}_{\substack{ e_i \in \{0,1\} \\ \mbox{ \rm \begin{footnotesize}
     for
     \end{footnotesize}} i \not\in I }} \Bigg( \prod_{i \not \in I} \frac{p_i}{t_i(e_i)} \Bigg)  \Bigg( \mathop{\sum \dots \sum}_{\substack{ e_i \in \{0,1\} \\ \mbox{ \rm \begin{footnotesize}
     for
     \end{footnotesize}} i \in I }} \Lambda(e_1,\dots,e_k) \Bigg)^2.
\end{align*} 
  But this is exactly what we wanted to show.
\end{proof}

   \subsection{Underrepresented residue classes}\label{ssec:unusually.rare} 

We now recall what the sets treated by Theorem~\ref{thm:largesieveimprovement} look like:  For given $\mathcal{P} \subseteq \{p \in [\tfrac{1}{2}M,M]: p \mbox{ prime}\}$ and   $\Omega(p) \subset \mathbb{Z}/p\mathbb{Z}$, we think about  $\mathcal{A} \subseteq [1,N] \cap \mathbb{N}$ which have the property that for every $p \in \mathcal{P}$ and $c \in    \Omega(p)$,   $$|\{\overline{a} \in \mathbb{Z}/p^2\mathbb{Z}: a \in \mathcal{A}, a \equiv c \mbox{ mod } p\}| \leq \Theta.$$ Given that for each $c \in    \Omega(p)$, set $\mathcal{A}$ occupies only $\Theta$ residue classes modulo $p^2$, this suggests that residue classes  $c \in \Omega(p)$ are underrepresented in $\mathcal{A}$ compared to residue classes  $c \not \in \Omega(p)$.  

\medskip

For instance, we might be interested in sets $\mathcal{A} \subseteq [1,N]\cap \mathbb{N}$ which satisfy the following property for some $\varepsilon \in (0,1)$:  For all $j \in \{1, \dots, k\}$ and all products $p_1 \dots p_j \leq  N^{1/2}$ of $j$ distinct primes,  
\begin{align*}
 \#\left\{n \in \mathcal{A}: ( n  \mbox{ \rm mod } p_i )\in \Omega(p_i)  \mbox{ for } 1 \leq i \leq j\right\}\leq \varepsilon\prod_{\substack{ i =1}}^j \left(\frac{ \omega(p_i) }{p_i} \right) |\mathcal{A}|,
\end{align*} 
where $\omega(p_i)=|\Omega(p_i)|$. 
This means that for any product $p_1\dots p_j$ of at most $k$ primes less than $N^{1/2}$, the residue classes corresponding to $\Omega(p_1) \times \dots \times \Omega(p_j)$  are underrepresented in $\mathcal{A}$. In this case the results of Section~\ref{ssec:denominator.manipulation} can be used to give a nice upper bound on $|\mathcal{A}|$, which is not much larger than the bound we would get if the residue classes $\Omega(p_i)$ were missing entirely rather than just rare.

\begin{lemma}\label{lem:uncommon.residue.classes}
Let $\mathcal{A} \subseteq [1,N] \cap \mathbb{N}$. Let $\mathcal{P}$ be a subset of the primes. Let $\varepsilon \in (0,1)$. 
Let $\Omega(p) \subseteq \mathbb{Z}/p\mathbb{Z}$.

Suppose that for all $j \in \{1, \dots, K\}$ and all  $p_1 \dots p_j \leq  N^{1/2}$ with   $p_1, \dots, p_j \in \mathcal{P}$ and $p_1, \dots, p_j$ distinct,
\begin{align}\label{inequ:uncommon.residue.classes0}
 \#\left\{n \in \mathcal{A}: ( n  \mbox{ \rm mod } p_i)\in \Omega(p_i)  \mbox{ for } 1 \leq i \leq j\right\}\leq \varepsilon\prod_{\substack{ i =1}}^j \left(\frac{ \omega(p_i) }{p_i} \right) |\mathcal{A}|.
\end{align}
Then the following  upper bound is satisfied: 
 \begin{align*}
 |\mathcal{A}| \leq  \frac{2N}{ 1 -2^{K+2}\varepsilon} \left(        \sum_{\substack{d \leq N^{1/2}   }}^* \mu(d)^2 \prod_{ p \mid d} \left(\frac{\omega(p)}{p-\omega(p)} \right) \right)^{-1},
    \end{align*}
    where the sum $\sum^*$ is restricted to integers $d$ which are products of at most $K$ primes in $\mathcal{P}$.
\end{lemma}

Comparing Lemma~\ref{lem:uncommon.residue.classes} with Montgomery's large sieve (see Lemma~\ref{lem:largesieve.montgomery}), we see that these upper bounds only differ  by a factor of $\frac{1}{1-2^{K+2}\varepsilon}$, but we have the additional restriction that   $d$ must have no more than $K$ prime factors.  

\medskip
In Section~\ref{sec:ostmannproofs}, Lemma~\ref{lem:uncommon.residue.classes} will be used to bound the size of $|\mathcal{A}_1 \cap [1,x]| |\mathcal{A}_2 \cap [1,x]| $ for $\mathcal{A}_1 + \mathcal{A}_2 = \mathcal{S}+\mathcal{S}$.

\begin{proof}[Proof of Lemma~\ref{lem:uncommon.residue.classes}]
Without loss of generality we assume that $\omega(p) \in \{1,\dots,p-1\}$ for all $p \in \mathcal{P}$. 

\medskip
 First we consider  a fixed product $p_1 \dots p_k$  with $p_1, \dots, p_k \in \mathcal{P}$ distinct and $k \leq K$.  
We begin by applying Corollary~\ref{cor:denominator.rearrangement}.  Taking
\begin{align*}
\Lambda(e_1,\dots,e_k) = \#\left\{n \in \mathcal{A}: \begin{array}{l} n  \mbox{ \rm mod } p_i\in \Omega(p_i)  \mbox{ when } e_i=0 \\  n \mbox{ \rm mod } p_i\not\in \Omega(p_i)   \mbox{ when } e_i=1 \end{array}\right\},
\end{align*}
 $t_i(0) = \omega(p_i)$, $t_i(1) = p_i - \omega(p_i)$ and  $M(\underline{e},\underline{u}) = \{i \in \{1,\dots,k\}: e_i \neq u_i\}$, we  have the following  bound: 
\begin{align} \label{inequ:unevenstart}
         &\sum_{\substack{a=1 \\ (a,p_1\dots p_k)=1}}^{p_1\dots p_k} \left|\sum_{\substack{n\leq N \\ n \in \mathcal{A}}} \exp\left(\frac{2 \pi i n a}{p_1\dots p_k} \right) \right|^2   \\  &\geq \sum_{I \subseteq \{1, \dots, k\}} (-1)^{\# I}
          \mathop{\sum \dots \sum}_{\substack{ e_i \in \{0,1\}\\ \mbox{ \rm \begin{footnotesize}
     for
     \end{footnotesize}} i \not\in I }}  \Bigg( \prod_{i \not \in I} \frac{p_i}{t_i(e_i)} \Bigg)\Bigg(     \mathop{\sum \dots \sum}_{\substack{ e_i \in \{0,1\}\\ \mbox{ \rm \begin{footnotesize}
     for
     \end{footnotesize}} i \in I }}\Lambda(e_1,\dots,e_k)^2  \Bigg) \nonumber  \\ 
      &+\sum_{I \subseteq \{1, \dots, k\}} (-1)^{\# I}
          \mathop{\sum \dots \sum}_{\substack{ e_i \in \{0,1\}\\ \mbox{ \rm \begin{footnotesize}
     for
     \end{footnotesize}} i \not\in I }} \Bigg( \prod_{i \not \in I} \frac{p_i}{t_i(e_i)} \Bigg) \Bigg(     \mathop{\sum \dots \sum}_{\substack{ e_i \in \{0,1\}\\ \mbox{ \rm \begin{footnotesize}
     for
     \end{footnotesize}} i \in I }}\mathop{\sum \dots \sum}_{\substack{ u_i \in \{0,1\} \mbox{\rm \begin{footnotesize}
     for
     \end{footnotesize}} i \in I \\ u_i = e_i  \mbox{\rm \begin{footnotesize}
     for
     \end{footnotesize}} i \not\in I \\ \underline{u} \neq \underline{e}}}\Lambda(\underline{e})\Lambda(\underline{u})   \Bigg) \nonumber \\
          &=     \mathop{\sum \dots \sum}_{\substack{ e_i \in \{0,1\} }}\Lambda(e_1,\dots,e_k)^2 \Bigg(  \sum_{I \subseteq \{1, \dots, k\}} (-1)^{\# I}
          \prod_{i   \not\in I} \Big(\frac{p_i}{t_i(e_i)} \Big) \Bigg) \nonumber \\
      &+ \mathop{\sum \dots \sum}_{\substack{ e_i,u_i \in \{0,1\}\\ \underline{u} \neq \underline{e}}}\Lambda(e_1,\dots,e_k)\Lambda(u_1,\dots,u_k) \Bigg(\sum_{\substack{I \subseteq \{1, \dots, k\} \\ i \in I \mbox{\rm \begin{footnotesize}
     when
     \end{footnotesize}} u_i \neq e_i  }} (-1)^{\# I}\prod_{i   \not\in I} \Big(\frac{p_i}{t_i(e_i)} \Big)   \Bigg) \nonumber \\
     &=     \mathop{\sum \dots \sum}_{\substack{ e_i \in \{0,1\} }}\Lambda(e_1,\dots,e_k)^2  \prod_{i =1}^k \left(\frac{p_i-t_i(e_i)}{t_i(e_i)} \right) \nonumber \\
      &+ \mathop{\sum \dots \sum}_{\substack{ e_i,u_i \in \{0,1\}\\ \underline{e}\neq \underline{u}}}(-1)^{|M(\underline{e},\underline{u})|}\Lambda(e_1,\dots,e_k)\Lambda(u_1,\dots,u_k)\!\!\!\! \prod_{\substack{i  \in \{1, \dots, k\} \\ i \not \in M(\underline{e},\underline{u})  }}  \!\!\left(\frac{p_i - t_i(e_i)}{t_i(e_i)}\right) \nonumber \\
      &\geq 
        \Lambda(1,\dots,1)^2  \prod_{i =1}^k \left(\frac{\omega(p_i)}{p_i-\omega(p_i)} \right)  - \mathop{\sum \dots \sum}_{\substack{ e_i,u_i \in \{0,1\}\\ \underline{e}\neq \underline{u}}}\Lambda(e_1,\dots,e_k)\Lambda(u_1,\dots,u_k)\!\!\!\! \prod_{\substack{i  \in \{1, \dots, k\} \\ i \not \in M(\underline{e},\underline{u})  }}  \!\!\left(\frac{p_i - t_i(e_i)}{t_i(e_i)}\right). \nonumber
    \end{align}

    Now recall that   $ \#\left\{n \in \mathcal{A}:  (n  \mbox{ \rm mod } p_i)\in \Omega(p_i)  \mbox{ when } i \in J\right\}\leq \varepsilon \prod_{\substack{ i \in J}} \left(\frac{\omega(p_i)}{p_i} \right)|\mathcal{A}|$  for   $J \subseteq \{1, \dots, k\}$. So  
    \begin{align} \label{inequ:sparse.upperbound1}
    \Lambda(e_1,\dots,e_k) \leq \varepsilon \prod_{\substack{ i \in \{1,\dots,k\} \\ e_i=0}} \left(\frac{\omega(p_i)}{p_i} \right)  |\mathcal{A}|
    \end{align} 
    for any $(e_1,\dots,e_k) \neq (1,\dots,1)$. We use (\ref{inequ:sparse.upperbound1}) to give upper and lower bounds for various terms in  
    \begin{align}\label{inequ:sparse.lowerbound1}  \Lambda(1,\dots,1)^2  \prod_{i =1}^k \left(\frac{\omega(p_i)}{p_i-\omega(p_i)} \right)  - \mathop{\sum \dots \sum}_{\substack{ e_i,u_i \in \{0,1\}\\ \underline{e}\neq \underline{u}}}\Lambda(e_1,\dots,e_k)\Lambda(u_1,\dots,u_k)\!\!\!\! \prod_{\substack{i  \in \{1, \dots, k\} \\ i \not \in M(\underline{e},\underline{u})  }}  \!\!\left(\frac{p_i - t_i(e_i)}{t_i(e_i)}\right). 
    \end{align}
    
 \medskip
\underline{Case 1}: Contribution of $\underline{e}, \underline{u} \in \{0,1\}^k\setminus \underline{1}$. By  (\ref{inequ:sparse.upperbound1}),   for any $(e_1,\dots, e_k), (u_1,\dots,u_k) \in \{0,1\}^k\setminus (1,\dots, 1)$, 
    \begin{align*}
    &\Lambda(e_1,\dots,e_k) \Lambda(u_1,\dots,u_k) \prod_{\substack{i  \in \{1, \dots, k\} \\ i \not \in M(\underline{e},\underline{u})  }}  \!\!\left(\frac{p_i - t_i(e_i)}{t_i(e_i)}\right) \\&\leq \varepsilon^2 \prod_{\substack{ i \in \{1,\dots,k\} \\ e_i=0}} \left(\frac{\omega(p_i)}{p_i} \right)\prod_{\substack{ i \in \{1,\dots,k\} \\ u_i=0}} \left(\frac{\omega(p_i)}{p_i} \right)  \prod_{\substack{i  \in \{1, \dots, k\} \\ e_i = u_i =0}}  \!\!\left(\frac{p_i - \omega(p_i)}{\omega(p_i)}\right)\prod_{\substack{i  \in \{1, \dots, k\} \\ e_i = u_i=1 }}  \!\!\left(\frac{\omega(p_i)}{p_i-\omega(p_i)}\right)|\mathcal{A}|^2
     \\&\leq \varepsilon^2 \prod_{\substack{ i \in \{1,\dots,k\} \\ \{e_i,u_i\}=\{0,1\}}} \left(\frac{\omega(p_i)}{p_i-\omega(p_i)} \right) \prod_{\substack{i  \in \{1, \dots, k\} \\ e_i = u_i =0}}  \!\!\left(\frac{\omega(p_i)}{p_i-\omega(p_i)} \right)\left(\frac{p_i - \omega(p_i)}{p_i}\right)^2\prod_{\substack{i  \in \{1, \dots, k\} \\ e_i = u_i=1 }}  \!\!\left(\frac{\omega(p_i)}{p_i-\omega(p_i)}\right)|\mathcal{A}|^2
      \\&\leq  \varepsilon^2 \prod_{\substack{i  \in \{1, \dots, k\} }}  \!\!\left(\frac{\omega(p_i)}{p_i-\omega(p_i)}\right)|\mathcal{A}|^2.
    \end{align*}
    Hence the contribution of $\underline{e}, \underline{u} \in \{0,1\}^k\setminus \underline{1}$ to (\ref{inequ:sparse.lowerbound1}) is bounded below by 
    \begin{align*}
    - \mathop{\sum \dots \sum}_{\substack{ e_i,u_i \in \{0,1\}\\ \underline{e}\neq \underline{u}\\ \underline{e}, \underline{u} \neq (1,\dots,1)}}\Lambda(e_1,\dots,e_k)\Lambda(u_1,\dots,u_k)\!\!\!\! \prod_{\substack{i  \in \{1, \dots, k\} \\ i \not \in M(\underline{e},\underline{u})  }}  \!\!\left(\frac{p_i - t_i(e_i)}{t_i(e_i)}\right) \geq - 4^k \varepsilon^2\prod_{\substack{i  \in \{1, \dots, k\} }}  \!\!\left(\frac{\omega(p_i)}{p_i-\omega(p_i)}\right)|\mathcal{A}|^2.
    \end{align*}

\medskip
\underline{Case 2}: Contribution of $\underline{e} \in \{0,1\}^k\setminus \underline{1}$ and $\underline{u}=1$. By  (\ref{inequ:sparse.upperbound1}),   for any $(e_1,\dots, e_k)\in \{0,1\}^k\setminus (1,\dots, 1)$, 
 \begin{align*}
    &\Lambda(e_1,\dots,e_k)\Lambda(1,\dots,1) \prod_{\substack{i  \in \{1, \dots, k\} \\ i \not \in M(\underline{e},\underline{1})  }}  \!\!\left(\frac{p_i - t_i(e_i)}{t_i(e_i)}\right)=
    \Lambda(e_1,\dots,e_k)\Lambda(1,\dots,1) \prod_{\substack{i  \in \{1, \dots, k\} \\ e_i=1 }}  \!\!\left(\frac{\omega(p_i)}{p_i-\omega(p_i)}\right) \\&\leq \varepsilon \prod_{\substack{ i \in \{1,\dots,k\} \\ e_i=0}} \left(\frac{\omega(p_i)}{p_i} \right) \prod_{\substack{i  \in \{1, \dots, k\} \\ e_i=1 }}  \!\!\left(\frac{\omega(p_i)}{p_i-\omega(p_i)}\right)|\mathcal{A}|^2
      \leq  \varepsilon \prod_{\substack{i  \in \{1, \dots, k\} }}  \!\!\left(\frac{\omega(p_i)}{p_i-\omega(p_i)}\right)|\mathcal{A}|^2.
    \end{align*}
  Hence the contribution of $\underline{e}, \underline{u} \in \{0,1\}^k$ with $\underline{e} \neq \underline{u}$ and $\underline{1} \in \{\underline{e}, \underline{u}\}$  to (\ref{inequ:sparse.lowerbound1}) is bounded below by    
     \begin{align*}
    - \mathop{\sum \dots \sum}_{\substack{ e_i,u_i \in \{0,1\}\\ \underline{e}\neq \underline{u}\\ \underline{1} \in \{\underline{e}, \underline{u}\}}}\Lambda(e_1,\dots,e_k)\Lambda(u_1,\dots,u_k)\!\!\!\! \prod_{\substack{i  \in \{1, \dots, k\} \\ i \not \in M(\underline{e},\underline{u})  }}  \!\!\left(\frac{p_i - t_i(e_i)}{t_i(e_i)}\right) \geq - 2^{k+1}\varepsilon \prod_{\substack{i  \in \{1, \dots, k\} }}  \!\!\left(\frac{\omega(p_i)}{p_i-\omega(p_i)}\right)|\mathcal{A}|^2.
    \end{align*} 

\medskip
\underline{Case 3}: Contribution of $\underline{e},\underline{u}=1$. Note that $\Lambda(1,\dots,1) = |\mathcal{A}|  - \sum_{ \underline{e} \neq \underline{1}} \Lambda(e_1,\dots,e_k)$. Hence by (\ref{inequ:sparse.upperbound1}), 
\begin{align*}
\Lambda(1,\dots,1) \geq  \Bigg( 1-\sum_{ \underline{e} \neq \underline{1}} \varepsilon \prod_{\substack{ i \in \{1,\dots,k\} \\ e_i=0}} \left(\frac{\omega(p_i)}{p_i} \right) \Bigg)  |\mathcal{A}| \geq  ( 1-2^k\varepsilon )  |\mathcal{A}| . 
\end{align*}
    Putting the three cases above together with (\ref{inequ:unevenstart}), we finally obtain
    \begin{align*}
     \sum_{\substack{a=1 \\ (a,p_1\dots p_k)=1}}^{p_1\dots p_k} \left|\sum_{\substack{n\leq N \\ n \in \mathcal{A}}} \exp\left(\frac{2 \pi i n a}{p_1\dots p_k} \right) \right|^2  &\geq \left( \left( 1-2^k\varepsilon \right)^2 -  4^k\varepsilon^2 -2^{k+1}\varepsilon \right)\prod_{i =1}^k \left(\frac{\omega(p_i)}{p_i-\omega(p_i)} \right) |\mathcal{A}|^2\\ 
     &=  \left( 1 -2^{k+2}\varepsilon \right)\prod_{i =1}^k \left(\frac{\omega(p_i)}{p_i-\omega(p_i)} \right) |\mathcal{A}| ^2
    \end{align*}
Using Linnik's large sieve inequality (Lemma~\ref{lem:largesieve.linnik}), summing over $p_1 \dots p_k$ with $k \leq K$, and observing $1 -2^{k+2}\varepsilon \geq 1 -2^{K+2}\varepsilon$, we then obtain the desired bound 
\begin{align*}
 |\mathcal{A}| \leq  \frac{2N}{ 1 -2^{K+2}\varepsilon} \left(        \sum_{\substack{d \leq N^{1/2}   }}^* \mu(d)^2 \prod_{ p \mid d} \left(\frac{\omega(p)}{p-\omega(p)} \right) \right)^{-1}. &\qedhere
    \end{align*}
\end{proof}

\subsection{Proof of Theorem~\ref{thm:largesieveimprovement}}\label{ssec:proof.prop1}

We  now discuss the proof of Theorem~\ref{thm:largesieveimprovement}. 
In Section~\ref{ssec:unusually.rare} we already mentioned that sets $\mathcal{A}$ treated by Theorem~\ref{thm:largesieveimprovement} typically have the property that residue classes  $c \in \Omega(p)$ are underrepresented in $\mathcal{A}$ compared to residue classes  $c \not \in \Omega(p)$. We also proved a nice sieve bound for a very strict type of underrepresentation of residue classes $\Omega(p)$. However, for the purpose of proving Theorem~\ref{thm:largesieveimprovement}, this criterion was a bit too strong (we cannot ensure that it is satisfied by all relevant $\mathcal{A}$). Instead of insisting that   residue classes in $\Omega(p_1) \times \dots \Omega(p_j)$ are rare modulo $p_1 \dots p_j$ for every $j \in \{1, \dots, k\}$, we only want to  place restrictions  on residue classes modulo $p_1 \dots p_k$ (products of exactly $k$ primes). \medskip

Theorem~\ref{thm:largesieveimprovement} is a simple combination of  Lemma~\ref{lem:proof.prop1.rare} and Lemma~\ref{lem:proof.prop1.common}, stated below.   

\begin{lemma}  \label{lem:proof.prop1.rare}
Let $\mathcal{A} \subseteq [1,N] \cap \mathbb{N}$. Let $\mathcal{P}$ be a  set of   primes. 

Let $\Omega(p) \subset \mathbb{Z}/p\mathbb{Z}$ with $\omega(p) = |\Omega(p)| \in \{1,\dots, p-1\}$.  

Let $c^*(k,\omega,\mathcal{P})$ and $d^*(k,\omega,\mathcal{P})$ be as given in Theorem~\ref{thm:largesieveimprovement} and assume $c^*(k,\omega,\mathcal{P}), d^*(k,\omega,\mathcal{P}) >0$. 
 Define 
 \begin{align*}
   \widetilde{ \varepsilon}(k,\omega,\mathcal{P}) &=    \min\left\{1, |\mathcal{P}|^{-k}\left( \sum_{p \in \mathcal{P}} \left( \frac{\omega(p )}{p -\omega(p )}\right) \right)^k  \right\}    \left( \frac{  c^*(k,\omega,\mathcal{P})^k d^*(k,\omega,\mathcal{P}     )^k } {  2^{4k+6} k   } \right).   
    \end{align*}

Suppose that for all distinct $p_1, \dots, p_k \in \mathcal{P}$,  we have $p_1 \dots p_k \leq N^{1/2}$ and 
\begin{align} \label{inequ:uncommon.residue.classes1}
 \#\left\{n \in \mathcal{A}:  n  \mbox{ \rm mod } p_i\in \Omega(p_i)  \mbox{ for } 1 \leq i \leq k\right\}\leq \widetilde{ \varepsilon}(k,\omega,\mathcal{P})\prod_{\substack{ i =1}}^k \left(\frac{\omega(p_i)}{p_i} \right) |\mathcal{A}|.
\end{align}    
Then the following  upper  bound is satisfied: 
\begin{align*}
|\mathcal{A}| \leq \frac{16N}{ c^*(k,\omega,\mathcal{P})^k}\left(\mathop{\sum \dots \sum}_{\substack{p_1, \dots, p_k \in \mathcal{P} \\ p_1> \dots > p_k }} \prod_{i =1}^k \left(\frac{\omega(p_i)}{p_i-\omega(p_i)} \right) \right)^{-1}.
\end{align*}
\end{lemma}

The proof of Lemma~\ref{lem:proof.prop1.rare} is postponed to Section~\ref{sssec:lemmaproof}.

\begin{lemma} \label{lem:proof.prop1.common}
Let $\mathcal{A} \subseteq [1,N] \cap \mathbb{N}$.  Let $\mathcal{P} $ be a set of primes,  let  $\Omega(p) \subseteq \mathbb{Z}/p\mathbb{Z}$ and $\omega(p)=|\Omega(p)|$.   

Suppose that for every $c \in    \Omega(p)$,   $|\{\overline{a} \in \mathbb{Z}/p^2\mathbb{Z}: a \in \mathcal{A}, a \equiv c \mbox{ mod } p\}| \leq \Theta$.  

Let $\varepsilon \in (0,1)$ and suppose that there exist distinct $p_1, \dots, p_k \in \mathcal{P}$ with $p_1 \dots p_k \leq N^{1/2}$ and 
\begin{align}  \label{inequ:rare.residue.fail}
 \#\left\{n \in \mathcal{A}: ( n  \mbox{ \rm mod } p_i)\in \Omega(p_i)  \mbox{ for } 1 \leq i \leq k\right\}> \varepsilon\prod_{\substack{ i = 1}}^k \left(\frac{\omega(p_i)}{p_i} \right) |\mathcal{A}|.
\end{align}
Then $|\mathcal{A}| \leq \frac{1}{\varepsilon}\left(\frac{2N \Theta^k}{p_1 \dots p_k}\right).$ In particular, if $\mathcal{P} \subseteq [\frac{1}{2}M,M]$ and $M^k \leq N^{1/2}$, we have
\begin{align*}
|\mathcal{A}| \leq \frac{2^{k+1}}{\varepsilon}\left(\frac{N \Theta^k}{M^k}\right).
\end{align*}
     \end{lemma}
     
     \begin{proof}
 Write $\mathcal{B} = \left\{n \in \mathcal{A}:  (n  \mbox{ \rm mod } p_i )\in \Omega(p_i)  \mbox{ for } 1 \leq i \leq k\right\}$.    Then recall that for every $c \in \Omega(p)$ and $p \in \mathcal{P}$, we have  $|\{\overline{a} \in \mathbb{Z}/p^2\mathbb{Z}: a \in \mathcal{A}, a \equiv c \mbox{ mod } p\}| \leq \Theta$. Hence $|\{\overline{b} \in \mathbb{Z}/p_i^2\mathbb{Z}: b \in \mathcal{B}\}| \leq \Theta \omega(p_i)$ for $1 \leq i \leq k$.  Thus $\mathcal{B} \subseteq [1,N]$ occupies at most $ \Theta^k \omega(p_1) \dots \omega(p_k)$ residue classes modulo $(p_1 \dots p_k)^2$, and  
    \begin{align*} 
|\mathcal{B}| \leq \Theta^k \omega(p_1) \dots \omega(p_k) \left( \frac{N  }{(p_1 \dots p_k)^2} +1\right) \leq \frac{2N \Theta^k \omega(p_1) \dots \omega(p_k)}{(p_1 \dots p_k)^2}.
    \end{align*}  
   On the other hand, we also assumed  
 $ |\mathcal{B}|> \varepsilon\prod_{\substack{ i =1}}^k \left(\frac{\omega(p_i)}{p_i} \right) |\mathcal{A}|.$
Combining these two inequalities, we get the desired bound  $$
|\mathcal{A}| \leq \frac{1}{\varepsilon}\left(\frac{2N  \Theta^k}{p_1 \dots p_k}\right).$$  In particular, if $\mathcal{P} \subseteq [\frac{1}{2}M,M]$, then 
$
|\mathcal{A}| \leq \frac{2^{k+1}}{\varepsilon}\left(\frac{N \Theta^k}{M^k}\right).
$
     \end{proof} 
\underline{Note}: In the proof of Lemma~\ref{lem:proof.prop1.common}, it was important that (\ref{inequ:rare.residue.fail}) holds for a long product $p_1 \dots p_k$ rather than some shorter product $p_1 \dots p_j$ with $j<k$, as the latter would give the much weaker bound $
|\mathcal{A}| \leq \frac{2^{k+1}}{\varepsilon}\left(\frac{N \Theta^k}{M^j}\right).
$ The need for a long product $p_1 \dots p_k$ is the reason we cannot combine  Lemma~\ref{lem:proof.prop1.common} with the simple Lemma~\ref{lem:uncommon.residue.classes}, and need to replace the  strict underrepresentation criterion  (\ref{inequ:uncommon.residue.classes0}) of Lemma~\ref{lem:uncommon.residue.classes} with the weaker condition (\ref{inequ:uncommon.residue.classes1}) given in  Lemma~\ref{lem:proof.prop1.rare}. \medskip

We now deduce  Theorem~\ref{thm:largesieveimprovement} from Lemma~\ref{lem:proof.prop1.rare}
 and Lemma~\ref{lem:proof.prop1.common}:

 \begin{proof}[Proof of Theorem~\ref{thm:largesieveimprovement} $($assuming Lemma~\ref{lem:proof.prop1.rare}$)$] Suppose  $\mathcal{A} \subseteq [1,N] \cap \mathbb{N}$, $\mathcal{P} \subseteq \{p \in [\frac{1}{2}M,M]: p \mbox{ prime}\}$ with $M^k \leq N^{1/2}$ and  $\Omega(p) \subset \mathbb{Z}/p\mathbb{Z}$  with    $|\{\overline{a} \in \mathbb{Z}/p^2\mathbb{Z}: a \in \mathcal{A}, a \equiv c \mbox{ mod } p\}| \leq \Theta$ for  $c \in    \Omega(p)$ and $p \in \mathcal{P}$. Then $\mathcal{A}$ must satisfy the conditions of either Lemma~\ref{lem:proof.prop1.rare}  or Lemma~\ref{lem:proof.prop1.common} with $\varepsilon= \widetilde{ \varepsilon}(k,\omega,\mathcal{P})$. If the conditions of Lemma~\ref{lem:proof.prop1.rare} are satisfied, then 
 \begin{align*} 
|\mathcal{A}| \leq \frac{16N}{ c^*(k,\omega,\mathcal{P})^k}\left(\mathop{\sum \dots \sum}_{\substack{p_1, \dots, p_k \in \mathcal{P} \\ p_1> \dots > p_k }} \prod_{i =1}^k \left(\frac{\omega(p_i)}{p_i-\omega(p_i)} \right) \right)^{-1}.
\end{align*}
On the other hand, if the conditions of Lemma~\ref{lem:proof.prop1.common} are satisfied with $\varepsilon= \widetilde{ \varepsilon}(k,\omega,\mathcal{P})$ , then 
\begin{align*} \nonumber
|\mathcal{A}| &\leq \frac{2^{k+1}}{ \widetilde{ \varepsilon}(k,\omega,\mathcal{P})}\left(\frac{N \Theta^k}{M^k}\right)  \\
&=\frac{  2^{5k+7}\Theta^k k  N  }{c^*(k,\omega,\mathcal{P})^k d^*(k,\omega,\mathcal{P})^k}  \max\left\{ \left(\frac{\left( \sum_{p \in \mathcal{P}} \left( \frac{\omega(p )}{p -\omega(p )}\right) \right) }{M } \right)^k , \left( \frac{|\mathcal{P}| }{M} \right)^k \right\} \left( \sum_{p \in \mathcal{P}} \left( \frac{\omega(p )}{p -\omega(p )}\right) \right)^{-k}  \nonumber \\
&\leq 128N \left( \frac{  32 \Theta     }{c^*  d^* } \right)^k \max\left\{ \left(\frac{\left( \sum_{p \in \mathcal{P}} \left( \frac{\omega(p )}{p -\omega(p )}\right) \right) }{M } \right)^k , \left( \frac{|\mathcal{P}| }{M} \right)^k \right\} \left(\mathop{\sum \dots \sum}_{\substack{p_1, \dots, p_k \in \mathcal{P} \\ p_1> \dots > p_k }} \prod_{i =1}^k \left(\frac{\omega(p_i)}{p_i-\omega(p_i)} \right) \right)^{-1}.
\end{align*}  
Recall now that we defined the quantity 
 \begin{align*}
 \xi(k,\omega,\mathcal{P},\Theta)=64\left(\frac{ 32 \Theta }{c^*(k,\omega,\mathcal{P})  d^*(k,\omega,\mathcal{P}) } \right)^k\max\left\{ \left(\frac{   d^*(k,\omega,\mathcal{P}) }{ 32 \Theta }  \right)^k,   \left(\frac{|\mathcal{P}|}{ M }\right)^k   ,  \left(  \frac{ \sum_{p \in \mathcal{P}} \left( \frac{\omega(p )}{p -\omega(p )}\right) }{M } \right)^k \right\}.
\end{align*}
Thus we see that Lemma~\ref{lem:proof.prop1.rare}   and Lemma~\ref{lem:proof.prop1.common} imply
\begin{align*}
&|\mathcal{A}| \leq  \xi(k,\omega,\mathcal{P},\Theta) \, 2N \left(  \mathop{\sum \dots \sum}_{\substack{ p_1, \dots, p_k \in \mathcal{P} \\ p_1> \dots > p_k }}  \prod_{i=1}^k \left( \frac{\omega(p_i)}{p_i-\omega(p_i)}\right) \right)^{-1}.
\end{align*}
This is what we wanted to prove.
 \end{proof}

 \subsection{Proof of Lemma~\ref{lem:proof.prop1.rare}}\label{sssec:lemmaproof}

In this section,  we prove Lemma~\ref{lem:proof.prop1.rare}.  This lemma involves condition (\ref{inequ:uncommon.residue.classes1}), 
\begin{align*} 
 \#\left\{n \in \mathcal{A}:  (n  \mbox{ \rm mod } p_i)\in \Omega(p_i)  \mbox{ for } 1 \leq i \leq k \right\}\leq \widetilde{ \varepsilon}(k,\omega,\mathcal{P})\prod_{\substack{ i =1}}^k \left(\frac{\omega(p_i)}{p_i} \right) |\mathcal{A}|.
\end{align*} 
We first describe how (\ref{inequ:uncommon.residue.classes1}) can be used to  bound   the number of $n \in \mathcal{A}$ for which there are many $p \in \mathcal{P}$ with $(n  \mbox{ \rm mod } p)\in \Omega(p)$. This is a key ingredient of the proof of Lemma~\ref{lem:proof.prop1.rare}.

\begin{lemma}\label{lem:lem10proof.inter}
Let $\mathcal{A} \subseteq [1,N] \cap \mathbb{N}$. Let $\mathcal{P}$ be a set of primes, let   $\Omega(p) \subseteq  \mathbb{Z}/p\mathbb{Z}$ and $\omega(p)=|\Omega(p)|$ with $0 < \omega(p) <p$.  
 Let $\varepsilon, \delta \in (0,1)$. Suppose that for all distinct $p_1, \dots, p_k \in \mathcal{P}$,  we have $p_1 \dots p_k \leq N^{1/2}$ and 
\begin{align} \label{inequ:uncommon.residue.classes2}
  \#\left\{n \in \mathcal{A}:  (n  \mbox{ \rm mod } p_i)\in \Omega(p_i)  \mbox{ for } 1 \leq i \leq k\right\}\leq \varepsilon \prod_{\substack{ i =1}}^k \left(\frac{\omega(p_i)}{p_i} \right) |\mathcal{A}|.
\end{align}    

 For any $n_1, n_2 \in \mathcal{A}$, we define $Z(n_1,n_2)=\{p \in \mathcal{P}: (n_1 \mbox{ mod } p)  \in \Omega(p), (n_2 \mbox{ mod } p) \in \Omega(p)\}$ and set $Z(n)=Z(n,n)$ and $z(n)=|Z(n)|$. We further define the following useful quantities:
 \begin{align*}
 &S(\omega, \mathcal{P})=  \sum_{p \in \mathcal{P}} \left( \frac{\omega(p)}{p-\omega(p)} \right) \,\, \mbox{ and } 
 \,\,
 M_1(\omega, \mathcal{P}) = \max_{p \in \mathcal{P}} \left( \frac{p-\omega(p)}{\omega(p)}\right)
\,\, \mbox{ and } \,\, 
 M_2(\omega, \mathcal{P}) = \max_{p \in \mathcal{P}} \left( \frac{\omega(p)}{p-\omega(p)}\right).
 \end{align*}
Assume $\delta S(\omega, \mathcal{P}) > \max\{k, kM_1(\omega,\mathcal{P}),kM_2(\omega,\mathcal{P})\}$.  We then have the following four inequalities: 
 \begin{align} \label{inequ:useful1}
 &\#\left\{n \in \mathcal{A}: z(n) >  \delta \sum_{p \in \mathcal{P}} \left( \frac{\omega(p)}{p-\omega(p)} \right)\right\} \leq \frac{k! \varepsilon }{(\delta  S(\omega, \mathcal{P}) -k )^k }  \mathop{\sum \dots \sum}_{\substack{ p_1, \dots, p_k\in \mathcal{P} \\ p_1> \dots > p_k }}     \prod_{\substack{ i =1}}^k \left(\frac{\omega(p_i)}{p_i} \right) |\mathcal{A}|, \\  \label{inequ:useful2}
 &\#\left\{(n_1,n_2) \in \mathcal{A} \times \mathcal{A}: \sum_{p \in Z(n_1,n_2)}   \left( \frac{p-\omega(p)} {\omega(p)}\right)  >  \delta   \sum_{p \in \mathcal{P}} \left( \frac{\omega(p)}{p-\omega(p)} \right)\right\}  \\ &\leq  \frac{k! \varepsilon}{ (\delta  S(\omega, \mathcal{P}) - k M_1(\omega, \mathcal{P}))^k}  \mathop{\sum \dots \sum}_{\substack{ p_1, \dots, p_k\in \mathcal{P} \\ p_1> \dots > p_k }}     \prod_{\substack{ i =1}}^k \left(\frac{\omega(p_i)}{p_i} \right)  |\mathcal{A}|^2, \nonumber  \\  \label{inequ:useful3}
 &\#\left\{n \in \mathcal{A}: \sum_{p \in Z(n)}   \left( \frac {\omega(p)}{p-\omega(p)}\right)  >  \delta  S(\omega, \mathcal{P})\right\}  \leq  \frac{\varepsilon    S(\omega,\mathcal{P})^k   }{ (\delta  S(\omega, \mathcal{P}) - k M_2(\omega, \mathcal{P}))^k}    |\mathcal{A}|,
\end{align}
\begin{align} 
  \label{inequ:useful4}
  &\sum_{\substack{(n_1,n_2) \in \mathcal{A} \times \mathcal{A} }}    \mathop{\sum \dots \sum}_{\substack{ p_1, \dots, p_k\in Z(n_1,n_2) \\ p_1> \dots > p_k }} \prod_{i=1}^k  \left( \frac{p_i-\omega(p_i)} {\omega(p_i)}\right)   \leq  \varepsilon   \mathop{\sum \dots \sum}_{\substack{ p_1, \dots, p_k\in \mathcal{P} \\ p_1> \dots > p_k }}     \prod_{\substack{ i =1}}^k \left(\frac{\omega(p_i)}{p_i} \right)  |\mathcal{A}|^2.
\end{align} 
\end{lemma}

\begin{proof} We begin by recalling the inequality (\ref{inequ:uncommon.residue.classes2}), given below:
\begin{align*}  
  \#\left\{n \in \mathcal{A}:  (n  \mbox{ \rm mod } p_i)\in \Omega(p_i)  \mbox{ for } 1 \leq i \leq k\right\}\leq \varepsilon \prod_{\substack{ i =1}}^k \left(\frac{\omega(p_i)}{p_i} \right) |\mathcal{A}|.
\end{align*}

 \underline{Proof of (\ref{inequ:useful1})}: 
Changing the order of summation and using 
 condition (\ref{inequ:uncommon.residue.classes2}), we first note
\begin{align*}
 \sum_{n \in \mathcal{A}}   \binom{z(n)}{k}   &=\mathop{\sum \dots \sum}_{\substack{ p_1, \dots, p_k\in \mathcal{P} \\ p_1> \dots > p_k }}      \#\left\{n \in \mathcal{A}:  (n  \mbox{ \rm mod } p_i)\in \Omega(p_i)  \mbox{ for } 1 \leq i \leq k\right\}  \\
 &\leq       \varepsilon  \mathop{\sum \dots \sum}_{\substack{ p_1, \dots, p_k\in \mathcal{P} \\ p_1> \dots > p_k }}     \prod_{\substack{ i =1}}^k \left(\frac{\omega(p_i)}{p_i} \right) |\mathcal{A}|.
\end{align*} 
We then observe the following:
\begin{align*}
 \#\left\{n \in \mathcal{A}: z(n) >  \delta \sum_{p \in \mathcal{P}} \left( \frac{\omega(p)}{p-\omega(p)} \right)\right\}  &= \sum_{\substack{ n \in \mathcal{A} \\ z(n) > \delta S(\omega, \mathcal{P})}}   \frac{k! }{ z(n) (z(n)-1) \dots (z(n)-k+1)}   \binom{z(n)}{k} 
 \\
  &\leq  \frac{k! }{(\delta  S(\omega, \mathcal{P}) -k )^k }   \sum_{\substack{ n \in \mathcal{A} }}   \binom{z(n)}{k}  \\
  &\leq  \frac{k! \varepsilon }{(\delta  S(\omega, \mathcal{P}) -k )^k }  \mathop{\sum \dots \sum}_{\substack{ p_1, \dots, p_k\in \mathcal{P} \\ p_1> \dots > p_k }}     \prod_{\substack{ i =1}}^k \left(\frac{\omega(p_i)}{p_i} \right) |\mathcal{A}|.
\end{align*}
 \underline{Proof of (\ref{inequ:useful2}) and (\ref{inequ:useful4})}: To prove the second and fourth inequality, we  observe
\begin{align*}
 &\sum_{\substack{(n_1,n_2) \in \mathcal{A} \times \mathcal{A} }}    \mathop{\sum \dots \sum}_{\substack{ p_1, \dots, p_k\in Z(n_1,n_2) \\ p_1> \dots > p_k }} \prod_{i=1}^k  \left( \frac{p_i-\omega(p_i)} {\omega(p_i)}\right) \\  &=\mathop{\sum \dots \sum}_{\substack{ p_1, \dots, p_k\in \mathcal{P} \\ p_1> \dots > p_k }}   \prod_{i=1}^k  \left( \frac{p_i-\omega(p_i)} {\omega(p_i)}\right)  \#\left\{n \in \mathcal{A}:  (n  \mbox{ \rm mod } p_i)\in \Omega(p_i)  \mbox{ for } 1 \leq i \leq k\right\}^2    \\
 &\leq  \varepsilon^2  \mathop{\sum \dots \sum}_{\substack{ p_1, \dots, p_k\in \mathcal{P} \\ p_1> \dots > p_k }}   \prod_{i=1}^k  \left( \frac{p_i-\omega(p_i)} {\omega(p_i)}\right)         \prod_{\substack{ i =1}}^k \left(\frac{\omega(p_i)}{p_i} \right)^2 |\mathcal{A}|^2  \leq  \varepsilon   \mathop{\sum \dots \sum}_{\substack{ p_1, \dots, p_k\in \mathcal{P} \\ p_1> \dots > p_k }}     \prod_{\substack{ i =1}}^k \left(\frac{\omega(p_i)}{p_i} \right)  |\mathcal{A}|^2.
\end{align*} 
Recalling $M_1(\omega, \mathcal{P}) = \max_{p \in \mathcal{P}} ( \frac{p-\omega(p)}{\omega(p)}) $, we then observe the following:
\begin{align*}
 &\#\left\{(n_1,n_2) \in \mathcal{A} \times \mathcal{A}: \sum_{p \in Z(n_1,n_2)}   \left( \frac{p-\omega(p)} {\omega(p)}\right)  >  \delta  S(\omega, \mathcal{P})\right\} \\ 
 & \leq  \sum_{\substack{(n_1,n_2) \in \mathcal{A} \times \mathcal{A} }}   \frac{k!}{ (\delta  S(\omega, \mathcal{P}) - k M_1(\omega, \mathcal{P}))^k} \mathop{\sum \dots \sum}_{\substack{ p_1, \dots, p_k\in Z(n_1,n_2) \\ p_1> \dots > p_k }} \prod_{i=1}^k  \left( \frac{p_i-\omega(p_i)} {\omega(p_i)}\right)\\ 
 & \leq  \frac{k!}{ (\delta  S(\omega, \mathcal{P}) - k M_1(\omega, \mathcal{P}))^k} \sum_{\substack{(n_1,n_2) \in \mathcal{A} \times \mathcal{A} }}   \mathop{\sum \dots \sum}_{\substack{ p_1, \dots, p_k\in Z(n_1,n_2) \\ p_1> \dots > p_k }} \prod_{i=1}^k  \left( \frac{p_i-\omega(p_i)} {\omega(p_i)}\right) \\
  & \leq  \frac{k! \varepsilon}{ (\delta  S(\omega, \mathcal{P}) - k M_1(\omega, \mathcal{P}))^k}  \mathop{\sum \dots \sum}_{\substack{ p_1, \dots, p_k\in \mathcal{P} \\ p_1> \dots > p_k }}     \prod_{\substack{ i =1}}^k \left(\frac{\omega(p_i)}{p_i} \right)  |\mathcal{A}|^2.
\end{align*} 
 \underline{Proof of (\ref{inequ:useful3})}: Finally, we observe that 
\begin{align*}
 &\sum_{\substack{n \in \mathcal{A}  }}    \mathop{\sum \dots \sum}_{\substack{ p_1, \dots, p_k\in Z(n) \\ p_1> \dots > p_k }} \prod_{i=1}^k  \left( \frac{\omega(p_i)}{p_i-\omega(p_i)} \right) \\ &=\mathop{\sum \dots \sum}_{\substack{ p_1, \dots, p_k\in \mathcal{P} \\ p_1> \dots > p_k }}   \prod_{i=1}^k  \left( \frac{\omega(p_i)}{p_i-\omega(p_i)} \right)  \#\left\{n \in \mathcal{A}:  (n  \mbox{ \rm mod } p_i)\in \Omega(p_i)  \mbox{ for } 1 \leq i \leq k\right\}     \\
 &\leq  \varepsilon   \mathop{\sum \dots \sum}_{\substack{ p_1, \dots, p_k\in \mathcal{P} \\ p_1> \dots > p_k }}   \prod_{i=1}^k  \left( \frac{\omega(p_i)}{p_i-\omega(p_i)} \right)          |\mathcal{A}| 
\leq \frac{ \varepsilon }{k!}   S(\omega,\mathcal{P})^k         |\mathcal{A}|.
\end{align*} 
Recalling $M_2(\omega, \mathcal{P}) = \max_{p \in \mathcal{P}} ( \frac{\omega(p)}{p-\omega(p)}) $, we then observe the following:
\begin{align*}
 &\#\left\{n \in \mathcal{A}: \sum_{p \in Z(n)}   \left( \frac {\omega(p)}{p-\omega(p)}\right)  >  \delta  S(\omega, \mathcal{P})\right\} \\ 
 & \leq  \sum_{\substack{n \in \mathcal{A}  }}   \frac{k!}{ (\delta  S(\omega, \mathcal{P}) - k M_2(\omega, \mathcal{P}))^k} \mathop{\sum \dots \sum}_{\substack{ p_1, \dots, p_k\in Z(n) \\ p_1> \dots > p_k }} \prod_{i=1}^k  \left( \frac {\omega(p_i)}{p_i-\omega(p_i)}\right)\\  
  & \leq  \frac{\varepsilon    S(\omega,\mathcal{P})^k   }{ (\delta  S(\omega, \mathcal{P}) - k M_2(\omega, \mathcal{P}))^k}    |\mathcal{A}|.
\end{align*} 
We have shown that all four inequalities hold.
\end{proof}

Using Lemma~\ref{lem:lem10proof.inter}, we now complete the proof of Lemma~\ref{lem:proof.prop1.rare}.  

\begin{proof}[Proof of Lemma~\ref{lem:proof.prop1.rare}]
First we consider  a fixed product $p_1 \dots p_k$  with distinct  $p_1, \dots, p_k \in \mathcal{P}$ and write \begin{align*}
\Lambda(e_1,\dots,e_k) = \#\left\{n \in \mathcal{A}: \begin{array}{l} n  \mbox{ \rm mod } p_i\in \Omega(p_i)  \mbox{ when } e_i=0 \\  n \mbox{ \rm mod } p_i\not\in \Omega(p_i)   \mbox{ when } e_i=1 \end{array}\right\},
\end{align*}
and $t_i(0) = \omega(p_i)$, $t_i(1) = p_i - \omega(p_i)$ and  $M(\underline{e},\underline{u}) = \{i \in \{1,\dots,k\}: e_i \neq u_i\}$. As in the proof of Lemma~\ref{lem:uncommon.residue.classes}, we begin by recalling Corollary~\ref{cor:denominator.rearrangement}, which implies the following:
 \begin{align}
\nonumber   \sum_{\substack{a=1 \\ (a,p_1\dots p_k)=1}}^{p_1\dots p_k} \left|\sum_{\substack{n\leq N \\ n \in \mathcal{A}}} \exp\left(\frac{2 \pi i n a}{p_1\dots p_k} \right) \right|^2   &\geq
       \sum_{I \subseteq \{1, \dots, k\}} (-1)^{\# I}
          \mathop{\sum \dots \sum}_{\substack{ e_i \in \{0,1\}\\ \mbox{ \rm \begin{footnotesize}
     for
     \end{footnotesize}} i \not\in I }} \prod_{i   \not\in I} \Big(\frac{p_i}{t_i(e_i)} \Big)  \Bigg(     \mathop{\sum \dots \sum}_{\substack{ e_i \in \{0,1\}\\ \mbox{ \rm \begin{footnotesize}
     for
     \end{footnotesize}} i \in I }} \Lambda(e_1,\dots,e_k)   \Bigg)^2 \\  &=     \mathop{\sum \dots \sum}_{\substack{ e_i,u_i \in \{0,1\} }}\Lambda(\underline{e})\Lambda(\underline{u}) \Bigg(\sum_{\substack{I \subseteq \{1, \dots, k\} \\ i \in I \mbox{\rm \begin{footnotesize}
     when
     \end{footnotesize}} u_i \neq e_i  }} \!\!\!\!\!\! (-1)^{\# I}\prod_{i   \not\in I} \Big(\frac{p_i}{t_i(e_i)} \Big)   \Bigg) \nonumber\\ 
     &=    \mathop{\sum \dots \sum}_{\substack{ e_i,u_i \in \{0,1\} }}\Lambda(\underline{e})\Lambda(\underline{u})(-1)^{|M(\underline{e},\underline{u})|}\!\!\!\! \prod_{\substack{i  \in \{1, \dots, k\} \\ i \not \in M(\underline{e},\underline{u})  }}  \!\!\left(\frac{p_i - t_i(e_i)}{t_i(e_i)}\right).   \label{inequ:cor2}
    \end{align} 
    
Given a pair $(n_1, n_2) \in \mathcal{A} \times \mathcal{A}$, we then introduce the following notation: 
\begin{align*}
X(n_1,n_2)&=\{p \in \mathcal{P}: (n_1 \mbox{ mod } p)  \not\in \Omega(p), (n_2 \mbox{ mod } p) \not\in \Omega(p)\}, \\ 
Y_1(n_1,n_2)&=\{p \in \mathcal{P}: (n_1 \mbox{ mod } p)  \not\in \Omega(p), (n_2 \mbox{ mod } p) \in \Omega(p)\}, \\
Y_2(n_1,n_2)&=\{p \in \mathcal{P}: (n_1 \mbox{ mod } p)  \in \Omega(p), (n_2 \mbox{ mod } p) \not\in \Omega(p)\}, \\
Z(n_1,n_2)&=\{p \in \mathcal{P}: (n_1 \mbox{ mod } p)  \in \Omega(p), (n_2 \mbox{ mod } p) \in \Omega(p)\}.
\end{align*}
We further set $Z(n)=Z(n,n)$    and $z(n)=|Z(n)|$. 

\medskip
Summing (\ref{inequ:cor2}) over $p_1, \dots, p_k \in \mathcal{P}$ with $p_1>\dots>p_k$ and carefully considering the contribution of any given  $(n_1, n_2) \in \mathcal{A} \times \mathcal{A}$, we then get the following:  
 \begin{align}\label{inequ:prooflem1}
        &\mathop{\sum \dots \sum}_{\substack{ p_1, \dots, p_k \in \mathcal{P} \\ p_1> \dots > p_k }}     \sum_{\substack{a=1 \\ (a,p_1\dots p_k)=1}}^{p_1\dots p_k} \left|\sum_{\substack{n\leq N \\ n \in \mathcal{A}}} \exp\left(\frac{2 \pi i n a}{p_1\dots p_k} \right) \right|^2      
       \\
         &\geq  \sum_{n_1 \in \mathcal{A}} \sum_{n_2 \in \mathcal{A}} \mathop{\sum \dots \sum}_{\substack{ p_1, \dots, p_k \in \mathcal{P} \\ p_1> \dots > p_k }}     \prod_{\substack{i  \in \{1, \dots, k\} \\ p_i \in X(n_1,n_2) }}  \!\!\left(\frac{\omega(p_i)}{p_i - \omega(p_i)}\right)  
       \!\!\!\!\!\!\!\!\!\!  \prod_{\substack{i  \in \{1, \dots, k\} \\ p_i \in Y_1(n_1,n_2) \cup Y_2(n_1,n_2) }}  \!\!\!\!\!\!\!\!\left(-1\right) 
          \prod_{\substack{i  \in \{1, \dots, k\} \\ p_i \in Z(n_1,n_2) }}  \!\!\left(\frac{p_i - \omega(p_i)}{\omega(p_i)}\right) \nonumber \\
             &\geq  \sum_{n_1 \in \mathcal{A}} \sum_{n_2 \in \mathcal{A}} \mathop{\sum \dots \sum}_{\substack{ p_1, \dots, p_k \in X(n_1,n_2) \\ p_1> \dots > p_k }}     \prod_{\substack{i =1}}^k   \left(\frac{\omega(p_i)}{p_i - \omega(p_i)}\right)   
          - \frac{1}{(k-1)!}
            \sum_{n_1 \in \mathcal{A}} \sum_{n_2 \in \mathcal{A}} \mathop{\sum \dots \sum}_{\substack{ p_1, \dots, p_k \in \mathcal{P}   \\ p_1 \not\in X(n_1,n_2)  }}  \prod_{i=1}^k   \zeta(p_i;n_1,n_2) \nonumber \\
            &\mbox{where } \zeta(p_i;n_1,n_2) = \begin{cases} \left(\frac{\omega(p)}{p - \omega(p)}\right) \quad \quad  \mbox{ if } p \in X(n_1,n_2) \nonumber \\
            \left(\frac{p - \omega(p)}{\omega(p)}\right) 
          \quad\quad  \mbox{ if }   p \in Z(n_1,n_2)    \\ 
            1   \quad\quad\quad\quad\quad \,\,\,\, \mbox{ if } p \in Y_1(n_1,n_2) \cup  Y_2(n_1,n_2). \nonumber  \end{cases}
    \end{align}   
\underline{Step 1}: Lower bounds for the positive contribution of $X(n_1,n_2)$ to (\ref{inequ:prooflem1}).   

 We first use inequality (\ref{inequ:useful3}) of  Lemma~\ref{lem:lem10proof.inter}  and  recall that for 
 $M_2(\omega, \mathcal{P}) = \max_{p \in \mathcal{P}} ( \frac{\omega(p)}{p-\omega(p)}) $,
\begin{align*}
 &\#\left\{n \in \mathcal{A}: \sum_{p \in Z(n)}   \left( \frac {\omega(p)}{p-\omega(p)}\right)  >  \delta_1  S(\omega, \mathcal{P})\right\}  \leq  \frac{\widetilde{\varepsilon}(k,\omega,\mathcal{P})    S(\omega,\mathcal{P})^k   }{ (\delta_1  S(\omega, \mathcal{P}) - k M_2(\omega, \mathcal{P}))^k}    |\mathcal{A}|.
\end{align*} 
Since $\mathcal{P} \setminus X(n_1,n_2) \subseteq Z(n_1) \cup Z(n_2)$, we then find that
\begin{align*}
 &\sum_{n_1 \in \mathcal{A}} \sum_{n_2 \in \mathcal{A}} \mathop{\sum \dots \sum}_{\substack{ p_1, \dots, p_k \in X(n_1,n_2) \\ p_1> \dots > p_k }}     \prod_{\substack{i =1}}^k   \left(\frac{\omega(p_i)}{p_i - \omega(p_i)}\right)   \\
 &\geq \frac{1}{k!} \sum_{n_1 \in \mathcal{A}} \sum_{n_2 \in \mathcal{A}}   \left( \sum_{p \in \mathcal{P}} \left(\frac{\omega(p)}{p - \omega(p)}\right)- \sum_{p \in Z(n_1) \cup Z(n_2)} \left(\frac{\omega(p)}{p - \omega(p)}\right) - kM_2(\omega, \mathcal{P})  \right)^k
 \\
 &\geq   \frac{1}{k!}     \left(   (1- 2\delta_1) S(\omega, \mathcal{P})- kM_2(\omega, \mathcal{P})  \right)^k  \left( 1-\frac{ \widetilde{\varepsilon}(k,\omega,\mathcal{P}) S(\omega,\mathcal{P})^k   }{ (\delta_1  S(\omega, \mathcal{P}) - k M_2(\omega, \mathcal{P}))^k}  \right)^2  |\mathcal{A}|^2.
\end{align*} 
We now choose $\delta_1 =\frac{1}{3}$ and recall definition
\begin{align*}
c^*(k,\omega, \mathcal{P}) = \left( \sum_{p \in \mathcal{P}} \left(\frac{\omega(p)}{p-\omega(p)} \right) \right)^{-1} \left\{\frac{1}{3} \sum_{p \in \mathcal{P}} \left(\frac{\omega(p)}{p-\omega(p)} \right) - k \max_{p \in \mathcal{P}}  \left(\frac{\omega(p)}{p-\omega(p)} \right) \right\},
\end{align*}
which implies that 
\begin{align*}
 &\sum_{n_1 \in \mathcal{A}} \sum_{n_2 \in \mathcal{A}} \mathop{\sum \dots \sum}_{\substack{ p_1, \dots, p_k \in X(n_1,n_2) \\ p_1> \dots > p_k }}     \prod_{\substack{i =1}}^k   \left(\frac{\omega(p_i)}{p_i - \omega(p_i)}\right)   
 \geq   \frac{  c^*(k,\omega,\mathcal{P})^k}{k!}   \left( 1-\frac{ \widetilde{\varepsilon}(k,\omega,\mathcal{P})      }{  c^*(k,\omega,\mathcal{P})^k  }  \right)^2    S(\omega,\mathcal{P})^k  |\mathcal{A}|^2.
\end{align*} 
Recalling that $\widetilde{\varepsilon}(k,\omega,\mathcal{P}) \leq \frac{c^*(k,\omega,\mathcal{P})^k}{2}$ and looking at (\ref{inequ:prooflem1}),  we have ensured that  
\begin{align} \label{inequ:prooflem2}
&\mathop{\sum \dots \sum}_{\substack{ p_1, \dots, p_k \in \mathcal{P} \\ p_1> \dots > p_k }}     \sum_{\substack{a=1 \\ (a,p_1\dots p_k)=1}}^{p_1\dots p_k} \left|\sum_{\substack{n\leq N \\ n \in \mathcal{A}}} \exp\left(\frac{2 \pi i n a}{p_1\dots p_k} \right) \right|^2 \\ &\geq  
\frac{ c^*(k,\omega,\mathcal{P})^k}{ 4k!}  S(\omega,\mathcal{P})^k |\mathcal{A}|^2  
          - \frac{1}{(k-1)!}
            \sum_{n_1 \in \mathcal{A}} \sum_{n_2 \in \mathcal{A}} \mathop{\sum \dots \sum}_{\substack{ p_1, \dots, p_k \in \mathcal{P}   \\ p_1 \not\in X(n_1,n_2)  }}  \prod_{i=1}^k   \zeta(p_i;n_1,n_2). \nonumber
\end{align}

To bound the negative terms, we now separately consider small and large values of $z(n_1)$ and $z(n_2)$. 

\bigskip
\bigskip
\bigskip
  
\underline{Step 2}: Upper bounds for the negative contribution of small $z(n_i)$ to (\ref{inequ:prooflem2}). 

For a given $\delta$, we now consider the set $$\mathcal{C}(\delta)= \left\{ (n_1,n_2) \in \mathcal{A} \times \mathcal{A} :  z(n_1),z(n_2)\leq  \delta S(\omega, \mathcal{P}) \mbox{ and } \sum_{p \in Z(n_1,n_2)} \left(\frac{p-\omega(p)}{\omega(p)} \right)  \leq  \delta S(\omega, \mathcal{P}) \right\}.$$
We have that for any $(n_1,n_2) \in \mathcal{C}(\delta)$, 
\begin{align*} 
            &\mathop{\sum }_{\substack{  p \in \mathcal{P} \\
             p  \not\in X(n_1,n_2)  } }    \zeta(p ;n_1,n_2) \leq  z(n_1) + z(n_2)+ \sum_{p \in Z(n_1,n_2)} \left(\frac{p-\omega(p)}{\omega(p)} \right)   \leq 3  \delta S(\omega, \mathcal{P}),\\
             &\mathop{\sum }_{\substack{  p \in \mathcal{P}   } }    \zeta(p ;n_1,n_2) \leq S(\omega, \mathcal{P})+ z(n_1) + z(n_2)+ \sum_{p \in Z(n_1,n_2)} \left(\frac{p-\omega(p)}{\omega(p)} \right)   \leq (1+3  \delta) S(\omega, \mathcal{P}). 
\end{align*}
It follows that for any $(n_1,n_2) \in \mathcal{C}(\delta)$, 
\begin{align*}
\frac{1}{(k-1)!}
            \mathop{\sum \dots \sum}_{\substack{ p_1, \dots, p_k \in \mathcal{P}   \\ p_1 \not\in X(n_1,n_2)  }}  \prod_{i=1}^k   \zeta(p_i;n_1,n_2)  
            \leq  \frac{3\delta (1+ 3\delta)^{k-1}}{(k-1)!}  S(\omega, \mathcal{P})^k  .
\end{align*}
We now take $\delta =  (3k\cdot 2^{k+3} )^{-1} c^*(k,\omega,\mathcal{P})^k$ and note that
\begin{align*}
 \frac{3\delta (1+ 3\delta)^{k-1}}{(k-1)!}  \leq \frac{1}{4}  \frac{ c^*(k,\omega,\mathcal{P})^k}{ 4k!}.
\end{align*}
Continuing on from (\ref{inequ:prooflem2}) and simply using bound $|\mathcal{C}(\delta)| \leq |\mathcal{A}|^2$, this choice of $\delta$ ensures that 
\begin{align}\label{inequ:prooflem3}
&\mathop{\sum \dots \sum}_{\substack{ p_1, \dots, p_k \in \mathcal{P} \\ p_1> \dots > p_k }}     \sum_{\substack{a=1 \\ (a,p_1\dots p_k)=1}}^{p_1\dots p_k} \left|\sum_{\substack{n\leq N \\ n \in \mathcal{A}}} \exp\left(\frac{2 \pi i n a}{p_1\dots p_k} \right) \right|^2 \\ &\geq  \frac{3}{4}
\frac{ c^*(k,\omega,\mathcal{P})^k}{ 4k!}  S(\omega,\mathcal{P})^k  |\mathcal{A}|^2 
          - \frac{1}{(k-1)!}
            \sum_{\substack{ (n_1,n_2) \in  \in \mathcal{A} \times \mathcal{A} \\ (n_1,n_2) \not \in \mathcal{C}(\delta)}} \mathop{\sum \dots \sum}_{\substack{ p_1, \dots, p_k \in \mathcal{P}   \\ p_1 \not\in X(n_1,n_2)  }}  \prod_{i=1}^k   \zeta(p_i;n_1,n_2). \nonumber
\end{align}

\underline{Step 3}: Upper bounds for the negative contribution of $(\mathcal{A} \times \mathcal{A} ) \setminus\mathcal{C}(\delta)$ to (\ref{inequ:prooflem3}).

\medskip

\underline{Step 3A}: A  first upper bound via Lemma~\ref{lem:lem10proof.inter}.

Inequalities (\ref{inequ:useful1}) and (\ref{inequ:useful2}) of Lemma~\ref{lem:lem10proof.inter} and the definitions of $d^*(k,\omega,\mathcal{P})$  and $\mathcal{C}(\delta)$   tell  us   that
\begin{align*}
|(\mathcal{A} \times \mathcal{A} ) \setminus \mathcal{C}(\delta)| &\leq   \widetilde{\varepsilon}(k,\omega,\mathcal{P}) \! \left( \frac{2k! }{(\delta  S(\omega, \mathcal{P}) -k )^k }  + \frac{k!  }{ (\delta  S(\omega, \mathcal{P}) - k M_1(\omega, \mathcal{P}))^k} \right) \mathop{\sum \dots \sum}_{\substack{ p_1, \dots, p_k\in \mathcal{P} \\ p_1> \dots > p_k }}     \prod_{\substack{ i =1}}^k \left(\frac{\omega(p_i)}{p_i} \right)  |\mathcal{A}|^2 \\
&\leq   \widetilde{\varepsilon}(k,\omega,\mathcal{P}) \! \left( \frac{3  }{d^*(k,\omega,\mathcal{P})^k}   \right)    |\mathcal{A}|^2. 
\end{align*}
Further, for any $(n_1, n_2) \in (\mathcal{A} \times \mathcal{A} ) \setminus \mathcal{C}(\delta)$, 
\begin{align*}
 \mathop{\sum \dots \sum}_{\substack{ p_1, \dots, p_k \in \mathcal{P}   \\ p_1 \not\in X(n_1,n_2)  }}  \prod_{i=1}^k   \zeta(p_i;n_1,n_2)   &\leq \left( \sum_{p \in X(n_1,n_2)} \left(\frac{\omega(p)}{p-\omega(p)} \right) + z(n_1)+ z(n_2) +\sum_{p \in Z(n_1,n_2)}\left(\frac{p-\omega(p)}{\omega(p)} \right)\right)^k \\
 &\leq  4^k \max\left\{S(\omega,\mathcal{P})^k,|\mathcal{P}|^k, \left( \sum_{p \in Z(n_1,n_2)}\left(\frac{p-\omega(p)}{\omega(p)} \right)  \right)^k  \right\}. 
\end{align*}

Overall the contribution of 
$(n_1, n_2) \in (\mathcal{A} \times \mathcal{A} ) \setminus \mathcal{C}(\delta)$  to (\ref{inequ:prooflem3}) is therefore bounded by 
\begin{align}\label{inequ:inter1}
&\frac{1}{(k-1)!}\sum_{\substack{(n_1,n_2) \in \mathcal{A} \times \mathcal{A} \\ (n_1,n_2) \not\in \mathcal{C}(\delta)}}     \mathop{\sum \dots \sum}_{\substack{ p_1, \dots, p_k \in \mathcal{P}   \\ p_1 \not\in X(n_1,n_2)  }}  \prod_{i=1}^k   \zeta(p_i;n_1,n_2)  \\  
  &\leq   \left(\frac{ 4^k \max\left\{S(\omega,\mathcal{P})^k,|\mathcal{P}|^k\right\}}{(k-1)!}  \right)  |(\mathcal{A} \times \mathcal{A}) \setminus \mathcal{C}(\delta)| +  \frac{4^k}{(k-1)! } \sum_{(n_1,n_2) \in \mathcal{A} \times \mathcal{A}}^* \left( \sum_{p \in Z(n_1,n_2)}\left(\frac{p-\omega(p)}{\omega(p)} \right)  \right)^k   \nonumber\\ 
  &\leq   \frac{\widetilde{\varepsilon}(k,\omega,\mathcal{P})}{(k-1)!} \max\left\{1,\left(\frac{|\mathcal{P}|}{S(\omega,\mathcal{P})}\right)^k\right\} \left( \frac{ 4^{k+1} }{d^*(k,\omega,\mathcal{P})^k}   \right)   S(\omega,\mathcal{P})^k  |\mathcal{A}|^2 \nonumber \\ &+  \frac{4^k}{(k-1)! } \sum_{(n_1,n_2) \in \mathcal{A} \times \mathcal{A}}^* \left( \sum_{p \in Z(n_1,n_2)}\left(\frac{p-\omega(p)}{\omega(p)} \right)  \right)^k,   \nonumber
\end{align}
where $^*$ indicates that the summation is restricted to $(n_1,n_2)$ with $\sum_{p \in Z(n_1,n_2)}\left(\frac{p-\omega(p)}{\omega(p)} \right) \geq S(\omega, \mathcal{P})$.

\underline{Step 3B}: Consideration of large $\sum_{p \in Z(n_1,n_2)}\left(\frac{p-\omega(p)}{\omega(p)} \right)$.

 Suppose that $\sum_{p \in Z(n_1,n_2)}\left(\frac{p-\omega(p)}{\omega(p)} \right) \geq S(\omega, \mathcal{P})$. 
Recall  $\delta =  (3k\cdot 2^{k+3} )^{-1} c^*(k,\omega,\mathcal{P})^k$ and recall 
\begin{align*} 
&d^*(k,\omega,\mathcal{P})=\left( \sum_{p \in \mathcal{P}} \left(\frac{\omega(p)}{p-\omega(p)} \right) \right)^{-1} \left\{\frac{c^*(k,\omega,\mathcal{P})^k}{   2^{k+5} k} \sum_{p \in \mathcal{P}} \left(\frac{\omega(p)}{p-\omega(p)} \right) - k \max\left\{1,\max_{p \in \mathcal{P}}  \left(\frac{p-\omega(p)}{\omega(p)} \right) \right\} \right\} 
\end{align*}

Observe that
\begin{align*}
   k \max_{p \in \mathcal{P}}  \left(\frac{p-\omega(p)} {\omega(p)}\right) \leq  \!\left( \sum_{p \in \mathcal{P}} \left(\frac{\omega(p)}{p-\omega(p)} \right) \right)\! (  \delta - d^*(k,\omega, \mathcal{P}) )  \leq \left( \sum_{p \in Z(n_1,n_2)} \!\left(\frac{p-\omega(p)}{\omega(p)} \right) \right) (  \delta - d^*(k,\omega, \mathcal{P}) ).  
\end{align*}

Since $M_1(\omega, \mathcal{P}) = \max_{p \in \mathcal{P}}  \left(\frac{p-\omega(p)} {\omega(p)}\right)$,   we then get
\begin{align*}
  \mathop{\sum \dots \sum}_{\substack{ p_1, \dots, p_k\in Z(n_1,n_2) \\ p_1> \dots > p_k }} \prod_{i=1}^k  \left( \frac{p_i-\omega(p_i)} {\omega(p_i)}\right)
 &\geq \frac{1}{k!} \left( \sum_{p \in Z(n_1,n_2)}\left(\frac{p-\omega(p)}{\omega(p)} \right)  - k M_1(\omega, \mathcal{P}) \right)^k \\
 &\geq \frac{(1-  \delta + d^*(k,\omega, \mathcal{P}) )^k}{k!}   \left( \sum_{p \in Z(n_1,n_2)}\left(\frac{p-\omega(p)}{\omega(p)} \right)    \right)^k.
\end{align*} 
Thus  if $\sum_{p \in Z(n_1,n_2)}\left(\frac{p-\omega(p)}{\omega(p)} \right) \geq S(\omega, \mathcal{P})$, we have inequality 
\begin{align*}  &  \left( \sum_{p \in Z(n_1,n_2)}\left(\frac{p-\omega(p)}{\omega(p)} \right)  \right)^k   \leq \frac{k!}{(1-  \delta + d^*(k,\omega, \mathcal{P}) )^k}  \left(\mathop{\sum \dots \sum}_{\substack{ p_1, \dots, p_k\in Z(n_1,n_2) \\ p_1> \dots > p_k }} \prod_{i=1}^k  \left( \frac{p_i-\omega(p_i)} {\omega(p_i)}\right) \right).
\end{align*}

Applying inequality (\ref{inequ:useful4}) of  Lemma~\ref{lem:lem10proof.inter}, we finally observe that 
\begin{align*}
&\sum_{(n_1,n_2) \in \mathcal{A} \times \mathcal{A}}^* \left( \sum_{p \in Z(n_1,n_2)}\left(\frac{p-\omega(p)}{\omega(p)} \right)  \right)^k \\ &\leq \frac{k!}{(1-  \delta + d^*(k,\omega, \mathcal{P}) )^k}  \left(\sum_{\substack{(n_1,n_2) \in \mathcal{A} \times \mathcal{A} }}   \mathop{\sum \dots \sum}_{\substack{ p_1, \dots, p_k\in Z(n_1,n_2) \\ p_1> \dots > p_k }} \prod_{i=1}^k  \left( \frac{p_i-\omega(p_i)} {\omega(p_i)}\right) \right) \\
&\leq \frac{ k! \widetilde{\varepsilon}(k,\omega,\mathcal{P})}{(1-  \delta + d^*(k,\omega, \mathcal{P}) )^k} \mathop{\sum \dots \sum}_{\substack{ p_1, \dots, p_k\in \mathcal{P} \\ p_1> \dots > p_k }}     \prod_{\substack{ i =1}}^k \left(\frac{\omega(p_i)}{p_i} \right)  |\mathcal{A}|^2.
\end{align*}

\underline{Step 3C}: Putting everything together. Taking another look at (\ref{inequ:inter1}), we therefore get   \begin{align} \label{inequ:inter2}
&\frac{1}{(k-1)!}\sum_{\substack{(n_1,n_2) \in \mathcal{A} \times \mathcal{A} \\ (n_1,n_2) \not\in \mathcal{C}}}     \mathop{\sum \dots \sum}_{\substack{ p_1, \dots, p_k \in \mathcal{P}   \\ p_1 \not\in X(n_1,n_2)  }}  \prod_{i=1}^k   \zeta(p_i;n_1,n_2)  \\  
  &\leq   \frac{\widetilde{\varepsilon}(k,\omega,\mathcal{P})}{(k-1)!} \max\left\{1,\left(\frac{|\mathcal{P}|}{S(\omega,\mathcal{P})}\right)^k\right\} \left( \frac{ 4^{k+1} }{d^*(k,\omega,\mathcal{P})^k}   \right)   S(\omega,\mathcal{P})^k  |\mathcal{A}|^2 \nonumber \\ &+  \frac{4^k}{(k-1)! }\frac{\widetilde{  \varepsilon}(k,\omega,\mathcal{P})}{(1-  \delta + d^*(k,\omega, \mathcal{P}) )^k}     S(\omega,\mathcal{P})^k  |\mathcal{A}|^2    \nonumber \\
  &\leq   \frac{\widetilde{\varepsilon}(k,\omega,\mathcal{P})}{(k-1)!} \max\left\{1,\left(\frac{|\mathcal{P}|}{S(\omega,\mathcal{P})}\right)^k\right\} \left( \frac{ 4^{2k+1} }{d^*(k,\omega,\mathcal{P})^k}   \right)   S(\omega,\mathcal{P})^k  |\mathcal{A}|^2. \nonumber
\end{align}

However, the definition $\widetilde{\varepsilon}(k,\omega,\mathcal{P}) = \frac{1}{16 k} \min\left\{1,\left(\frac{S(\omega,\mathcal{P})}{|\mathcal{P}|}\right)^k  \right\}    \left( \frac{  c^*(k,\omega,\mathcal{P})^k d^*(k,\omega,\mathcal{P}     )^k } {  4^{2k+1}   } \right) $ ensures  
\begin{align*}
\frac{\widetilde{\varepsilon}(k,\omega,\mathcal{P})}{(k-1)!}  \max\left\{1,\left(\frac{|\mathcal{P}|}{S(\omega,\mathcal{P})}\right)^k\right\}   \left( \frac{  4^{2k+1}   }{ d^*(k,\omega,\mathcal{P}     )^k }  \right) \leq   \frac{1}{4}
\frac{ c^*(k,\omega,\mathcal{P})^k}{ 4 k!}.           
\end{align*}
Combining (\ref{inequ:prooflem3}) with (\ref{inequ:inter2}), we therefore find that 
 \begin{align*}
&\mathop{\sum \dots \sum}_{\substack{ p_1, \dots, p_k \in \mathcal{P} \\ p_1> \dots > p_k }}     \sum_{\substack{a=1 \\ (a,p_1\dots p_k)=1}}^{p_1\dots p_k} \left|\sum_{\substack{n\leq N \\ n \in \mathcal{A}}} \exp\left(\frac{2 \pi i n a}{p_1\dots p_k} \right) \right|^2  \geq 
\frac{ c^*(k,\omega,\mathcal{P})^k}{ 8 } \left\{ \mathop{\sum \dots \sum}_{\substack{ p_1, \dots, p_k \in \mathcal{P} \\ p_1> \dots > p_k }}  \prod_{i=1}^k \left( \frac{\omega(p_i)}{p_i-\omega(p_i)}\right)\right\} |\mathcal{A}|^2.
\end{align*}
Combining this lower bound with Linnik's large sieve inequality (Lemma~\ref{lem:largesieve.linnik}), we get the proposed upper bound on $|\mathcal{A}|$.  
    \end{proof}

\section{Lower and upper bounds for sumset decompositions}\label{sec:ostmannproofs}

In this section we prove Theorem~\ref{thm:2squares}. We show the following: If $\mathcal{A}+\mathcal{B}=\mathcal{S}+\mathcal{S}$ and $|\mathcal{A}|, |\mathcal{B}| \geq 2$, then 
\begin{align*} 
\sqrt{x}/(\log x)^{9/2} \ll |\mathcal{A}\cap [1,x]|, |\mathcal{B}\cap [1,x]| \ll \sqrt{x} (\log x)^4.
\end{align*}
The proof uses   various   sieve inequalities and our strategy closely resembles the treatment of asymptotic additive decompositions of the primes, with the large sieve inequality replaced by Theorem~\ref{thm:largesieveimprovement}. 

\subsection{Useful notation} The following notation is used throughout this section: 

\begin{defn}\label{def:residueclasses}
Consider a set   $\mathcal{C}$ and recall that $\mathcal{C}(x)=\mathcal{C} \cap [1,x]$. For a given prime $p$, we then define
\begin{align*}
&V_{\mathcal{C}}(p;x) = \{ (c \mbox{ mod } p) \in \mathbb{Z}/p\mathbb{Z}:  c  \in \mathcal{C}(x)\}, \\
&W_{\mathcal{C}}(p;x) = \{  (c \mbox{ mod } p)  \in V_{\mathcal{C}}(p;x): \forall c_1, c_2 \in \mathcal{C}(x) \,\, c_1 \equiv c_2 \equiv c (p) \Rightarrow c_1 \equiv c_2  (p^2)\},\\
&U_{\mathcal{C}}(p;x) = V_{\mathcal{C}}(p;x) \setminus W_{\mathcal{C}}(p;x). 
\end{align*} 
Furthermore,   $v_{\mathcal{C}}(p;x) = |V_{\mathcal{C}}(p;x)| $, $w_{\mathcal{C}}(p;x) =  |W_{\mathcal{C}}(p;x)| $ and $u_{\mathcal{C}}(p;x) =  |U_{\mathcal{C}}(p;x)| $.
\end{defn}

\subsection{Residue class structure}
 We begin the proof of Theorem~\ref{thm:2squares} by thinking about the structure of the residue classes of $\mathcal{A}$ and $\mathcal{B}$ with $\mathcal{A} + \mathcal{B} =  \mathcal{S}+\mathcal{S}$.

\begin{lemma}\label{lem:basicstructure}

Suppose $\mathcal{A}, \mathcal{B} \subseteq \mathbb{Z}$ with  $\mathcal{A} + \mathcal{B}   =  \mathcal{S}+\mathcal{S}$. Using the notation from Definition~\ref{def:residueclasses}, we   have the following relationships for any $x>0$ and any prime $p$ with $p \equiv 3 \mbox{ mod } 4$:
\begin{align*}
    &U_{\mathcal{A} }(p;x) \cap (-V_{\mathcal{B} }(p;x)) = \emptyset  \quad \mbox{ and } \quad U_{\mathcal{B} }(p;x) \cap (-V_{\mathcal{A} }(p;x)) = \emptyset, \\
    &u_{\mathcal{A} }(p;x)  + v_{\mathcal{B}}(p;x) \leq p  \quad\quad\,\,\,\, \mbox{ and }  \quad u_{\mathcal{B} }(p;x) +v_{\mathcal{A} }(p;x)  \leq p.
\end{align*}
\end{lemma}

\begin{proof}
Suppose  $\mathcal{A}  + \mathcal{B}   =  \mathcal{S}+\mathcal{S}$ and $p \equiv 3 \mbox{ mod } 4$. Recall the following    property of $\mathcal{S} + \mathcal{S}$: If $n \in \mathcal{S} + \mathcal{S}$, then $n \equiv 0$ mod $p$ implies $n \equiv 0$ mod $p^2$.

\medskip 
Let $a  \in \mathcal{A}(x)  $ and $b \in \mathcal{B}(x)  $. Then $a+b \in \mathcal{S}+\mathcal{S}$. Suppose $a + b\equiv 0$ mod $p$. This implies $a + b \equiv 0$ mod $p^2$. Suppose further that there exists $a' \in \mathcal{A}(x)$ with $ a'  \equiv a $ mod $p$, but $a'  \not \equiv a $ mod $p^2$. Then $a + b \equiv 0$ mod $p$ implies both $b \equiv -a $ mod $p^2$ and $b \equiv -a' $ mod $p^2$, a contradiction. Hence $a + b \equiv 0$ mod $p$ tells us that $\mathcal{A}(x)  $ contains no elements of the residue classes $a+kp$ mod $p^2$ for $k \in \{1, \dots, p-1\}$ and, similarly, $\mathcal{B}(x)  $ contains no elements of the residue classes $b+kp$ mod $p^2$ for $k \in \{1, \dots, p-1\}$. 

\medskip If $(a \mbox{ mod } p) \in U_{\mathcal{A}}(p;x)$, then $\mathcal{A}(x)$ contains at least two distinct residue classes of the form $a+kp$, and the discussion above shows that there cannot be any $b \in \mathcal{B}(x)$ with $a + b\equiv 0$ mod $p$. Hence we have $-a \not\in  V_{\mathcal{B}}(p;x)$ and  $U_{\mathcal{A} }(p;x) \cap (-V_{\mathcal{B} }(p;x)) = \emptyset$. Similarly,   $ U_{\mathcal{B} }(p;x) \cap (-V_{\mathcal{A} }(p;x)) = \emptyset$. Since $|\mathbb{Z}/p\mathbb{Z}|=p$, disjointedness also implies $u_{\mathcal{A} }(p;x)  + v_{\mathcal{B}}(p;x) \leq p$ and $ u_{\mathcal{B} }(p;x) +v_{\mathcal{A} }(p;x)  \leq p$.
\end{proof}

\subsection{Bounds on the product of set cardinalities}
We now give simple bounds on the product of cardinalities $|\mathcal{A}  \cap (-\infty,x]||\mathcal{B} \cap (-\infty,x]| $. (From later results it follows that the given bounds also hold for the product  $|\mathcal{A}(x) ||\mathcal{B} (x)|.$)

\begin{lemma}\label{lem:productsize}
Suppose $\mathcal{A}\subseteq \mathbb{Z}$ and $\mathcal{B}  \subseteq \mathbb{Z}$ with  $\mathcal{A}  + \mathcal{B} = \mathcal{S}+\mathcal{S}$.
 Then for large $x$, 
\begin{align*}
 x \log(x)^{-1/2} \ll   |\mathcal{A}  \cap (-\infty,x]||\mathcal{B} \cap (-\infty,x]| \ll x \log(x)^2.
\end{align*} 
$($The implied constants and the required size of $x$ depend on $\mathcal{A}$ and $\mathcal{B}$.$)$
\end{lemma}

\begin{proof} Throughout the proof we use the notation given in Definition~\ref{def:residueclasses}.  We begin our proof with the lower bound, which is very easy to obtain.

\medskip

\underline{Lower bound}: Write $K=\min( \mathcal{A}  \cup \mathcal{B} )$.   We have $\mathcal{A}+\mathcal{B} =  \mathcal{S}+\mathcal{S}$ and so  $c \in  \mathcal{S}+\mathcal{S}$ implies $c=a+b$ for some $a \in \mathcal{A}$ and $b\in \mathcal{B}$ with $a \leq c-K$ and $b \leq c-K$. Then $$(\mathcal{A}\cap (-\infty,x]) + (\mathcal{B}\cap (-\infty,x])\supseteq (\mathcal{S} + \mathcal{S}) \cap \left[1,x + K\right].$$  
The number of integers up to $x$ which are sums of two squares is $\gg x \log(x)^{-1/2}$. It follows that 
$$|\mathcal{A}  \cap (-\infty,x]||\mathcal{B} \cap (-\infty,x]| \gg_K x\log(x)^{-1/2},$$
with the implied constant dependent on $K =\min( \mathcal{A}  \cup \mathcal{B} )$.

\medskip
Now we prove the upper bound in Lemma~\ref{lem:productsize}. Fix $x$ and assume $ |\mathcal{A}(x)| \leq |\mathcal{B}(x)|$. 
    
    \medskip 
    \underline{Upper bound  1}: Suppose first that there exists some $p \equiv 3$ mod $4$ with $p \in [\frac{1}{2}\sqrt{x},\sqrt{x}]$ and $$\#\{b \in \mathcal{B}(x): (b \mbox{ mod } p) \in W_{\mathcal{B} }(p;x)\} > \frac{w_{\mathcal{B} }(p;x)}{100p}  |\mathcal{B}(x)|.$$ 
    Since $( h \mbox{ mod } p) \in W_{\mathcal{B} }(p;x)$ implies $(\forall b_1, b_2 \in \mathcal{B}(x) \,\, b_1 \equiv b_2 \equiv h (p) \Rightarrow b_1 \equiv b_2 (p^2))$, the elements of $\{b \in \mathcal{B}(x): (b \mbox{ mod } p) \in W_{\mathcal{B} }(p;x)\}$ occupy only $w_{\mathcal{B} }(p;x)$ residue classes modulo $p^2$. But then $$ \frac{w_{\mathcal{B} }(p;x)}{100p} |\mathcal{B}(x)| < \#\{b \in \mathcal{B}(x): (b \mbox{ mod } p) \in W_{\mathcal{B} }(p;x)\} \leq w_{\mathcal{B} }(p;x)\left(\frac{x}{p^2}+1\right).$$
    Since the inequality is strict, $w_{\mathcal{B}}(p;x) \neq 0$. 
Rearranging, we get $ |\mathcal{B}(x)|  \ll \sqrt{x}$ and $|\mathcal{A}(x)||\mathcal{B}(x)| \ll x$. It follows that also 
$|\mathcal{A}  \cap (-\infty,x]||\mathcal{B} \cap (-\infty,x]| \ll_K x$, with the implied constant dependent on $\min(\mathcal{A} \cup \mathcal{B})$.

    \medskip \underline{Upper bound  2}: Now we may assume that for every $p \equiv 3$ mod $4$ with $p \in [\frac{1}{2}\sqrt{x},\sqrt{x}]$, 
    $$\#\{b \in \mathcal{B}(x): (b \mbox{ mod } p) \in W_{\mathcal{B} }(p;x)\} \leq \frac{w_{\mathcal{B}}(p;x)}{100p} |\mathcal{B}(x)|.$$ 

Since $  b \in \mathcal{B}(x)$ implies $  (b \mbox{ mod } p) \in V_{\mathcal{B} }(p;x)$,   also   $\#\{b \in \mathcal{B}(x): (b \mbox{ mod } p) \in (\mathbb{Z}/p\mathbb{Z}) \setminus V_{\mathcal{B} }(p;x)\} = 0$ and  
     $$\#\{b \in \mathcal{B}(x): (b \mbox{ mod } p) \in W_{\mathcal{B} }(p;x)\cup ((\mathbb{Z}/p\mathbb{Z}) \setminus V_{\mathcal{B} }(p;x))\} \leq \frac{w_{\mathcal{B}}(p;x)+p-v_{\mathcal{B}}(p;x) }{100p} |\mathcal{B}(x)|.$$

We   apply Lemma~\ref{lem:uncommon.residue.classes} with $\mathcal{P}=\{p \in [\frac{1}{2}\sqrt{x},\sqrt{x}]: p \equiv 3 \mbox{ mod } 4\}$ and $\Omega(p)=W_{\mathcal{B} }(p;x)\cup ((\mathbb{Z}/p\mathbb{Z}) \setminus V_{\mathcal{B} }(p;x))$ and $\varepsilon = \frac{1}{100}$.    Then $\omega(p) =  p -v_{\mathcal{B}}(p;x)  +w_{\mathcal{B}}(p;x)  = p- u_{\mathcal{B}}(p;x) 
$.  We find that 
 \begin{align*}
    |\mathcal{B}(x)| \ll x \Bigg( \sum_{\substack{p \in [\frac{1}{2}\sqrt{x},\sqrt{x}] \\ p \equiv 3 (4) }} \left( \frac{p- u_{\mathcal{B}}(p;x) }{u_{\mathcal{B}}(p;x)  } \right)\Bigg)^{-1}.
\end{align*} 
     On the other hand, applying the large sieve (Lemma~\ref{lem:largesieve.montgomery}) to $\mathcal{A}(x)$ with $\Omega(p)= (\mathbb{Z}/p\mathbb{Z}) \setminus V_{\mathcal{A}}(p;x)$
      gives 
      \begin{align*}
    |\mathcal{A}(x)| \ll x \Bigg( \sum_{\substack{p \in [\frac{1}{2}\sqrt{x},\sqrt{x}] \\ p \equiv 3 (4) }} \left(\frac{p-v_{\mathcal{A}}(p;x)}{v_{\mathcal{A}}(p;x)} \right)\Bigg)^{-1}.
\end{align*} 
Recall now from Lemma~\ref{lem:basicstructure} that $u_{\mathcal{B}}(p;x)+ v_{\mathcal{A}}(p;x) \leq p$. Since $ \frac{p-x}{x} \geq \frac{p-y}{y}$ for $x \leq y$ when $x,y \in [1,p]$,   
\begin{align*}
\left( \frac{p- u_{\mathcal{B}}(p;x) }{u_{\mathcal{B}}(p;x)  } \right) \left(\frac{p-v_{\mathcal{A}}(p;x)}{v_{\mathcal{A}}(p;x)} \right)&\geq \left(\frac{v_{\mathcal{A}}(p;x)} {p-v_{\mathcal{A}}(p;x)}\right) \left( \frac{p- v_{\mathcal{A}}(p;x) }{v_{\mathcal{A}}(p;x)  } \right)= 1.
\end{align*} 
By the Cauchy-Schwarz inequality, we then get the   upper bound 
\begin{align*}
     |\mathcal{A}(x)||\mathcal{B}(x)| &\ll x^2 \Bigg( \sum_{\substack{p \in [\frac{1}{2}\sqrt{x},\sqrt{x}] \\ p \equiv 3 (4) }} \left( \frac{p- v_{\mathcal{A}}(p;x) }{v_{\mathcal{A}}(p;x)  } \right)\Bigg)^{-1}  \Bigg( \sum_{\substack{p \in [\frac{1}{2}\sqrt{x},\sqrt{x}] \\ p \equiv 3 (4) }} \left( \frac{p- u_{\mathcal{B}}(p;x) }{u_{\mathcal{B}}(p;x)  } \right)\Bigg)^{-1}\\
     &\leq x^2 \Big( \sum_{\substack{p \in [\frac{1}{2}\sqrt{x},\sqrt{x}]   \\ p \equiv 3 (4) }} \left( \frac{p- v_{\mathcal{A}}(p;x) }{v_{\mathcal{A}}(p;x)  } \right)^{1/2} \left( \frac{p- u_{\mathcal{B}}(p;x) }{u_{\mathcal{B}}(p;x)  } \right)^{1/2}\Big)^{-2} 
     \ll x\log(x)^2. 
\end{align*}
It follows that also 
$|\mathcal{A}  \cap (-\infty,x]||\mathcal{B} \cap (-\infty,x]| \ll_K x\log(x)^2$.
\end{proof}

\subsection{Bounds on the individual sets   } 

We now apply many different sieve estimates  to prove upper and lower bounds on $|\mathcal{A}(x)|$ and $|\mathcal{B}(x)|$ for $\mathcal{A}+\mathcal{B}=\mathcal{S}+\mathcal{S}$.   The sieve estimates used are Brun's small sieve, Montgomery's large sieve, Gallagher's larger sieve, as well as Theorem~\ref{thm:largesieveimprovement}.

\subsubsection{Application of Brun's small sieve}

We first show that $\mathcal{A}(x)$ and $\mathcal{B}(x)$ have at least $\log(x)^{c}$ elements. We use Brun's small sieve, stated in Lemma~\ref{lem:small.brun}. 

\begin{lemma}\label{lem:size.step1}
Suppose $\mathcal{A}\subseteq \mathbb{Z}$ and $\mathcal{B}  \subseteq \mathbb{Z}$ with  $\mathcal{A}  + \mathcal{B} = \mathcal{S}+\mathcal{S}$. Suppose  that $|\mathcal{A}|,|\mathcal{B}| \geq 2$. Let $c \in \mathbb{N}$.

 Then  for sufficiently large $x$, 
\begin{align*}
\log(x)^{c} \ll   |\mathcal{A}(x)|, |\mathcal{B}(x)| \ll x \log(x)^{-c},
\end{align*}
where the implied constants  and the required size of $x$ depend on $ \mathcal{A}$, $ \mathcal{B} $ and $c$.
\end{lemma}
 
\begin{proof}
Suppose  $\mathcal{A}  + \mathcal{B} = \mathcal{S}+\mathcal{S}$.  Choose some $k\geq 2$ and assume $|\mathcal{A}| \geq k$ and $|\mathcal{B}| \geq k$. Let $\{a_1, \dots, a_k\} $ be the set containing the $k$ smallest elements of $\mathcal{A}  $, and let $\{b_1, \dots, b_k\} $ be the set containing the $k$ smallest elements of $\mathcal{B}  $. 
 Fix $x$ and assume (without loss of generality) that $|\mathcal{A}(x)| \leq |\mathcal{B}(x)|$.

\medskip
Consider $p \equiv 3 \mbox{ mod } 4$. Since  $\{a_1, \dots, a_k\} \subseteq \mathcal{A} $, we have that for any $b \in \mathcal{B}(x)$,  $$( b\equiv -a_i \mbox{ mod } p ) \Rightarrow ( b \equiv -a_i \mbox{ mod } p^2) \quad  \mbox{ for all } i \in \{1,\dots,k\}.$$ Let $\mathcal{P} = \{p \in   [ \log(x)^k, x^{1/k}]: p \equiv 3 \mbox{ mod } 4 \}$. We then define $$W(x;k) = \{ b \in \mathcal{B}(x): \exists p \in \mathcal{P} \,\,\exists i \in \{1, \dots, k\} \mbox{ with } b \equiv -a_i (p)\}$$ and $U(x;k) = \mathcal{B}(x) \setminus W(x;k) $.

\medskip \underline{Step 1}: Bounding $|W(x;k)| $. For each $p \in \mathcal{P}$ and each $i \in \{1, \dots, k\}$, there exist only $O_k( \frac{x}{p^2})$ integers  $b \in \mathcal{B}(x) $ with $ b \equiv -a_i (p^2)$.  Hence 
$$|W(x;k)| \ll \sum_{p \in \mathcal{P}} \frac{kx}{p^2}\leq \sum_{p \geq \log(x)^k} \frac{kx}{p^2} \ll_k \frac{x}{\log(x)^k}.$$

\medskip \underline{Step 2}: Bounding $|U(x;k)| $. By construction, all elements $n$ of $U(x;k)$ satisfy $n \not\equiv -a_i (p)$ for all  $p \in \mathcal{P}$ and all $i \in \{1, \dots, k\}$. In other words, if $n \in U(x;k)$, then $ \prod_{i=1}^k(n+a_i) $ is coprime to $p$ for all $p \in \mathcal{P}$. In particular, for  $\mathcal{C} = \{\prod_{i=1}^k(n+a_i): n \in \mathbb{N} \cap [1,x] \}$, 
\begin{align*}
|U(x;k)| \leq \#\{c \in \mathcal{C}:\forall p \in \mathcal{P}\,\, (c,p)=1 \}.
\end{align*}
We now apply a small sieve (Lemma~\ref{lem:small.brun}) to $\mathcal{C}$. We note $X = |\mathcal{C}| \asymp  x$ and take $\omega(p) = k$. For $p \geq 2k$ we then get $\frac{\omega(p)}{p} \leq \frac{1}{2}$ and condition (1) of Lemma~\ref{lem:small.brun} holds.  
For distinct $p_1, \dots, p_j \in \mathcal{P}$, an integer $\prod_{i=1}^k(n+a_i)  $ is divisible by $p_1\dots p_j$ if and only if $n \equiv -\alpha$ mod $p_1\dots p_j$ for some $\alpha \in \mathbb{Z}/(p_1 \dots p_j)\mathbb{Z}$ with $$\alpha \in \{a_1, \dots, a_k\} \mbox{ mod } p_i \quad \mbox{ for all } i \in \{1, \dots, j\}.$$ Assuming $p_1, \dots,\ p_j \geq \max\{a_1, \dots, a_k\}- \min(\mathcal{A})$ (so that we get $k$ different residue classes mod $p_i$), there are $k^j$ possible choices for $\alpha$ mod $p_1\dots p_j$.  Thus 
 $$\left|\#\{c \in \mathcal{C}: p_1 \dots p_j \,\, \mbox{\rm divides } c\} - \frac{k^j}{p_1 \dots p_j} \#\mathcal{C}\right| \leq k^j$$
 and condition (2) of Lemma~\ref{lem:small.brun} holds for $\omega(p)=k$. Finally, we observe that for $2 \leq w \leq x^{1/k}$,
 $$\sum_{\substack{w \leq p <x^{1/k} \\ p \in \mathcal{P}}} \frac{\omega(p) \log(p)}{p} =\sum_{\substack{\max\{w,\log(x)^k\} \leq p <x^{1/k} \\ p \equiv 3(4)}} \frac{k\log(p)}{p} \leq k\log\left(\frac{x^{1/k}}{w}\right)  +O(1),$$
 and so condition (3) is   satisfied. Lemma~\ref{lem:small.brun}  applies if $\log(x)^k \geq \max\{2k,a_1- \min(\mathcal{A}), \dots, a_k- \min(\mathcal{A})\} $. So, assuming $x$ is sufficiently large (where the required size of $x$ depends on $\mathcal{A}$ and $k$), we then have
 \begin{align*}
|U(x;k)| &\leq \#\{c \in \mathcal{C}: \forall p \in \mathcal{P} \,\, (c,p)=1\} \\ &\ll_{k}  x\prod_{\substack{\log(x)^k \leq p <x^{1/k} \\ p \equiv 3(4)}} \left(1-\frac{k}{p} \right) \ll
  x\exp\Big(\sum_{\substack{\log(x)^k \leq p <x^{1/k} \\ p \equiv 3(4)}} \frac{1}{p} \Big)^{-k} \\ 
  &\ll  
  x\exp\Big(  \frac{ \log\log(x)- \log\log\log(x)}{2}+ O_k(1) \Big)^{-k} \ll_k  
  \frac{x}{ \log(x)^{k/3}}.
 \end{align*}
 Overall, we thus have $|\mathcal{B}(x)| \ll_{k,\mathcal{A}}  x \log(x)^{-k/3}$ and $|\mathcal{B}\cap (-\infty,x]| \ll_{k,\mathcal{A}}  x \log(x)^{-k/3}$. But by Lemma~\ref{lem:productsize} we have $|\mathcal{A}\cap (-\infty,x]| |\mathcal{B}\cap (-\infty,x]| \gg_{\mathcal{A},\mathcal{B}} x \log(x)^{-1/2}$ and so     $|\mathcal{A}(x)|  \gg_{k,\mathcal{A},\mathcal{B}}  \log(x)^{k/3-1/2}$.
 
 \medskip
\underline{Observation}: We have shown that if $|\mathcal{A}|, |\mathcal{B}| \geq k$  and $x \gg_{k,\mathcal{A},\mathcal{B}}   1$, then  $|\mathcal{A}(x)|, |\mathcal{B}(x)| \gg_{k,\mathcal{A},\mathcal{B}}  \log(x)^{k/3-1/2}$. 
 
 \medskip
\underline{Final Step}: In the statement of Lemma~\ref{lem:size.step1}, we assumed that $| \mathcal{A} | \geq 2$ and $| \mathcal{B} | \geq 2$. So we first apply the observation with $k=2$ to get $|\mathcal{A}(x)| \gg_{\mathcal{A},\mathcal{B}} \log(x)^{2/3-1/2}$ and $|\mathcal{B}(x)| \gg_{\mathcal{A},\mathcal{B}} \log(x)^{2/3-1/2}$ for sufficiently large $x$.  Then we apply the observation with $k=3c+9$ to get  $|\mathcal{A}(x)| \gg \log(x)^{c+2}$ and $|\mathcal{B}(x)| \gg  \log(x)^{c+2}$, where the implied constants depend on $\mathcal{A}$, $\mathcal{B}$ and $c$. Using  Lemma~\ref{lem:productsize}, which tells us that 
$|\mathcal{A}\cap (-\infty,x]| |\mathcal{B}\cap (-\infty,x]| \ll_{\mathcal{A},\mathcal{B}} x \log(x)^{2}$, this 
  also implies $|\mathcal{A}(x)|,|\mathcal{B}(x)| \ll_{\mathcal{A},\mathcal{B},c} x\log(x)^{-c}.$
\end{proof}

\subsubsection{Application of the Johnsen-Selberg sieve and Gallagher's larger sieve}

We now show that $\mathcal{A}(x)$ and $\mathcal{B}(x)$ have at least $x^{1/4-\varepsilon}$ elements. To do so, we use  Johnsen-Selberg   and Gallagher's larger sieve.

\begin{lemma}\label{lem:size.step2}
Suppose $\mathcal{A}\subseteq \mathbb{Z}$ and $\mathcal{B}  \subseteq \mathbb{Z}$ with  $\mathcal{A}  + \mathcal{B} = \mathcal{S}+\mathcal{S}$. Suppose  that $|\mathcal{A}|,|\mathcal{B}| \geq 2$. Let $\varepsilon >0$.

 Then  for sufficiently large $x$, 
\begin{align*}
x^{1/4-\varepsilon} \ll   |\mathcal{A}(x)|, |\mathcal{B}(x)| \ll x^{3/4+\varepsilon},
\end{align*}
where the implied constants  and the required size of $x$ depend on $ \mathcal{A}$, $ \mathcal{B} $ and $\varepsilon$.
\end{lemma}

\begin{proof}
Suppose $\mathcal{A}  + \mathcal{B} = \mathcal{S}+\mathcal{S}$ with $|\mathcal{A}|, |\mathcal{B}|\geq 2$. Fix $x$ and assume  $|\mathcal{A}(x)| \leq |\mathcal{B}(x)|$.  Let $\varepsilon \in (0,\frac{1}{4})$ and  
\begin{align*}
&\mathcal{P}_1(\varepsilon;x) = \{p \in  [1,\log(x)^{2/\varepsilon}]: p \equiv 3 \mbox{ mod } 4 \mbox{ and }  v_{\mathcal{A}}(p;x) \leq p^{1-\varepsilon}\}, \\ 
&\mathcal{P}_2(\varepsilon;x) = \{p \in  [1,\log(x)^{2/\varepsilon}]: p \equiv 3 \mbox{ mod } 4 \mbox{ and } v_{\mathcal{A}}(p;x) > p^{1-\varepsilon}\},  
\end{align*} 
where we  use the notation given in Definition~\ref{def:residueclasses}.
Recall from Lemma~\ref{lem:basicstructure} that $u_{\mathcal{B}}(p;x)+v_{\mathcal{A}}(p;x) \leq p$ for $p \equiv 3 \mbox{ mod } 4$,  and observe that $\mathcal{P}_2(\varepsilon;x) \subseteq \{ p \in  [1,\log(x)^{2/\varepsilon}]: u_{\mathcal{B}}(p;x) \leq p- p^{1-\varepsilon}\}$.  

\medskip

\underline{Case 1}: $|\mathcal{P}_1(\varepsilon;x)| \geq |\mathcal{P}_2(\varepsilon;x)|$. We apply Gallagher's larger sieve (Lemma~\ref{lem:largesieve.gallagher}) with $\mathcal{P}=\mathcal{P}_1(\varepsilon;x)$ and get
\begin{align*}
    |\mathcal{A}(x)| &\leq \frac{-\log(x) + \sum_{p \in \mathcal{P}_1(\varepsilon;x)} \log(p)}{-\log(x) + \sum_{p \in \mathcal{P}_1(\varepsilon;x)} \frac{\log(p) }{v_{\mathcal{A}}(p;x)}}
    \\ &\leq \frac{-\log(x) + \sum_{p \in \mathcal{P}_1(\varepsilon;x)} \log(p)}{-\log(x) + \sum_{p \in \mathcal{P}_1(\varepsilon;x)} \frac{\log(p) }{p^{1-\varepsilon}}} \\
    &\leq \frac{-\log(x) + |\mathcal{P}_1(\varepsilon;x)| \log(\log(x)^{2/\varepsilon})}{-\log(x) + |\mathcal{P}_1(\varepsilon;x)| \frac{\log(\log(x)^{2/\varepsilon}) }{\log(x)^{2/\varepsilon-2}}}  \\
    &\ll_\varepsilon  \frac{-\log(x) + \log(x)^{2/\varepsilon}}{-\log(x) + (1/8)\log(x)^{2}}.
\end{align*}
Hence $|\mathcal{A}(x)| \ll_\varepsilon  \log(x)^{2/\varepsilon-2}$. However, by Lemma~\ref{lem:size.step1} (applied with $c =  \lceil 2/\varepsilon -1 \rceil $), we also have $|\mathcal{A}(x)| \gg_{\varepsilon, \mathcal{A},\mathcal{B}} \log(x)^{2/\varepsilon-1} $. For large $x$,  this gives a contradiction and Case 1 cannot occur. 

\medskip

\underline{Case 2}: $|\mathcal{P}_1(\varepsilon;x)| < |\mathcal{P}_2(\varepsilon;x)|$.  Recall that $p \in \mathcal{P}_2(\varepsilon;x)$ implies $ u_{\mathcal{B}}(p;x) \leq p- p^{1-\varepsilon}$  and hence 
\begin{align*}
v_\mathcal{B}(p^2;x)&:= |\{ (b \mbox{ mod } p^2) \in \mathbb{Z}/p^2\mathbb{Z}: b \in \mathcal{B}(x)\}|\leq pu_{\mathcal{B}}(p;x)+ w_{\mathcal{B} }(p;x)  \\
&\leq (p-1)u_{\mathcal{B} }(p;x)+ p\leq p^2- p^{2-\varepsilon}  + p^{1-\varepsilon} \leq p^2 - \frac{1}{2}p^{2-\varepsilon}.
\end{align*}
We apply  the Johnsen-Selberg sieve (Lemma~\ref{lem:johnsen}) to $\mathcal{B}(x)$. For $p \not\in \mathcal{P}_2(\varepsilon;x)$ we simply take $\Omega(p)=\Omega(p^2)=\emptyset$. For $p  \in \mathcal{P}_2(\varepsilon;x)$ we still take $\Omega(p)=\emptyset$, but   $\Omega(p^2)= (\mathbb{Z}/p^2\mathbb{Z}) \setminus \{ (b \mbox{ mod } p^2) \in \mathbb{Z}/p^2\mathbb{Z}: b \in \mathcal{B}(x)\}$, so that $\omega(p)=0$ and  $\omega(p^2)=p^2- v_\mathcal{B}(p^2;x)$.
We set $ m(x) =\left\lfloor \frac{\log(x^{1/4})}{\log(\log(x)^{2/\varepsilon})} \right\rfloor$  and find
\begin{align*}
    |\mathcal{B}(x)| &\leq 2x \Bigg(\sum_{\substack{q \leq x^{1/4} \\ p \mid q \Rightarrow p \in \mathcal{P}_2(\varepsilon;x)}} \mu(q)^2 \prod_{p \mid q}\left(\frac{p^2-v_\mathcal{B}(p^2;x)}{v_\mathcal{B}(p^2;x)} \right) \Bigg)^{-1} \\
    &\leq 2x \Bigg(\sum_{\substack{q \leq x^{1/4} \\ p \mid q \Rightarrow p \in \mathcal{P}_2(\varepsilon;x)}} \mu(q)^2 \prod_{\substack{p \mid q }}\left(\frac{1}{1 - \frac{1}{2}p^{-\varepsilon}}-1\right) \Bigg)^{-1} \\
    &\leq 2x \Bigg(\sum_{\substack{q \leq x^{1/4} \\ p \mid q \Rightarrow p \in \mathcal{P}_2(\varepsilon;x)}} \mu(q)^2 \prod_{p \mid q} \sum_{k=1}^\infty \left(\frac{1}{2p^{\varepsilon}}\right)^k \Bigg)^{-1} \\
     &\leq 2x \Bigg(\frac{1}{m(x)!}\Bigg(\sum_{\substack{p \leq x^{1/(4m(x))} \\ p \in \mathcal{P}_2(\varepsilon;x)}}  \frac{1}{2p^{\varepsilon}}\Bigg)^{m(x)} \Bigg)^{-1} \\
     &\leq 2x \Bigg(\frac{1}{m(x)}\sum_{\substack{p \in \mathcal{P}_2(\varepsilon;x)}}  \frac{1}{2p^{\varepsilon}}\Bigg)^{-m(x)}  
    \\
    & \ll_\varepsilon 2x \Bigg(  \frac{\log(x)^{2/\varepsilon}}{8m(x)\log(\log(x)^{2/\varepsilon})\log(x)^2 }\Bigg)^{-m(x)}  \\
     &\ll_\varepsilon 2x \Bigg( \exp\left(-\frac{(1-\varepsilon)\log( \log(x)^{2/\varepsilon - 3 -\varepsilon})\log(x^{1/4})}{\log(\log(x)^{2/\varepsilon})}\right)\Bigg) \ll_\varepsilon x^{3/4+\varepsilon}.
\end{align*} 
Thus for large $x$, 
we have $|\mathcal{A}(x)|, |\mathcal{B}(x)| \ll_\varepsilon x^{3/4+\varepsilon}$.  By Lemma~\ref{lem:productsize}, this also implies the lower bound $|\mathcal{A}(x)|, |\mathcal{B}(x)| \gg_{\varepsilon,\mathcal{A},\mathcal{B}} x^{1/4-\varepsilon} \log(x)^{-1/2}$.
\end{proof}

\subsubsection{Application of Theorem~\ref{thm:largesieveimprovement}}

The Johnsen-Selberg sieve gives us no better lower and upper bounds on $|\mathcal{A}(x)|$ and  $|\mathcal{B}(x)|$ than $x^{1/4}$ and $x^{3/4}$. To go beyond these bounds, we need our new sieve estimate, Theorem~\ref{thm:largesieveimprovement}. We first  
  describe how Theorem~\ref{thm:largesieveimprovement} can be combined with Gallagher's sieve to iteratively prove new bounds. We then apply this result to deduce Theorem~\ref{thm:2squares}.

\begin{lemma}\label{lem:size.step3} 
Suppose $L:[1,\infty) \rightarrow \mathbb{R}$ and $U:[1,\infty) \rightarrow \mathbb{R}$   with $x^{1/10} \leq L(x) \leq  x^{1/2} \leq U(x)$ for all $x$. 

Suppose further that $\mathcal{A}\subseteq \mathbb{Z}$ and $\mathcal{B}  \subseteq \mathbb{Z}$ have the property that $\mathcal{A}+\mathcal{B}=\mathcal{S}+\mathcal{S}$ and for   $x \gg_{\mathcal{A},\mathcal{B}} 1$, 
\begin{align*}
L(x) \ll_{\mathcal{A},\mathcal{B}}   |\mathcal{A}(x)|, |\mathcal{B}(x)| \ll_{\mathcal{A},\mathcal{B}} U(x).
\end{align*}
Let $k \in \{1, \dots, 5\}$ with $L(x)^k \leq x^{1/2}$ for all $x$. Then for all sufficiently large $x$,
\begin{align*}
\frac{L(x)^k}{\log(L(x))^{2k+1/2}} \ll_{\mathcal{A},\mathcal{B}}   |\mathcal{A}(x)|, |\mathcal{B}(x)| \ll_{\mathcal{A} ,\mathcal{B}} \frac{x \log(L(x))^{2k}}{  L(x)^k} .
\end{align*}
\end{lemma}
Lemma~\ref{lem:size.step3} allows us to take some lower   bound  $L(x)$ and use it to construct new, hopefully better bounds.

\begin{proof}
Suppose $\mathcal{A}  + \mathcal{B} = \mathcal{S}+\mathcal{S}$ and for  large $x$, 
\begin{align*}
L(x) \ll_{\mathcal{A},\mathcal{B}}   |\mathcal{A}(x)|, |\mathcal{B}(x)| \ll_{\mathcal{A},\mathcal{B}} U(x).
\end{align*}
Fix $x$ and assume that $|\mathcal{A}(x)| \leq |\mathcal{B}(x)|$.  Then define 
\begin{align*}
&\mathcal{P}_1(x) = \left\{p \in  [\tfrac{1}{2}L(x),L(x)]: p \equiv 3 \mbox{ mod } 4 \mbox{ and }  v_{\mathcal{A}}(p;x) \leq \tfrac{p}{100 \log(p)}\right\}, \\ 
&\mathcal{P}_2(x) = \left\{p \in  [\tfrac{1}{2}L(x),L(x)]: p \equiv 3 \mbox{ mod } 4 \mbox{ and } v_{\mathcal{A}}(p;x) > \tfrac{p}{100 \log(p)}\right\},  
\end{align*} 
where we  use the notation given in Definition~\ref{def:residueclasses}.
Recall from Lemma~\ref{lem:basicstructure} that $u_{\mathcal{B}}(p;x)+v_{\mathcal{A}}(p;x) \leq p$ for $p \equiv 3 \mbox{ mod } 4$,  and observe that $\mathcal{P}_2(x) \subseteq \{ p \in [\tfrac{1}{2}L(x),L(x)]: u_{\mathcal{B}}(p;x) \leq p- \tfrac{p}{100 \log(p)}\}$.

\medskip

\underline{Case 1}: $|\mathcal{P}_1(x)| \geq |\mathcal{P}_2(x)|$. We apply Gallagher's larger sieve (Lemma~\ref{lem:largesieve.gallagher}) with $\mathcal{P}=\mathcal{P}_1(x)$ and get
\begin{align*}
     |\mathcal{A}(x)| &\leq \frac{-\log(x) + \sum_{p \in \mathcal{P}_1(x)} \log(p)}{-\log(x) + \sum_{p \in \mathcal{P}_1(x)} \frac{\log(p) }{v_{\mathcal{A}}(p;x)}}\\
     &\leq \frac{-\log(x) + \sum_{p \in \mathcal{P}_1(x)} \log(p)}{-\log(x) + \sum_{p \in \mathcal{P}_1(x)} \frac{100\log(p)^2 }{p}} \\
    &\leq \frac{-\log(x) + |\mathcal{P}_1(x)| \log(L(x))}{-\log(x) + |\mathcal{P}_1(x)|   \frac{100 \log(L(x))^2}{L(x)}}\\ 
    &\ll  \frac{-\log(x) + L(x)}{-\log(x) +   \frac{100}{9} \log(L(x)) }.
\end{align*}
Hence $|\mathcal{A}(x)| \ll   L(x)/\log(L(x))$. But recall our assumption $|\mathcal{A}(x)| \gg_{\mathcal{A},\mathcal{B}} L(x)$. For large $x$ (where the required size of $x$ depends on $\mathcal{A}$ and $\mathcal{B}$), this gives a contradiction.
 
\medskip

\underline{Case 2}: $|\mathcal{P}_1(x)| < |\mathcal{P}_2(x)|$. 
 Recall that $p \in \mathcal{P}_2(x)$ implies $ u_{\mathcal{B}}(p;x) \leq p- \frac{p}{100\log(p)}$  and hence 
\begin{align} \label{inequ:step3iteration}
w_{\mathcal{B}}(p;x) + (p-v_{\mathcal{B}}(p;x)) = p- u_{\mathcal{B}}(p;x)  \geq \tfrac{p}{100\log(p)}.
\end{align}
We will apply  Theorem~\ref{thm:largesieveimprovement} to $\mathcal{B}(x)$. We make the following choices: We let $k \in \{1,\dots,5\}$ with $L(x)^k \leq x^{1/2}$ and take $M=L(x)$.  We   take $\mathcal{P} = \mathcal{P}_2(x) $. Inequality (\ref{inequ:step3iteration}) tells us that for every $p \in \mathcal{P}_2(x)$ we can choose some  $\Omega(p) \subseteq W_{\mathcal{B}}(p;x) \cup ((\mathbb{Z}/p\mathbb{Z}) \setminus V_{\mathcal{B}}(p;x))$ with  $ \omega(p)= |\Omega(p)|\asymp \tfrac{p}{100\log(p)}$. We also set $\Theta=1$.  Then for every $p \in \mathcal{P}$ and $c\in \Omega(p)$, 
$$|\{(b \mbox{ mod } p^2) \in \mathbb{Z}/p^2\mathbb{Z}: b \in \mathcal{B}(x), b \equiv c \mbox{ mod } p\}| \leq 1=\Theta.$$
Thus Theorem~\ref{thm:largesieveimprovement} is applicable  provided that   $c^*(k,\omega,\mathcal{P})$, $d^*(k,\omega,\mathcal{P})$, $\xi(k,\omega,\mathcal{P},\Theta)  >0$. In that case,
\begin{align}\label{inequ:applyingthm3}
&|\mathcal{B}(x)| \ll \xi(k,\omega,\mathcal{P},\Theta) \, x\left(  \mathop{\sum \dots \sum}_{\substack{ p_1, \dots, p_k \in \mathcal{P} \\ p_1> \dots > p_k }}  \prod_{i=1}^k \left( \frac{\omega(p_i)}{p_i-\omega(p_i)}\right) \right)^{-1}.
\end{align}    
 
We now compute simple estimates for $c^*(k,\omega,\mathcal{P})$, $d^*(k,\omega,\mathcal{P})$ and $\xi(k,\omega,\mathcal{P},\Theta)$. First we observe  that $\omega(p) \asymp \tfrac{p}{100\log(p)}$ for $p \in \mathcal{P} \subseteq [\tfrac{1}{2}L(x),L(x)]$ implies 
\begin{align*}
\frac{1 }{  \log(L(x)) }   \ll \left(\frac{\omega(p)}{p-\omega(p)} \right) \ll  \frac{1 }{  \log(L(x)) }   \quad \mbox{ and } \quad
\sum_{p \in \mathcal{P}} \left(\frac{\omega(p)}{p-\omega(p)} \right) 
\gg \frac{L(x)}{\log(L(x))^2}.
\end{align*} 
Note also that $|\mathcal{P}| \gg L(x)/\log(L(x))$. Recalling Definition~\ref{def:sievefactors}, we therefore have for $k \in \{1, \dots, 5\}$
\begin{align*}
&c^*(k,\omega,\mathcal{P})=\left( \sum_{p \in \mathcal{P}} \left(\frac{\omega(p)}{p-\omega(p)} \right) \right)^{-1} \left\{\frac{1}{3} \sum_{p \in \mathcal{P}} \left(\frac{\omega(p)}{p-\omega(p)} \right) - k \max_{p \in \mathcal{P}}  \left(\frac{\omega(p)}{p-\omega(p)} \right) \right\}\asymp 1, 
 \\
&d^*(k,\omega,\mathcal{P})=\left( \sum_{p \in \mathcal{P}} \left(\frac{\omega(p)}{p-\omega(p)} \right) \right)^{-1} \left\{\frac{c^*(k,\omega,\mathcal{P})^k}{   2^{k+5} k} \sum_{p \in \mathcal{P}} \left(\frac{\omega(p)}{p-\omega(p)} \right) - k  \max_{p \in \mathcal{P}}  \left(\frac{p-\omega(p)}{\omega(p)} \right) \right\}  \asymp 1,\\
&\xi(k,\omega,\mathcal{P},\Theta)=64\left(\frac{ 32 \Theta }{c^*(k,\omega,\mathcal{P})  d^*(k,\omega,\mathcal{P}) } \right)^k  \left(\frac{   d^*(k,\omega,\mathcal{P}) }{ 32 \Theta }  \right)^k \asymp 1.
\end{align*}
In short, since $L(x) \geq x^{1/10}$, we find upon careful inspection  that    $c^*(k,\omega,\mathcal{P})$, $d^*(k,\omega,\mathcal{P})$ and $\xi(k,\omega,\mathcal{P},\Theta)$ can be bounded below and above by  positive constants independent of $x$, provided $x$ is sufficiently large. Therefore (\ref{inequ:applyingthm3}) can be simplified to 
\begin{align*} 
&|\mathcal{B}(x)| \ll    x\left(  \mathop{\sum \dots \sum}_{\substack{ p_1, \dots, p_k \in \mathcal{P} \\ p_1> \dots > p_k }}   \left( \frac{1}{\log(L(x))}\right)^k \right)^{-1} \ll \frac{x \log(L(x))^{2k}}{  L(x)^k} .
\end{align*}  

Thus for large $x$, 
we have $|\mathcal{A}(x)|, |\mathcal{B}(x)| \ll_{\mathcal{A} ,\mathcal{B}} \frac{x \log(L(x))^{2k}}{  L(x)^k} $.  By Lemma~\ref{lem:productsize}, this also implies the lower bound $|\mathcal{A}(x)|, |\mathcal{B}(x)| \gg_{ \mathcal{A},\mathcal{B}} \left( \frac{x}{  \log(x)^{1/2}}\right) \left( \frac{  L(x)^k}{x \log(L(x))^{2k}} \right) = \frac{L(x)^k}{\log(L(x))^{2k+1/2}}$, which is what we wanted to show.
\end{proof}

To complete the proof of Theorem~\ref{thm:2squares}, we now just need to  apply Lemma~\ref{lem:size.step3} twice. 

\begin{proof}[Proof of Theorem~\ref{thm:2squares}]
Suppose $\mathcal{A}  + \mathcal{B} = \mathcal{S}+\mathcal{S}$  and $|\mathcal{A} |, |\mathcal{B} | \geq 2$.  By Lemma~\ref{lem:size.step2}, we already know 
\begin{align*}
x^{1/8} \ll_{\mathcal{A},\mathcal{B}}   |\mathcal{A}(x)|, |\mathcal{B}(x)| \ll_{\mathcal{A},\mathcal{B}} x^{7/8}.
\end{align*}
We first apply Lemma~\ref{lem:size.step3} with $L(x) = x^{1/8}$ and $k=3$. For $x \gg_{\mathcal{A},\mathcal{B}} 1$ we get 
\begin{align*}
\frac{x^{3/8}}{\log(x)^{6+1/2}} \ll_{\mathcal{A},\mathcal{B}}   |\mathcal{A}(x)|, |\mathcal{B}(x)| \ll_{\mathcal{A} ,\mathcal{B}} \frac{x \log(x)^{6}}{x^{3/8}} .
\end{align*}
We then apply Lemma~\ref{lem:size.step3} with $L(x) = x^{1/4}$ and $k=2$. For $x \gg_{\mathcal{A},\mathcal{B}} 1$ we get 
\begin{align*}
\frac{x^{1/2}}{\log(x)^{4+1/2}} \ll_{\mathcal{A},\mathcal{B}}   |\mathcal{A}(x)|, |\mathcal{B}(x)| \ll_{\mathcal{A} ,\mathcal{B}} \frac{x \log(x)^{4}}{ x^{1/2}} . 
\end{align*}
This is what we wanted to show.
\end{proof}

\subsection{Decomposition into three sets} A nice corollary of Theorem~\ref{thm:2squares} is that $\mathcal{S}+\mathcal{S}$ cannot be written as a sum of three different sets in a non-trivial way.

\begin{cor}
Suppose $\mathcal{A}_1, \mathcal{A}_2, \mathcal{A}_3 \subseteq \mathbb{Z}$ with $\mathcal{A}_1+\mathcal{A}_2+\mathcal{A}_3=\mathcal{S}+\mathcal{S}$. Then  $|\mathcal{A}_i| =1$ for some $i \in \{1,2,3\}$. 
\end{cor}

\begin{proof}
We begin by recalling a result of Ruzsa~\cite{Ruzsa:2007:CQS}: 
If $A, B, C \subset \mathbb{Z}$ are finite and non-empty, then
\begin{align}\label{inequ:ruzsa}
|A+B+C|^2 \leq |A+B| |A+C| |B+C|.
\end{align}
Let $\mathcal{A}_1, \mathcal{A}_2, \mathcal{A}_3 \subseteq \mathbb{Z}$ with $\mathcal{A}_1+\mathcal{A}_2+\mathcal{A}_3=\mathcal{S}+\mathcal{S}$.  Suppose for a contradiction that $|\mathcal{A}_i| \geq 2$ for all $i$. Applying Theorem~\ref{thm:2squares} with $\mathcal{A} = \mathcal{A}_i$ and $\mathcal{B}=\mathcal{A}_j + \mathcal{A}_k$ for $\{i,j,k\}=\{1,2,3\}$, we then get 
\begin{align*}
&| \mathcal{A}_i (x)|   \ll \sqrt{x} \log(x)^{4}, \\
&|\mathcal{A}_j(x)  +  \mathcal{A}_k(x)|  \leq   |(\mathcal{A}_j + \mathcal{A}_k)(2x)| \ll \sqrt{x} \log(x)^4. 
\end{align*}
On the other hand, writing $K=\min(\mathcal{A}_1 \cup \mathcal{A}_2 \cup \mathcal{A}_3 \cup \{0\})$, equality $\mathcal{A}_1+\mathcal{A}_2+\mathcal{A}_3=\mathcal{S}+\mathcal{S}$  gives
\begin{align*}
|\mathcal{A}_1(x) +  \mathcal{A}_2(x) + \mathcal{A}_3(x)| &\geq |(\mathcal{S}+\mathcal{S})(x+2K)| - |K| \max_{j \neq k}  |(\mathcal{A}_j +\mathcal{A}_k)(x)|- |K|^2 \max_{i}  | \mathcal{A}_i(x)| -|K|^3\\  &\gg  \frac{x}{\sqrt{\log(x)}}.
\end{align*}
(Here all implied constants may depend on $\mathcal{A}_i$ for $i \in \{1,2,3\}$.)

\medskip Now we apply inequality (\ref{inequ:ruzsa}) with $A=\mathcal{A}_1(x)$ and $B=\mathcal{A}_2(x)$ and $C=\mathcal{A}_3(x)$. For large $x$, then
\begin{align*}
 \frac{x^2}{ \log(x) }   &\ll |\mathcal{A}_1(x) +  \mathcal{A}_2(x) + \mathcal{A}_3(x)|^2\\
&\leq |\mathcal{A}_1(x) +  \mathcal{A}_2(x)|  |\mathcal{A}_1(x)  + \mathcal{A}_3(x)| |   \mathcal{A}_2(x) + \mathcal{A}_3(x)| \\
&\ll (\sqrt{x} \log(x)^4)^3 = x^{3/2} \log(x)^{12}.
\end{align*}
But this is a contradiction. Therefore  $|\mathcal{A}_i| =1$ for some $i \in \{1,2,3\}$. 
\end{proof} 

\section{Explicit constructions of sumsets}\label{sec:constructions}

In the previous section we assumed that $\mathcal{A}_1+ \mathcal{A}_2$ is a non-trivial additive decomposition  of the set of sums of  two squares and proved upper and lower bounds on   $|\mathcal{A}_i\cap [1,x]|$. Now we approach the problem from a different direction and   explicitly construct alternative sumset decompositions of $\mathcal{S} + \mathcal{S}$. We focus on sets which are either subsets or supersets of the squares.

\subsection{Constructing many sumset decompositions} 
We first think about subsets of $\mathcal{S}$. Part~(1)  
  of Theorem~\ref{thm:construction2} is the following lemma:
\begin{lemma}\label{lem:subsets}
Let  $\mathcal{S} = \{n^2: n \in \mathbb{N} \cup \{0\}\}$.  Let $p$ be a prime with $p \equiv 1 \bmod 4$.

 Let  $M_1, M_2 \subseteq \mathbb{N}$ with $M_1 \cap M_2 = \emptyset$. Let $ \mathcal{A}_i=\mathcal{S}\setminus\{p^2 m^2: m\in M_i\}$ for $i \in \{1,2\}$.  Then 
\begin{align*} 
&\mathcal{A}_1 + \mathcal{A}_2= \mathcal{S}+\mathcal{S}.
 \end{align*}  
\end{lemma}

In particular, choosing  $p=5$, we see that the set of sums of two squares can be produced by a sumset $\mathcal{A}_1+\mathcal{A}_2$ with $\mathcal{A}_1,\mathcal{A}_2 \subseteq \mathcal{S}$ where a fifth of all squares has been removed from $\mathcal{A}_1$ or a tenth of the squares has been removed from both $\mathcal{A}_1$ and $\mathcal{A}_2$. Further, $\mathcal{S}+\mathcal{S}$ can be written as a sumset in uncountably many different ways. 

\begin{proof}[Proof of Lemma~\ref{lem:subsets}] Let  $M_1, M_2 \subseteq \mathbb{N}$ with $M_1 \cap M_2 = \emptyset$. Let $ \mathcal{A}_i=\mathcal{S}\setminus\{p^2 m^2: m\in M_i\}$ for $i \in \{1,2\}$. 
We want to show that  an arbitrary sum of two squares $x^2+y^2$ is in $\mathcal{A}_1+\mathcal{A}_2$. 

\medskip First note that $M_1 \cap M_2 = \emptyset$ implies $\mathcal{A}_1 \cup \mathcal{A}_2 = \mathcal{S}$. So $x^2, y^2 \in \mathcal{A}_1 \cup \mathcal{A}_2$. Thus if   $x^2 \in \mathcal{A}_1 \cap \mathcal{A}_2$ or $y^2 \in \mathcal{A}_1 \cap \mathcal{A}_2$, then clearly $x^2+y^2 \in \mathcal{A}_1+\mathcal{A}_2$.

\medskip Now we consider the case $x^2, y^2 \not \in  \mathcal{A}_1 \cap \mathcal{A}_2$. Then $p^2 \mid x^2$ and $p^2 \mid y^2$. In particular, $p^2 \mid x^2+y^2$. Let $k$ denote the largest number with $p^k \mid x^2+y^2$, and remember that $p \equiv 1 \bmod 4$.  Now we recall Jacobi's formula~\cite{Jacobi:1829:FTF} for $r_2(n) = \{(a,b) \in \mathbb{Z}^2: a^2+b^2=n\}$, the number of representations of $n$ as a sum of two squares: If $n=2^a p_1^{b_1} \dots p_k^{b_k} q_1^{c_1} \dots q_s^{c_s} $ with $p_i \equiv 1 \bmod 4$ and $q_i \equiv 3 \bmod 4$ and $2 \mid c_i$ for each $i$, then 
\begin{align*}
    r_2(n)= 4 (b_1+1) \dots (b_k+1).
\end{align*}
In particular, since $p \equiv 1 \bmod 4$ and $p^k \mid x^2+y^2$ and $p^{k+1} \nmid x^2+y^2$, we  see that 
  $$r_2(x^2+y^2)=\left(\tfrac{k+1}{k-1} \right)r_2\Big(\Big(\frac{x}{p}\Big)^2+\Big(\frac{y}{p}\Big)^2\Big).$$ Since $k \geq 2$, we have $r_2(x^2+y^2) > r_2((\frac{x}{p})^2+(\frac{y}{p})^2)$.
However, the number of representations of 
$x^2+y^2$ of type $ (p\, m)^2+(p\,n)^2$ is just
$r_2((\frac{x}{p})^2+(\frac{y}{p})^2)$.
Hence there must exist an alternative  representation $a^2+b^2=x^2+y^2$  with both $a$ and $b$ coprime to $p$. (If one summand is coprime to $p$, then both are.) But then $a, b \in  \mathcal{A}_1 \cap \mathcal{A}_2$ and indeed $x^2+y^2 \in  \mathcal{A}_1 + \mathcal{A}_2$.
\end{proof}

 \subsection{No decompositions via supersets}
Finally,  we show that we cannot create new sumset decompositions of $\mathcal{S}+\mathcal{S}$ by   adding some non-squares to $\mathcal{S}$.  This is Part~(2) of Theorem~\ref{thm:construction2}.

\begin{lemma}\label{lem:const2} 
Let $\mathcal{A}_1, \mathcal{A}_2 \subseteq \mathbb{Z}$ with $\mathcal{S} \subseteq \mathcal{A}_1$ and $\mathcal{S} \subseteq \mathcal{A}_2$. Suppose further  that $\mathcal{A}_1 \neq \mathcal{S}$ or $\mathcal{A}_2 \neq \mathcal{S}$. Then 
\begin{align*} 
&\mathcal{A}_1 + \mathcal{A}_2 \neq \mathcal{S}+\mathcal{S}.
 \end{align*}  
\end{lemma} 

\begin{proof}
Suppose $c \in \mathbb{Z}\setminus \mathcal{S}$. We need to show that there exists some 
$x\in \mathbb{N} \cup \{0\}$ with $x^2+c \not\in \mathcal{S}+\mathcal{S}$. \medskip
 
\underline{Case A}: $c= p_1 \cdots p_k$ is a product of $k$ distinct primes of the form 
$p_i \equiv 1 \bmod 4$, where $1\leq i \leq k$. 

\medskip

We first choose a quadratic non-residue  $b \mbox{ mod } p_1 $, i.e.
$(\frac{b}{p_1})=-1$. By the Chinese Remainder Theorem 
there exists some $B \in \mathbb{N}$ with $B \equiv 3 \bmod 4$, $B \equiv b \bmod p_1$ and $B \equiv 1 \bmod p_i$ for $ i=2, \ldots, k$.
 By Dirichlet's theorem and as $\gcd(B,4p_1 \ldots p_k)=1$, 
  there exists a prime $q$ with
\[q\equiv B \bmod 4 p_1 \ldots p_k.\]
Since $q \equiv b \bmod p_1$, then   $(\frac{q}{p_1})=-1$. Since 
$q \equiv 1 \bmod p_i$, we have  $(\frac{q}{p_i})=1$ for 
$2 \leq i \leq k$. Using quadratic reciprocity and that $p_i \equiv 1 \bmod 4$ and $q\equiv 3 \bmod 4$, it follows that $$\left(\frac{-c}{q}\right) =
\left(\frac{-1}{q}\right)\left(\frac{p_1}{q}\right)\cdots  \left(\frac{p_k}{q}\right)=
(-1)\left(\frac{q}{p_1}\right)\cdots  \left(\frac{q}{p_k}\right)=
(-1)(-1) =1.$$

Hence $-c$ is a quadratic residue mod $q$, and there exists $m \bmod q$ with $m^2\equiv -c \bmod q$. Since  $\gcd(c,q)=\gcd( p_1 \cdots p_k,q)=1$, we also have that $2m$ is coprime to $q$. This will be useful in the next step.

\medskip
We consider $x'=m+ql$, where $l \in \mathbb{N}$. Then $(x')^2+c \equiv 0 \bmod q$. Since $2m$ is coprime to $q$, 
  one has $q^2\mid (x')^2+c=m^2+2mql+q^2l^2+c$ if and only if
  $ql\equiv (-c-m^2)(2m)^{-1} \bmod q^2$.
 Note that the right hand side of
  $ql\equiv(-c-m^2)(2m)^{-1} \bmod q^2$ does not depend on $l$.
  Hence there exists an $l^*\in \mathbb{N}$ with
  $ql^*\not\equiv (-c-m^2)(2m)^{-1} \bmod q^2$. With $x=m+ql^*$ we have
  \[ x^2+c\equiv m^2+c\equiv 0  \bmod q, \quad \mbox{ but } \quad x^2+c \not\equiv 0  \bmod q^2.\]
  We conclude that $x^2+c$ cannot be a sum of two squares,
  as $q\equiv 3 \bmod 4$.

\medskip
   
   \underline{Case B}:  There exists a prime $q\equiv 3 \bmod 4$
   with $q|c$, but $q^2 \nmid c$.  Set $x=q$. As $q|c$, also $q|q^2+c$.
   But as $q^2\nmid c$, also $q^2 \nmid q^2+c$. So  $q^2+c$ is not a sum of two squares.
   
    \medskip
   
   \underline{Case C}:  $2|c$, but $4 \nmid c$. Set $x=1$. Then $x^2+c\equiv 3 \bmod 4$,
   and $x^2+c$ is not a sum of two squares.

\medskip

\underline{The general case}:
  Consider an arbitrary $c\not\in \mathcal{S}$. If $c<0$, then $0^2+c<0$. So assume $c>0$. 
   Write $c=c' k^2$, where $k \in  \mathbb{N}$ and $c'>1$ is squarefree. 
   If $c'$ only consists of prime factors $p \equiv 1 \bmod 4$,
   we may apply Case A to $c'$.  If $c'$ has a prime factor $q\equiv 3 \bmod 4$, then $c'$ squarefree implies   $q^2 \nmid c'$, and so we can apply Case B to $c'$. Finally, if neither A nor B apply, we have  $2|c'$ and $4\nmid c'$, and so   can apply Case C to $c'$. In all  three cases there exists   $x' \in\mathbb{N}$ with $(x')^2+c' $  not  a sum of two squares. Some prime factor $p \equiv 3$ mod $4$  appears in $(x')^2+c'$ with odd exponent. 
   Now put $x=x'k$. Then 
   \[x^2+c=(x')^2k^2+c'k^2=k^2((x')^2+c')
     \]
     and  $p\equiv 3 \bmod 4$  also appears in $x^2+c$ with odd exponent.
    Hence $x^2+c$ is not a sum of two squares.
\end{proof}

\section*{ai declaration}

The work presented here was done   without the use of artificial intelligence tools, all proofs are solely  the authors' and a complete version of this paper was written without AI assistance. At the end, we asked Claude and ChatGPT for feedback on this draft, and we accordingly fixed a small number of typos and minor   errors.

\bibliography{bibstad}
\bibliographystyle{acm} 

\end{document}